\documentclass{imsart}

\RequirePackage{amsthm,amsmath,amsfonts,amssymb}
\RequirePackage[authoryear]{natbib}
\RequirePackage[colorlinks,citecolor=blue,urlcolor=blue]{hyperref}
\RequirePackage{graphicx}
\hypersetup{plainpages=True,pdfstartview=FitV,pageanchor=false,colorlinks=true,linkcolor=blue!50,citecolor=blue!50}

\usepackage{times}
\usepackage{dsfont}
\usepackage{ctable}
\newcommand{\thline}{\specialrule{.005em}{.02em}{.02em}}

\usepackage{makecell}
\usepackage{caption}

\usepackage{chngcntr}
\counterwithin{figure}{section}
\counterwithin{table}{section}
\usepackage{etoolbox}
\robustify\widetilde

\makeatletter
\let\@fnsymbol\@arabic
\makeatother

\startlocaldefs

\theoremstyle{plain}
\newtheorem{thm}{Theorem}[section]
\newtheorem{co}{Corollary}[section]
\newtheorem{df}{Definition}[section]

\newtheorem{lm}{Lemma}[section]

\newtheorem{rem}{Remark}[section]

\newcommand{\Ind}{\mathds{1}}
\def\le{\leqslant}
\def\ge{\geqslant}
\def\ve{\varepsilon}

\def\lcl#1{\bigl\lceil #1\bigr\rceil}
\def\pa#1{({#1})}
\def\lpa#1{\bigl({#1}\bigr)}
\def\Lpa#1{\Bigl({#1}\Bigr)}
\def\llpa#1{\biggl({#1}\biggr)}

\def\smath#1{\text{\scalebox{.9}{$#1$}}}

\def\sfrac#1#2{\smath{\frac{#1}{#2}}}

\newcommand\Dx{\frac{\mathrm{d}}{\mathrm{d}x}}
\newcommand\Var{\mathrm{Var}}
\newcommand\Cov{\mathrm{Cov}}
\newcommand\dd{\,\mathrm{d}}
\renewcommand\b{b} 
\newcommand{\mybinom}[3][0.8]{\scalebox{#1}{$\dbinom{#2}{#3}$}}

\newcommand*\overbarr[1]{%
  \vbox{%
    \hrule height 0.6pt
    \kern0.25ex
    \hbox{%
      \kern-0.2em
      \ifmmode#1\else\ensuremath{#1}\fi
      \kern-0.1em
    }
  }
}
\def\overbar#1{\,\overbarr{#1}}

\endlocaldefs

\begin{document}

\begin{frontmatter}
\title{Integrated empirical measures and generalizations of classical
goodness-of-fit statistics\thanks{An early preliminary version of
this paper appeared in \cite{Kuriki13}. Part of the work of both
authors was carried out while visiting each other's institute
(Institute of Statistical Mathematics, Tokyo and Institute of
Statistical Sciences, Taipei); we thank both Institutes for their
support.}}
\runtitle{\,Generalized goodness-of-fit tests}

\begin{aug}
\author[A]{\fnms{Hsien-Kuei}~\snm{Hwang}\ead[label=e1]{hkhwang@stat.sinica.edu.tw}\orcid{0000-0002-9410-6476}},
\author[B]{\fnms{Satoshi}~\snm{Kuriki}\ead[label=e2]{kuriki@ism.ac.jp}\orcid{0009-0000-1509-9013}
}

\address[A]{Institute of Statistical Science,
Academia Sinica\printead[presep={,\ }]{e1}}

\address[B]{The Institute of Statistical Mathematics\printead[presep={,\ }]{e2}}
\runauthor{H.-K.~Hwang and S.~Kuriki\,}
\end{aug}

\begin{abstract}
	
Based on $m$-fold integrated empirical measures, we study three new
classes of goodness-of-fits tests, generalizing Anderson-Darling,
Cram\'er-von Mises, and Watson statistics, respectively, and examine
the corresponding limiting stochastic processes. The limiting null
distributions of the statistics all lead to explicitly solvable cases
with closed-form expressions for the corresponding Karhunen-Lo\`{e}ve
expansions and covariance kernels. In particular, the eigenvalues are
shown to be $\frac1{k(k+1)\cdots (k+2m-1)}$ for the generalized
Anderson-Darling, $\frac1{(\pi k)^{2m}}$ for the generalized
Cram\'er-von Mises, and $\frac1{2\pi\lceil k/2\rceil^{2m}}$ for the
generalized Watson statistics, respectively. The infinite products
of the resulting moment generating functions are further simplified
to finite ones so as to facilitate efficient numerical calculations.
These statistics are capable of detecting different features of the
distributions and thus provide a useful toolbox for goodness-of-fit
testing.

\end{abstract}



\end{frontmatter}
\tableofcontents

\section{Introduction}
\label{sec:intro}

\subsection{Scope and summary}

For more than seven decades, empirical processes have been one of
the main tools used in understanding the behaviors of many
goodness-of-fit tests defined via empirical distributions. We 
explore in this paper the usage of the less addressed
\emph{iterated empirical measures} in generalizing the classical
goodness-of-fit tests, focusing on the rarely available
\emph{explicitly solvable cases}.

\paragraph{Three classical goodness-of-fit statistics.}
Assume that $\{X_1,\ldots,X_N\}$ is an independent and identically
distributed (i.i.d.) sequence of random variables with a common
cumulative distribution function $F$. In the one sample setting, the
goodness-of-fit test is formulated as a test for the hypotheses
\[
    H_0: F=F_0\ \ \mbox{against}\ \ H_1: F\ne F_0.
\]
When $F_0$ is a continuous distribution, we may assume, without loss
of generality, that $F_0(x)=x$, the uniform distribution on $[0,1]$.
Since the empirical distribution function $F_{N}(x) := \frac1N
\sum_{1\le i\le N} \Ind_{\{X_i\le x\}}$ provides a natural estimator of
$F(x)$ with $\Ind_{\{\cdot\}}$ the indicator function, a good means of
measuring the discrepancy between $F$ and $F_0$ is to compute the
integrated squared difference between $F_N$ and $F_0$ with a suitable
weight function. Among the large number of such statistics, three
simple, effective and widely used ones are: the Cram\'er-von Mises
statistic
\begin{align*}
	\omega_{N} 
	= N \int_0^1 (F_{N}(x)-x)^2 \dd x 
\end{align*}
\citep{Cramer28,vonMises31,Smirnov37}, the Anderson-Darling 
statistic 
\begin{align*}
    A_{N} 
    = N \int_0^1 \frac{(F_{N}(x)-x)^2 }{x(1-x)}\dd x
\end{align*}
\citep{Anderson52}, and Watson's ($U^2$-) statistic
\begin{align*}
    U_{N} 
    = N\int_0^1 \biggl( F_{N}(x)-x 
    -\int_0^1 (F_{N}(t)-t) \dd t \biggr)^2 \dd x
\end{align*}
\citep{Watson61}. For historical and technical reasons (some
explained below), these are generally regarded as the \emph{de facto}
standards in the area of quadratic statistics based on empirical
distribution functions. Briefly, the Anderson-Darling statistic
has more power than the Cram\'er-von Mises statistic in the tail area 
of the distribution due specially to the choice of the variance-based 
weight function $\Var(F_{N}(x))=\frac{x(1-x)}N$ under the null 
hypothesis $H_0$. On the other hand, Watson's statistic is designed 
for testing the uniformity of periodic data with period 1 so that 
$X_i+1\equiv X_i$.

In addition to integral-type goodness-of-fit statistics, the
literature also abounds with other types such as the maximum-type
statistics, one representative being the Kolmogorov-Smirnov statistic,
defined as the maximum of the integrand of the Cram\'er-von Mises 
statistic:
\[
    D_N = \sqrt{N} \sup_{x\in [0,1]} | F_N(x)-x |.
\]
Similarly, the maximum of the integrand of the Anderson-Darling
statistic can be defined. In particular, the maximizer or the point
maximizing the statistic of these maximum-type statistics indicates
more precisely how discrepant the empirical distribution to the null
distribution is. Such test statistics have a similar structure to
estimating the change-point in regression analysis on the unit
interval, where the maximizer is an estimate of the change-point. We
will discuss briefly this viewpoint again later.

\paragraph{Explicit Karhunen-Lo\`{e}ve expansions.}
An important property possessed by the three classical statistics is
their \emph{explicitly solvable eigenstructures} in the corresponding
limiting Gaussian processes of the signed square-root of the
integrand. For example, the stochastic process arising from the
Anderson-Darling statistic, defined as the limit of the map
\[
    x\mapsto \sqrt{N}\frac{F_{N}(x)-x}{\sqrt{x(1-x)}},
\]
under the null hypothesis $H_0$, has the limiting Gaussian process 
($B(x)$ being a Brownian bridge)
\[
    x\mapsto A(x):=\frac{B(x)}{\sqrt{x(1-x)}}, 
\]
which has an explicit eigenfunction expansion referred to as the 
Karhunen-Lo\`{e}ve (KL) expansion
\[
    A(x) 
    = \sum_{k\ge1}\frac{1}{\sqrt{\lambda_k}}\,f_k(x)\xi_k;
\]
here $\lambda_k^{-1}$ are the eigenvalues with $\lambda_k:=k(k+1)$,
$f_k$ are eigenfunctions expressible in terms of the Legendre
polynomials, and $\xi_k$ are i.i.d.\ random variables with the
standard normal distribution $N(0,1)$; see Theorem \ref{thm:KL} for
a more precise formulation.

The corresponding eigenvalues for the other two classical statistics
are given by $\lambda_k^{-1}$ with $\lambda_k=(\pi k)^2$
(Cram\'er-von Mises) and $\lambda_k = 2\pi\lcl{\frac k2}^2$ (Watson),
respectively, and the eigenfunctions are both expressible in terms of
trigonometric functions. Such explicit expansions proved to be
advantageous for most qualitative and numerical purposes, notably in
computing the density or the tail probabilities of the limiting
distribution (of the statistic); see the discussions after
Theorem~\ref{thm:gad-cos} and \eqref{SSF}. Also the feature of a few
leading eigenfunctions $f_1, f_2, \ldots$ can be used to determine
the statistical power of the goodness-of-fit under local alternatives
(\cite{Durbin72}; see also Sec.~\ref{subsec:finite}). In the words of
\cite{Deheuvels08}:
\begin{quote}
    \textsl{Unfortunately, for most Gaussian processes of interest 
    with respect to statistics, the values of the $\lambda_k$'s are 
    unknown, even though their existence remains guaranteed \dots 
    The practical application of this theory to statistics is 
    therefore limited to a small number of particular cases.
    \dots}
\end{quote}
In addition to the three statistics, some other fortunate examples of
tests whose KL expansions are explicitly identified are collected in
Sec.~\ref{subsec:survey}.

On the other hand, the ``small ball asymptotics'' in probability
represents another context that is closely connected to the study of
goodness-of-fit statistics of integral-type. As in the context here,
explicit eigenstructures prove useful and informative; see e.g.,
\cite{Chen03}, \cite{Gao03a}, \cite{Nikitin13}, and the references
therein.

In general, the goodness-of-fit statistics of integral-type are
constructed as positive-definite quadratic forms of the empirical
measure $\dd F_{N}(x)$. Thus weighted infinite sums of chi-square
random variables of the form
\begin{equation}\label{weightedsum}
    \sum_{k\ge1} \frac{\xi_k^2}{\lambda_k}, 
    \quad\xi_k\sim N(0,1)\; \mbox{i.i.d.}
\end{equation}
appear naturally as the limit laws. However, to determine the
eigenstructure, we are led to the boundary-value problem involving an
integral equation whose covariance kernel is often not exactly
solvable. For example, \cite{Lockhart98} mentioned that the limiting
distribution of \cite{Shapiro65}'s statistic was of the form
\eqref{weightedsum}, but its weights remained unknown in general.

On the other hand, it is well-known that the decay order of the
eigenvalues is determined by the smoothness of the covariance kernel,
so that even no closed-form expressions are available, one can often
deduce the required properties from smoothness of the covariance
kernel; see \cite{Chang99,Cochran88} for more information.

\paragraph{Explicit formulae for the moment generating functions of 
the limiting null distributions.} From \eqref{weightedsum}, one has
the following explicit form for the moment generating function (MGF)
\begin{equation}\label{infinite-product}
    \mathbb{E}\exp\biggl(s 
    \sum_{k\ge 1}\frac{\xi_k^2}{\lambda_k}\biggr)
    = \prod_{k\ge1}\frac1{\sqrt{1-\frac{2s}{\lambda_k}}},
\end{equation}
and for the three classical statistics, the corresponding MGFs can be 
further simplified as follows.
\begin{center}
\begin{tabular}{ccccc}
Anderson-Darling &&
Cram\'er-von Mises &&
Watson \\ \hline
\\ [-1em]
\raisebox{0pt}{$\displaystyle
    \sqrt{\frac{-2\pi s}{\cos\frac{\pi}{2}\sqrt{1+8s}}}$} &&
\raisebox{0pt}{$\displaystyle
    \sqrt{\frac{\sqrt{2s}}{\sin\sqrt{2s}}}$} &&
\raisebox{0pt}{$\displaystyle 
    \frac{\sqrt{\frac{s}{2}}}{\sin\sqrt{\frac{s}{2}}}$} 
    \\ 
\end{tabular}    
\end{center}
These representations in lieu of the infinite products 
\eqref{infinite-product} are more useful for computational and 
analytic purposes; see Sec.~\ref{subsec:mgf-ad}.

\paragraph{Contributions of this paper.} In this paper, we propose
three classes of goodness-of-fit statistics based on the $m$-fold
integral of the empirical measure
\begin{equation*}\label{FNm}
    \mathop{\int\cdots\int}_{0<x_{1}<\cdots<x_{m}<x} 
    \dd F_{N}(x_{1}) \dd x_{1}\cdots \dd x_{m}
    = \int_{0}^{x} \frac{(x-t)^{m-1}}{(m-1)!}\dd F_{N}(t)
    \qquad(m\ge1),
\end{equation*}
where $F_{N}(x)=\int_{0<x_1<x} \dd F_N(x_1)$, as generalizations of
Anderson-Darling, Cram\'er-von Mises, and Watson statistics, so
that they inherit the advantages (such as explicit KL expansions and
explicit MGFs for the limit laws) of the original three statistics.
We provide a systematic approach based on the use of ``template 
functions'' to generalize the classical statistics, and examine their 
limiting properties. Some properties of these generalized statistics 
are summarized in Table~\ref{tab:three-classes}.

\begingroup
\begin{footnotesize}
\renewcommand{\arraystretch}{2.7}
\begin{table}[h]
\begin{center}
\caption{\emph{Generalizations of the three classical statistics 
based on $m$-fold integrated empirical measure. Here $\eta_j$ denote 
the zeros of a polynomial equation; see \eqref{poly-eq}.}}
\label{tab:three-classes}
\bigskip
\addtolength{\tabcolsep}{-1.5pt}    
\begin{tabular}{cccc}
\makecell[c]{Proposed\\test} 
&\makecell[c]{Generalized\\Anderson-\\Darling} 
&\makecell[c]{Generalized\\Cram\'er-\\von Mises}  
&\makecell{Generalized\\Watson} \\ \hline
\makecell[c]{Eigenvalue\\ $\lambda_k^{-1},k\ge 1$} 
& $\frac{1}{k(k+1)\cdots(k+2m-1)}$ 
& $\frac{1}{(\pi k)^{2m}}$ 
& $\frac{1}{(2\pi\lceil k/2\rceil)^{2m}}$ \\
\makecell[c]{Eigenfunction} 
& \makecell[c]{associated\\Legendre} 
& \makecell[c]{trigonometric\\function} 
& \makecell[c]{trigonometric\\function} \\
\makecell[c]{Covariance\\kernel} 
& \makecell[c]{incomplete\\beta} 
& \makecell[c]{Bernoulli\\polynomial} 
& \makecell[c]{Bernoulli\\polynomial} \\
\makecell[c]{MGF of\\the limit law}
& \makecell[c]{$\sqrt{\prod\limits_{j=0}^{m-1}
    \frac{-2\pi s}{(2j)!(2j+1)!\cos\lpa{\pi\eta_j^{\frac12}}}}$} 
& $\sqrt{\frac{e^{\frac{m-1}{2}\pi i} \sqrt{2s}}
    {\prod\limits_{j=0}^{m-1}\sin\bigl((2s)^{\frac{1}{2m}}
    e^{\frac{j\pi i}{m}}\bigr)}}$
& $\frac{2^{-m} e^{\frac{m-1}{2}\pi i} \sqrt{2s}}
    {\prod\limits_{j=0}^{m-1}\sin\bigl(\frac12 
    (2s)^{\frac{1}{2m}}e^{\frac{j\pi i}{m}}\bigr)}$ \\
\end{tabular}
\end{center}
\end{table}
\end{footnotesize}
\endgroup

For example, the moment generating functions of the limiting
distributions of the generalized Anderson-Darling statistic for $m=2$
has the following closed-form expression
\begin{small}
\begin{equation*}
\prod_{k\ge1}\frac1{\sqrt{1-\frac{2s}{k(k+1)(k+2)(k+3)}}} =
    \frac{\pi s}{\sqrt{3\cos\left(\frac\pi2
    \sqrt{5-4\sqrt{1+2s}}\right)
    \cos\left(\frac\pi2
    \sqrt{5+4\sqrt{1+2s}}\right)}}
\end{equation*}
\end{small}
which is (Eq.~(\ref{adm2}) in Sec.~\ref{subsec:mgf-ad}).
The explicit form for $m=3$ is also given in (\ref{adm3}).

\paragraph{Some integrated empirical measures or Gaussian processes 
in the literature.}\hfill\\
Goodness-of-fit tests based on the one-time integrated empirical
measures was proposed by \cite{Henze00} in a way different
from this paper. They obtained the asymptotic distribution in terms
of one-time integrated Brownian bridge and the corresponding
eigenvalues $\lambda_k$ are shown to be connected to the solutions of
the equation $\tan\lambda^{\frac14} +\tanh\lambda^{\frac14}=0$, the
corresponding eigenfunctions being expressible in terms of
trigonometric and hyperbolic functions. Similarly, Watson-type tests
and two-sample tests based on the integrated empirical measures were
studied in \cite{Henze02} and in \cite{Henze03},
respectively. Recently \cite{Durio16} pointed out that some 
integrated goodness-of fit tests have good statistical power for
skew alternative.

In connection with ball probabilities of stochastic processes, 
\cite{Gao03a} showed that the eigenvalues $\lambda_k$ associated 
with the squared integral of the one-time integrated Brownian
motion are expressible in terms of the real zeros of the equation 
$\cosh\lambda^{\frac14}+\sec\lambda^{\frac14} =0$, and provided the 
moment generating function explicitly. In the case of multi-fold 
integrated Brownian motion, \cite{Tanaka08} illustrated the use of 
Fredholm determinant approach to deriving the KL expansion; see also 
Chap.\,5 of \cite{Tanaka17}.

\paragraph{Integrated processes and change-point analysis in regression}
The detection of change-points is of fundamental interest in
statistical data analyses. In a typical setting, the change-point
analysis aims at detecting the change of the mean in time series
data. For that purpose, the difference in the sample mean before and
after a candidate of change-point is used as a scan statistic.
However, the changes to be detected are not necessarily limited to
the mean parameter. \cite{MacNeill78a,MacNeill78} discussed the
change of parameters when a polynomial regression is used to fit the
time series data. He proposed test statistics based on the
accumulated (integrated) residuals, and derived their asymptotic
distributions in terms of functionals of a Brownian motion. The same
idea was also examined in \cite{Hirotsu86,Hirotsu17}. In the analysis
of ordered categorical data, he proposed to classify rows and/or
columns by detecting change-points among the rows/columns. He pointed
out that sometimes the change is observed as the change of
convexity/concavity of response curves, and the accumulated statistic
is useful for such detection purposes. The integrated empirical
processes proposed here are regarded as generalizations of their
statistics.

\paragraph{Organization of the paper.}
In Sec.~\ref{subsec:survey}, we give a brief historical account
relating to the statistics we propose. In particular, some
goodness-of-fit statistics of integral-type with explicit
eigenstructure are gathered there. In Sec.~\ref{sec:ad}, a class of
generalized Anderson-Darling statistics is proposed. From the
limiting null distribution of the integrand of the statistics, we
define and study the (generalized) Anderson-Darling process, with its
KL expansion and the covariance kernel given in explicit forms.
Sections~\ref{sec:w} and \ref{sec:cvm} then deal with the generalized
Watson and Cram\'er-von Mises statistics and the associated
processes, and discuss the corresponding KL expansions, the
covariance kernels, and eigenfunction expansions. We also simplify
the infinite-product representations in all cases of the moment
generating functions of the limiting distributions, and compare the
statistical power of the statistics by Monte Carlo simulations.

\subsection{A brief review of related literature}
\label{subsec:survey}

As the literature on integral types of goodness-of-fit tests is 
vast, we content ourselves with a brief summary of integral tests of 
the form 
\begin{align}\label{weighted-form}
    N \int_0^1 (F_N(x)-x)^2 \varphi(x) \dd x,
\end{align}
(with different $\varphi$) whose eigen-structures are explicitly 
known. For more information on the subject, see for example the books 
\cite{Durbin73}, \cite{Shorack86}, \cite{Nikitin95}, and 
\cite{Martynov15}, \cite{DAgostino17}, and
the survey papers \cite{Gonzalez13} and \cite{Broniatowski22}.

\paragraph{Test statistics with explicit KL expansions.}
We list some examples connected to goodness-of-fit and related tests
whose eigenstructures are explicitly known; this list is not intended
to be exhaustive or comprehensive, but for easier comparison with our
results.

As already stated, \cite{MacNeill78,MacNeill78a} proposed test
statistics in terms of an integrated empirical measure to detect the
change points in polynomial regressions. The eigenstructure of the
limiting Gaussian process is explicitly given; in particular, the
eigenvalues are expressible in terms of the zeros of spherical Bessel
functions, and the eigenfunctions in terms of associated Legendre
polynomials.

\cite{Deheuvels81} proposed a statistic for testing independence in
$m$-variate observations. An explicit KL expansion for the process
defined as the limit of the integral test is derived. The eigenvalue
has multiple indices $\lambda_{k_1,\ldots,k_m}= \frac{2^{m/2}}
{k_1\cdots k_m\pi^m}$ ($k_j\ge 1$) with the corresponding
eigenfunctions given by $\prod_{1\le j\le m}\sin(k_j\pi u_j)$.
\cite{Deheuvels03} defined the weighted Wiener process and Brownian
bridge with weight function $\varphi(t)=t^\beta$ ($\beta>-1$), and
obtained their Karhunen-Lo\`{e}ve expansions in terms of the zeros of
Bessel functions. A multivariate extension with weight function
$\varphi(t)=\prod_i t_i^{\beta_i}$ was given in \cite{Deheuvels05}.
In testing independence for a bivariate survival function,
\cite{Deheuvels08} proved that under the null hypothesis, the
proposed test statistic has an explicit Karhunen-Lo\`{e}ve expansion
with eigenvalues $\frac1{k(k+1)(k+2)(k+3)}$ ($k\ge1$) and
eigenfunctions in terms of the modified Jacobi polynomials. Note that
this sequence of eigenvalues is the same as our generalized
Anderson-Darling statistic when $m=2$; see Table
\ref{tab:three-classes}.

\cite{Pycke03} studied an $m$-variate extension of the
Anderson-Darling statistic. The eigenvalues now have the form
$\lambda_k=\frac{m}{k(k+1)}$ ($k\ge 1$) and the corresponding
eigenfunctions are expressible using the associated Legendre function.

\cite{Baringhaus08} proposed a class of goodness-of-fit
statistics for exponentiality. The KL expansion appearing in the
limit law has eigenvalues in terms of the zeros of the Bessel
functions with Bessel eigenfunctions.

As mentioned in Introduction, eigenvalue problem for deriving KL
expansion is closely related to the boundary-value problem. For
example, \cite{Chang01} discussed a boundary-value problem with
Bernstein and Euler polynomials as kernels, where the eigenvalues and
eigenfunctions are given explicitly.  Their results have a very
close connection to our generalized Watson statistics; see
Sec.~\ref{sec:w}. These results are summarized in Table
\ref{tab:explicit-KL}.

\begingroup
\renewcommand{\arraystretch}{2}
\begin{table}[h]
\begin{center}
\caption{Tests with an explicit KL expansion}
\label{tab:explicit-KL}
\bigskip
\begin{small}
\begin{tabular}{ccc}
Test & Eigenvalue & Eigenfunction 
\\ \hline
\makecell{$m$-variate test\\for independence\\ \citep{Deheuvels81}} 
& $\frac{2^{m/2}}{k_1\cdots k_m\pi^m}$ 
& $\prod_{1\le j\le m}\sin(k_j\pi u_j)$ 
\\ \thline
\makecell{Weighted Brownian\\motion/bridge\\  \citep{Deheuvels03}}
& \makecell{zeros of Bessel\\(1st kind)} 
& \makecell{Bessel function}
\\ \thline
\makecell{Independence of\\bivariate survival function\\ \citep{Deheuvels08}} 
& $\frac{1}{k(k+1)(k+2)(k+3)}$ 
& \makecell{modified \\Jacobi polynomial} 
\\ \thline
\makecell{$m$-variate \\Anderson-Darling\\ \citep{Pycke03}} 
& $\frac{m}{k(k+1)}$ 
& \makecell{Legendre} 
\\ \thline
\makecell{Decomposition of\\Anderson-Darling\\ \citep{Rodriguez99}} 
& \makecell{zeros of Bessel\\(1st kind)} 
& \makecell{Bessel function} 
\\ \thline
\makecell{Goodness-of-fit for\\exponentiality\\ \citep{Baringhaus08}} 
& \makecell{zeros of Bessel\\(1st kind)} 
& \makecell{Bessel function}
\\ \thline
\makecell{Bernoulli/Euler kernel \\boundary-value problem\\ \citep{Chang01}} 
& & \makecell{Bernoulli/Euler\\polynomial}
\\ \thline
\makecell{Weighted, time-change\\ Brownian Bridge \\ \citep{Pycke21a,Pycke23}} 
& & \makecell{Jacobi, Laguerre, Hermite,\\
Krawtchouk polynomials} 
\\
\end{tabular}
\end{small}
\end{center}
\end{table}
\endgroup

\paragraph{Weighted Cram\'{e}r-von Mises statistics.}
A natural extension of Anderson-Darling and Cram\'er-von Mises
statistics is the quadratic statistics of the form
\eqref{weighted-form}, often referred to as weighted Cram\'{e}r-von
Mises statistics and extensively studied in the literature. Table
\ref{tab:weight} summarizes some variants studied in the literature.
The eigenstructures of many of these statistics are not given
explicitly.

\begingroup
\renewcommand{\arraystretch}{1.5}
\begin{table}[h]
\begin{center}
\caption{Some weighted Cram\'{e}r-von Mises statistics}
\label{tab:weight}
\bigskip
\begin{small}
\begin{tabular}{cl}
$\varphi(t)$ & \qquad\qquad References \\ \hline
$1$ & \cite{Cramer28}, \cite{vonMises31} \\ \thline
$\frac1{t(1-t)}$
& \cite{Anderson52} \\ \thline
$\Bigl(\log\frac{1}{t(1-t)}\Bigr)^2$ &
\makecell[l]{\cite{Groeneboom81} \\
\cite[p.\,227]{Shorack86}} \\ \thline
$\frac{1}{t}$, $\frac{1}{1-t}$ &
\makecell[l]{\cite{Sinclair90}, \cite{Scott99} \\
\cite{Rodriguez99}} \\ \thline
$\frac1{1-t^2}$, $\frac1{t(2-t)}$ &
\makecell[l]{\cite{Rodriguez95} \\
\cite{Viollaz96} \\
\cite{Shi22}} \\ \thline
$t^\beta$ &
\makecell[l]{\cite{Deheuvels03} \\ \cite{Deheuvels05}} \\ \thline
$\frac1{\alpha^2 t^{2-\frac1\alpha}
\bigl(1-t^{\frac{1}\alpha}\bigr)}$ &
\cite{Mansuy05} \\ \thline
$\frac1{t^2}$, $\frac1{(1-t)^2}$ &
\makecell[l]{\cite{Luceno06} \\ \cite{Chernobai15}} \\ \thline
$\frac1{(1-t)^\alpha}$ &
\cite{Feuerverger16} \\ \thline
$\begin{array}{c}
 1,\ t^{2\beta},\ t^2,\ t(1-t) \\
 \bigl(t-\tfrac{1}{2}\bigr)^2,\ \frac{1}{t(1-t)}
 \end{array}$ &
\cite{Medovikov16} \\ \thline
$\begin{array}{c}
\mbox{\footnotesize
$I^{-1}_{\nu_1,\nu_2}(t)^{1-2\nu_1}I^{-1}_{\nu_2,\nu_1}(1-t)^{1-2\nu_2}$} \\
\mbox{\footnotesize ($I(\cdot)$: incomplete beta function)}
\end{array}$ & \cite{Pycke21} \\ \thline
$\frac1{t^{3/2}}$, $\frac1{(1-t)^{3/2}}$ &
\cite{Ma22} \\ \thline
$\frac{1}{t^2(1-t)^2}$ &
\cite{Liu23} \\
\end{tabular}
\end{small}
\end{center}
\end{table}
\endgroup

\section{The generalized Anderson-Darling statistics}
\label{sec:ad}

In this section, we propose our new generalized Anderson-Darling 
statistics and examine their small-sample and large-sample 
properties in detail.

\subsection{Legendre polynomials and the template functions}
\label{subsec:proposal-ad}

We construct first the template functions needed using Legendre 
polynomials.

\paragraph{Construction of the template functions.}
The original Anderson-Darling statistic \eqref{AN} can be written as 
\begin{equation}\label{AN}
    A_{N} 
    = \int_0^1 \frac{B_{N}(x)^2}{x(1-x)}\dd x, 
    \quad \mbox{where }\ 
    B_{N}(x)
    =\sqrt{N} (F_{N}(x)-x).
\end{equation}
Here we express $B_{N}(x)$ in terms of a \emph{template function} 
$\tau_1(t;x)$:
\begin{equation}\label{BN}
    B_{N}(x) 
    = \sqrt{N} \int_0^1 \tau_1(t;x) \dd F_{N}(t), 
    \quad\mbox{where }\  
    \tau_1 (t;x) 
    := \Ind_{\{t\le x\}}-x,
\end{equation}
with $\Ind_{\{t\le x\}}=1$ if $t\le x$, and $0$ otherwise.

The template function $t\mapsto \tau_1(t;x)$ is a piecewise constant
function for $t\in[0,1]$ with a breakpoint at $x$, and orthogonal to
any constant function, namely, $\int_0^1 \tau_1(t;x)\dd t=0$.

To extend this template function, we define $\tau_2(t;x)$ for
$t\in[0,1]$ to be a function such that it is a continuous and
piecewise linear function with a change point at $x$, and orthogonal
to any linear function, namely, $\int_0^1 \tau_2(t;x) (a t + b) \dd
t =0$ for all $a,b\in\mathbb{R}$.

In general, we define the $m$-th template function $\tau_m(t;x)$ for
$t\in[0,1]$ to be a $C^{m-2}$ and piecewise polynomial function of
degree $m-1$ in $t$ with a turning point at $x$, and orthogonal to any
polynomial function of degree $\le m-1$, namely,
\begin{align}\label{tau_m-ortho}
    \int_0^1 \tau_m(t;x)t^j\dd t =0
    \qquad(j=0,1,\dots,m-1).
\end{align}
\begin{df} [Legendre polynomials $P_k(x)$] The (shifted and 
normalized) Legendre polynomials are defined to be the polynomials 
of degree $k$ satisfying the orthogonality relations
\begin{align}\label{Pk-ortho}
    \int_0^1 P_k(t)P_l(t) \dd t = \Ind_{\{k=l\}}
    \qquad(k,l\ge0).
\end{align}
\end{df}
Note that the $P_k(x)$ are nothing but the polynomials generated by
the Gram-Schdmit orthogonalization procedure in the unit interval 
with the weight function $1$. They have the closed-form 
expression
\begin{align}\label{Pkx-closed-form}
    P_k(x) = \sqrt{2k+1}
    \sum_{0\le l\le k}\mybinom{k}{l}
    \mybinom{k+l}{l}(-1)^{k+l}x^l,
\end{align}
and satisfy the relation $P_k(1-x)=(-1)^kP_k(x)$.
In particular, $P_0(x)=1$ and 
\begin{center}
\renewcommand*{\arraystretch}{1.5}
  \begin{tabular}{ccc}
    $P_1(x)$ & $P_2(x)$ & $P_3(x)$ \\  \hline
    $\sqrt{3}(2x-1)$ & $\sqrt{5}(6x^2-6x+1)$ &
    $\sqrt{7}(20x^3-30x^2+12x-1)$ \\
  \end{tabular}
\end{center}
Also the leading coefficient is given by 
\begin{align*}
    [x^k]P_k(x) = \mybinom{2k}{k}\sqrt{2k+1}, 
\end{align*}
where $[x^k]f(x)$ denotes the coefficient of $x^k$ in the Taylor 
expansion of $f(x)$.

For convenience, define the $m$-fold integral of the polynomial 
$P_k(x)$
\begin{align}\label{Pk-m}
    P_k^{(-m)}(x) 
    = \mathop{\int\cdots\int}_{0<t_{1}<\cdots<t_{m}<x} 
    P_k(t_{1}) \dd t_{1}\cdots \dd t_{m}
    =\int_0^x \frac{(x-t)^{m-1}}{(m-1)!}\,P_k(t) \dd t.
\end{align}
The template functions we need are then given as follows.
\begin{lm} For $m\ge1$ and $x\in[0,1]$, let
\begin{align}\label{ad-tau_mtx}
    \tau_m(t;x)
    = \frac{(x-t)^{m-1}}{(m-1)!}\, \Ind_{\{t\le x\}}
    -\sum_{0\le k<m} P_k(t)P_k^{(-m)}(x).
\end{align}
Then $\tau_m(t;x)\in C^{m-2}$ for $t\in[0,1]$, satisfies 
\eqref{tau_m-ortho} and is a piecewise polynomial function.
\end{lm}
\begin{proof}
Obviously, the template function $\tau_m$ is of $C^{m-2}$. Only
\eqref{tau_m-ortho} needs to be proved.
For that purpose, it suffices to show that
\begin{align}\label{tau_m-Pj}
    \int_0^x \tau_m(t;x) P_j(t) \dd u 
    = 0\qquad(0\le j\le m-1).
\end{align}
By definition, 
\[
    \int_0^1 \tau_m(t;x) P_j(t) \dd t
    = P_j^{(-m)}(x)
    - \sum_{0\le k<m}P_k^{(-m)}(x)\int_0^1 P_j(t) P_k(t)\dd t,
\]
which, by the orthogonality relation \eqref{Pk-ortho}, yields
\eqref{tau_m-Pj}.
\end{proof}
Figure \ref{fig:template} depicts $\tau_m(\cdot;x)$ for
$1\le m\le 4$.
\begin{figure}[h]
\begin{center}
\begin{tabular}{cc}
    \includegraphics[height=3.2cm]{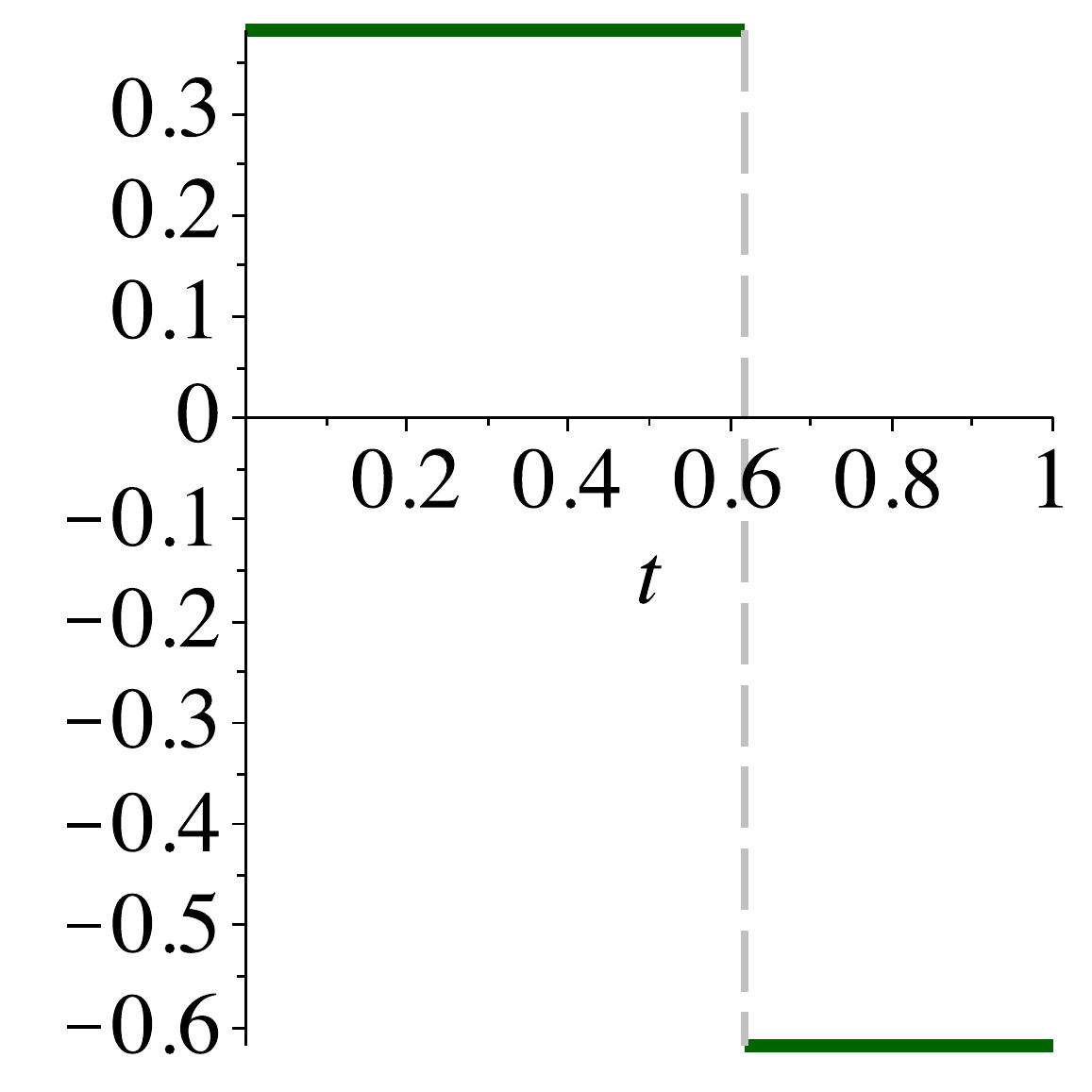} &
    \includegraphics[height=3.2cm]{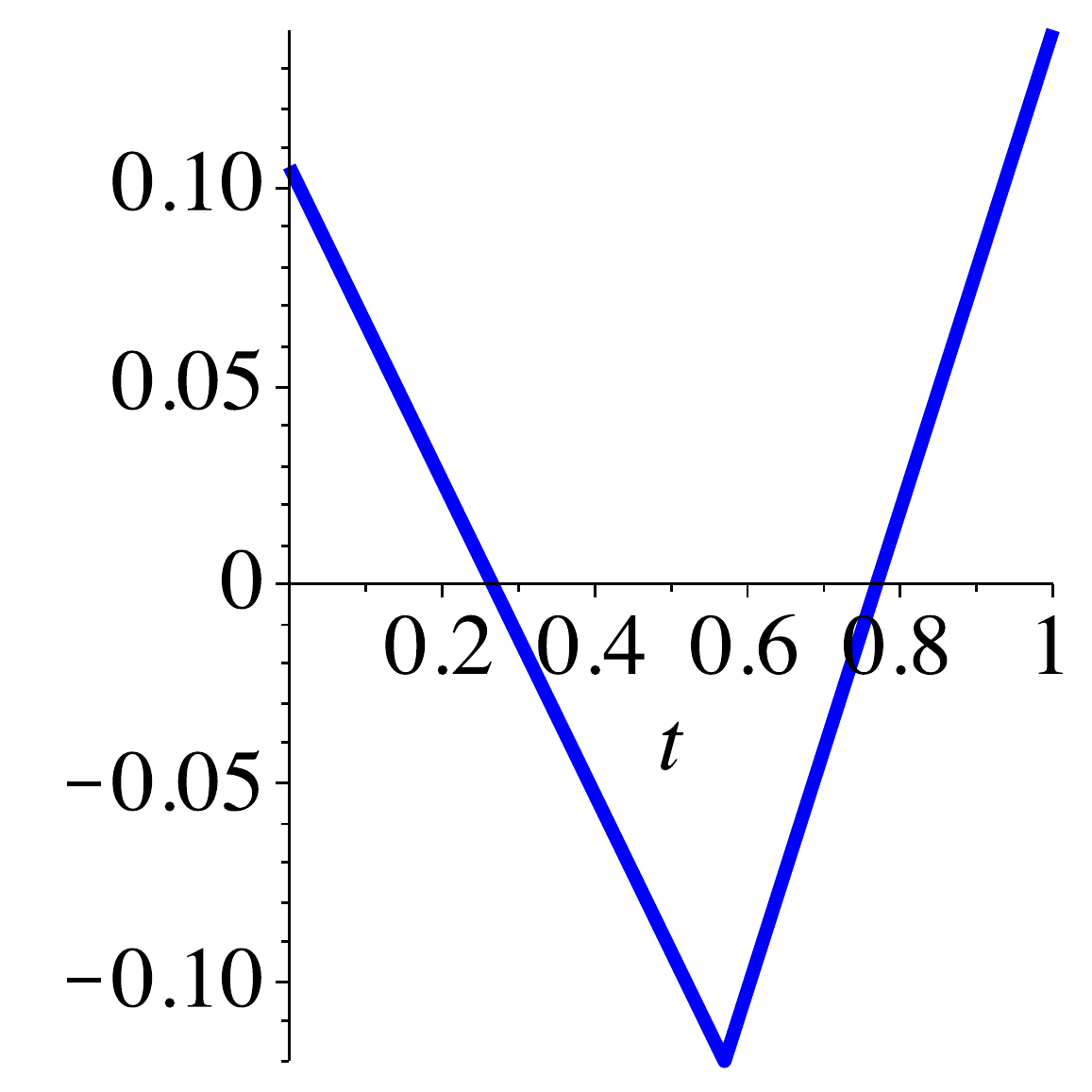} \\
    $\tau_1(t;x_0)$ & $\tau_2(t;x_0)$ \\ 
    \includegraphics[height=3.2cm]{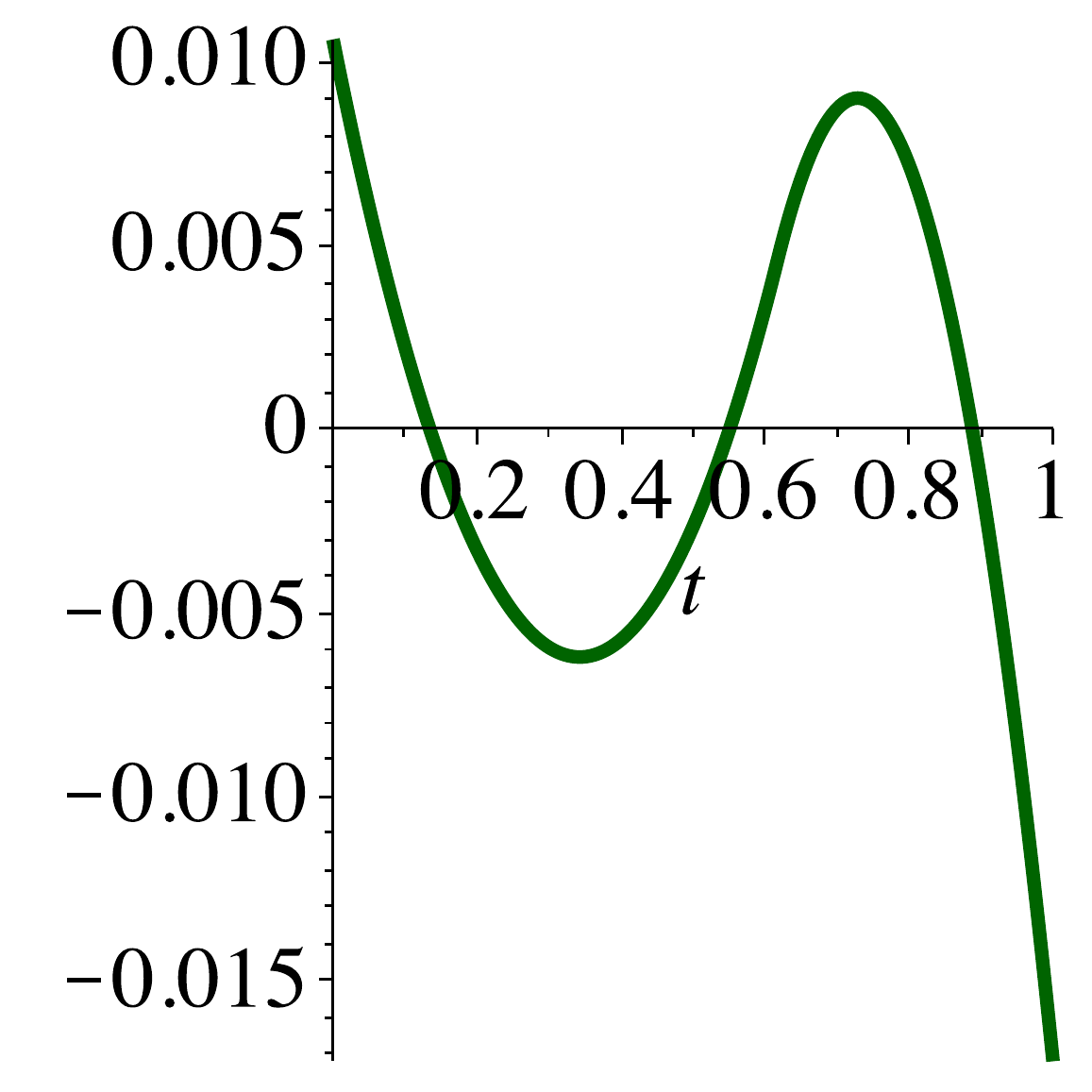} &
    \includegraphics[height=3.2cm]{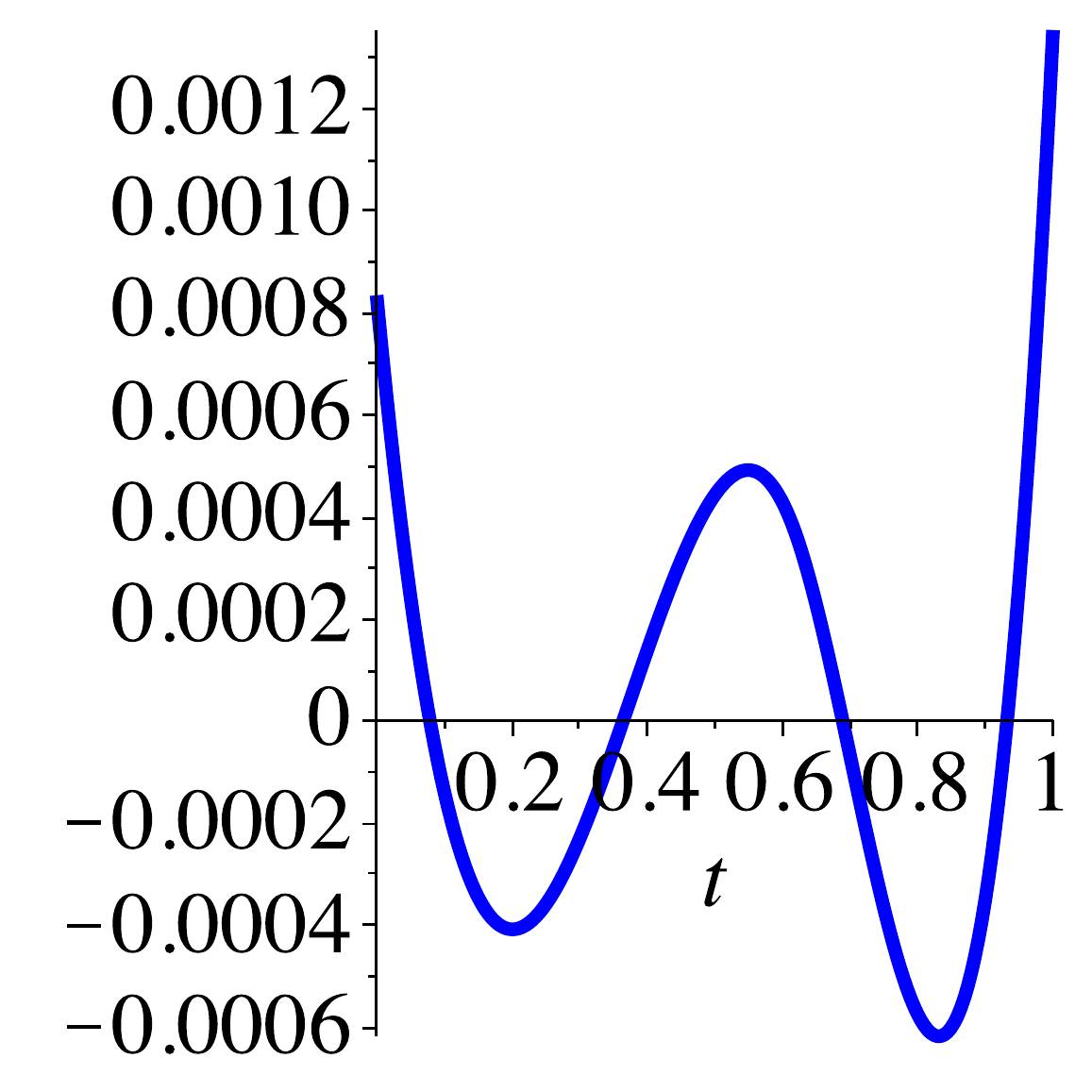} \\
	$\tau_3(t;x_0)$ & $\tau_4(t;x_0)$ 
\end{tabular}
\end{center}
\vspace*{-.3cm}
\caption{\emph{The template functions $\tau_m(t;x_0)$ for the 
generalized Anderson-Darling statistics for $1\le m\le 4$; here 
$x_0=\frac{\sqrt{5}-1}2$, the reciprocal of the golden ratio.}}
\label{fig:template}
\end{figure}

Observe, by \eqref{Pkx-closed-form} and by a direct integration, that 
\begin{align}\label{Pkx-m}
    P_k^{(-m)}(x)
    = \sqrt{2k+1}\sum_{0\le j\le k} 
    \frac{(-1)^{k+j}(k+j)!}{j!(k-j)!(m+j)!}\,
    x^{m+j}.
\end{align}
Without signs and the normalizing factor $\sqrt{2k+1}$, $P_k(x)$
corresponds to sequence A063007 (and, up to a shift of index,
A109983) in the Online Encyclopedia of Integer Sequences
\citep{OEIS}, and $P_k^{(-1)}(x)$ to A088617; they all
have direct combinatorial interpretations in terms of parameters on
certain lattice paths. Note that $P_k^{(-1)}(x)$ also appeared in
\cite{MacNeill78} but in the form  
\[
    \frac{(2k)!}{k!}
    \sum_{0\le j\le \lfloor k/2\rfloor}
    \frac{(-1)^j}
    {16^j j!^2(k-2j)!\binom{k-\frac12}{j}}
    \int_0^x \lpa{v-\tfrac12}^{k-2j}\dd v,
\]
which can be proved to be identical to \eqref{Pkx-m} with $m=1$ and 
without the factor $\sqrt{2k+1}$.

\subsection{The generalized Anderson-Darling statistic}
\label{subsec:sample-ad}

With the template functions $\tau_m(t;x)$ at hand, we now define, 
similar to \eqref{BN},
\begin{align}
    B^{[m]}_{N}(x)
    &:= \sqrt{N} \int_0^1 \tau_m(t;x) 
    \dd F_{N}(t) \nonumber \\
    &= \sqrt{N} \Biggl( 
    \mathop{\int\cdots\int}_{0<t_{1}<\cdots<t_{m}<x} 
    \dd F_{N}(t_1) \dd t_2 \cdots \dd t_m \nonumber\\
    &\qquad\qquad - \sum_{0\le k<m} P_k^{(-m)}(x) 
    \int_0^1 P_k(t) \dd F_{N}(t) \Biggr).
\label{BNm}
\end{align}
Note that $B^{[1]}_{N}(x)=B_N(x)$.
Then, as in \eqref{AN}, we define
the generalized Anderson-Darling (GAD) statistic as
\begin{equation}\label{ANm}
    A^{[m]}_{N} 
    := \int_0^1 \frac{B^{[m]}_{N}(x)^2}{(x(1-x))^{m}}
    \dd x,
\end{equation}
which equals the original Anderson-Darling statistic when $m=1$\,:
$A^{[1]}_{N}=A_{N}$. However, unlike $A_N^{[1]}$, the variance of
$B^{[m]}_{N}(x)$ in \eqref{BNm} with $x$ fixed is proportional to
$(x(1-x))^{2m-1}$ (see Theorem \ref{thm:cov}), which is different
from the denominator $(x(1-x))^{m}$ we used in \eqref{ANm} except
for the case $m=1$. The motivation of choosing such a weight will
become clear later.

On the other hand, it is less obvious whether the integral \eqref{ANm}
exists because of the seemingly singular factors in the denominator
at the end points. We now show that the GAD statistics $A^{[m]}_{N}$ 
can be expressed in terms of the samples $X_i$'s, $1\le i\le N$.

\begin{lm}
\label{lm:invariance}
Assume $X_i\in (0,1)$ for $1\le i\le N$. Then the GAD statistics 
$A^{[m]}_{N}$ in \eqref{ANm} are well-defined for $m\ge1$, and are 
invariant with respect to the changes $X_i\mapsto 1-X_i$ for all 
$i=1,\dots,N$.
\end{lm}
\begin{proof}
We first rewrite the template function $\tau_m(t;x)$. By
the completeness of the Legendre polynomial system and
the relation $P_k(1-x)=(-1)^k P_k(x)$,
\begin{align*}
    \frac{(x-t)^{m-1}}{(m-1)!}
    &= \sum_{0\le k<m} P_k(t)
    \biggl( \int_x^1 + \int_0^x \biggr)
     \frac{(x-u)^{m-1}}{(m-1)!} P_k(u) \dd u \\
    &= (-1)^{m-1} \sum_{0\le k<m} 
    P_k(1-t) P_k^{(-m)}(1-x)  
    + \sum_{0\le k<m} P_k(t)P_k^{(-m)}(x).
\end{align*}
Substituting this into \eqref{ad-tau_mtx}, we obtain
\begin{align}\label{tau_mtx-PP}
    \tau_m(t;x)
    &= (-1)^{m-1} \sum_{0\le k<m} 
    P_k(1-t)P_k^{(-m)}(1-x) \Ind_{\{t\le x\}}
	\nonumber \\
    &\qquad\qquad - \sum_{0\le k<m} 
	P_k(t)P_k^{(-m)}(x) \Ind_{\{t>x\}}.
\end{align}
Now expanding the square in \eqref{ANm} as a double integral, giving
\[
    A^{[m]}_{N} 
    = \int_0^1 \int_0^1 \biggl( 
    \int_0^1 \frac{\tau_m(t;x) \tau_m(u;x)}
    {(x(1-x))^m}\dd x \biggr) \dd F_{N}(t) \dd F_{N}(u),
\]
where the expression between the parentheses becomes, by 
\eqref{tau_mtx-PP},
\begin{equation}\label{parentheses}
\begin{aligned}
    & \sum_{0\le k,l<m} P_k(1-t) P_l(1-u)\int_{t\vee u}^1 
    \frac{P_k^{(-m)}(1-x) P_{l}^{(-m)}(1-x)}
    {(x(1-x))^m}\dd x \\
    &  - (-1)^{m-1} \sum_{0\le k,l<m} 
    P_k(1-t\wedge u) P_l(t\vee u)
    \int_{t\wedge u}^{t\vee u} 
    \frac{P_k^{(-m)}(1-x) 
    P_{l}^{(-m)}(x)}{(x(1-x))^m}\dd x \\
    & + \sum_{0\le k,l<m}P_k(t) P_l(u) \int_0^{t\wedge u}
    \frac{P_k^{(-m)}(x) P_{l}^{(-m)}(x)}
    {(x(1-x))^m}\dd x.
\end{aligned}
\end{equation}
Since $P_k^{(-m)}(x)$ and $P_k^{(-m)}(1-x)$ contain factors of the
forms $x^{m+j}$ and $(1-x)^{m+j}$ for $j\ge0$ (see \eqref{Pkx-m}), 
respectively, the three integrals in \eqref{parentheses} are 
well-defined if $t\vee u<1$ and $t\wedge u>0$. On the other hand, 
\eqref{parentheses} is invariant with respect to the change of 
variables $(t,u)\mapsto (1-t,1-u)$.
\end{proof}

With more calculations, it is possible to derive a more precise 
expression for $A_N^{[m]}$.
\begin{lm}\label{lm:ANm-ab}
Let $a_{ij}=X_i\wedge X_j$ and $b_{ij}=1-X_i\vee X_j$. Then
\begin{equation}\label{ANm-PQ}
    A_N^{[m]} 
    = -\sfrac1N\sum_{1\le i,j\le N}
    R_m(a_{ij},b_{ij})(\log a_{ij} + \log b_{ij})
    + \sfrac{1}{N}\sum_{1\le i,j\le N}Q_m(a_{ij},b_{ij}),
\end{equation}
where $Q_m(a,b)$ and $R_m(a,b)$ are symmetric polynomials in $a$ and $b$ 
of degree $2(m-1)$.    
\end{lm}
In particular, $R_m$ is given explicitly by  
\begin{align}
    R_m(a,b)
    &:= \frac{(1-a-b)^{m-1}}{\sqrt{2m-1}\,(m-1)!^2}\,
    P_m\left(\frac{-ab}{1-a-b}\right) \nonumber \\
    &= \frac{(-1)^{m-1}}{(m-1)!^2}\sum_{0\le \ell< m}
    \binom{m-1}{\ell}\binom{m-1+\ell}{\ell}
    (ab)^\ell(1-a-b)^{m-1-\ell}.
\label{Rm}
\end{align}
But the expression for $Q_m$ we obtained is very lengthy and 
complicated. A proof of Lemma~\ref{lm:ANm-ab} is long and provided
in Appendix~\ref{subsec:Anm-PQ}.

In particular, the two statistics $A^{[1]}_{N}$ and $A^{[2]}_{N}$ 
have the following expressions in terms of the sample $X_i$'s. Let 
$a_{ij}=X_i\wedge X_j$ and $b_{ij}=1-X_i\vee X_j$.
\begin{align*}
    A^{[1]}_{N}
    &= \sfrac{1}{N}\sum_{1\le i,j\le N} 
    ( -1 - \log(1-a_{ij}) - \log(1-b_{ij}) ) \\
    &= -N -\sfrac{1}{N}\sum_{1\le i\le N} 
    \bigl( (2N-2i+1)\log (1-X_{(i)}) 
    + (2i-1)\log X_{(i)} \bigr),
\end{align*}
where $X_{(1)}\le\cdots\le X_{(N)}$ are the order statistics, and
\begin{align*}
    A^{[2]}_{N}
    &= \sfrac{1}{N}\sum_{1\le i,j\le N} \Bigl(
    \sfrac{1}{6} (1+a_{ij}+b_{ij}+10 a_{ij} b_{ij}) \\
    & \qquad\quad
    + (1-a_{ij}-b_{ij}+2 a_{ij}b_{ij})
    (\log(1-a_{ij})+\log(1-b_{ij})) \Bigr) \\
    &= \sfrac{1}{N}\sum_{1\le i,j\le N} \Bigl(
    \sfrac{1}{6} (2-|X_i-X_j|+10 (X_i\wedge X_j - X_i X_j)) \\
    & \qquad\quad
    + (|X_i-X_j|+2 (X_i\wedge X_j - X_i X_j)) 
    \log(X_i\vee X_j - X_i X_j) \Bigr).
\end{align*}
A direct calculation of this double sum leads to quadratic time (in 
$N$) complexity. In terms of the order statistics, we can rearrange 
the terms on the right-hand side and obtain a procedure with linear 
time complexity:
\begin{align*}
    A_N^{[2]}
    &:= \sfrac13N +\sfrac1N\sum_{1\le i\le N}\Lpa{\sfrac{11}3N-4i+2}X_{(i)}
    + \sfrac2N\sum_{1\le i\le N}(N-i+X_{(i)})X_{(i)}\log X_{(i)} \\
    &\quad +\sfrac2N\sum_{1\le i\le N}
    (i-X_{(i)})X_{(i)}\log(1-X_{(i)})
    + \sfrac{2}{N}\Lpa{\sum_{1\le i\le N}X_{(i)}}
    \Lpa{\sum_{1\le i\le N}\log X_{(i)}} \\
    &\quad 
    -\sfrac5{3N}\Lpa{\sum_{1\le i\le N}X_{(i)}}^2
    -\sfrac{4}{N}\Lpa{\sum_{1\le i\le N}X_{(i)}}
    \Lpa{\sum_{1\le i\le N}X_{(i)}\log X_{(i)}} \\
    &\quad +\sfrac{2}N\sum_{1\le i\le N}(2X_{(i)}-1)
    (\log X_{(i)} -\log(1-X_{(i)}))\sum_{1\le j<i}X_{(j)}.
\end{align*}
Based on Lemma~\ref{lm:ANm-ab}, we now show that $A^{[m]}_N$ can be 
computed in time $O(N\log N)$.
\begin{thm}
\label{thm:ANm}
The GAD statistics $A_N^{[m]}$ ($m\ge 1$) can be computed in 
worst-case time complexity $O(N\log N)$. If in addition the $X_i$'s 
are already sorted, then $A_N^{[m]}$ can be computed in linear time.
\end{thm}
\begin{proof}
We may assume that the $X_i$'s are sorted in increasing order, for
otherwise, we can apply any standard sorting algorithm with
worst-case time complexity $O(N\log N)$. From the decomposition
\eqref{ANm-PQ}, it suffices to examine sums of the form
($a=1-X_i\wedge X_j$, $b=X_i\vee X_j$)
\begin{align*}
    \sum_{1\le i,j\le N}(a^k b^\ell +a^\ell b^k)
    &= 2\sum_{1\le i\le N}X_i^\ell
    \sum_{1\le j<i}(1-X_j)^k +
    \sum_{1\le i\le N}X_i^k \sum_{1\le j<i}(1-X_j)^\ell \\
    &\qquad + \sum_{1\le i\le N} 
    \lpa{(1-X_i)^k X_i^\ell + (1-X_i)^\ell X_i^k}
\end{align*}
Note that the sum
\begin{align*}
    \sum_{1\le i\le N}X_i^\ell \sum_{1\le j<i}(1-X_j)^k
    &= \sum_{1\le j\le N}(1-X_j)^k\sum_{j<i\le N}X_i^\ell \\
    &= \sum_{1\le j\le N}(1-X_j)^k
    \llpa{\sum_{1\le i\le N}X_i^\ell 
    -\sum_{1\le i<j}X_i^\ell -X_j^\ell}
\end{align*}
So all calculations can be reduced to computing sums of the forms
$\sum_{1\le i<j}X_i^k$, which can be pre-computed once and then
stored (in $O(N)$ time). The computation of other sums of the forms 
$\sum_{1\le i,j\le N}a^k b^\ell \log a$ and $\sum_{1\le i,j\le N}a^k
b^\ell \log b$ are similar. Except for sorting, all calculations can
be done in $O(N)$ time. This proves the theorem.
\end{proof} 
Sorting can be made more efficient on average if more \emph{a priori} 
information of the data is known, for example, if $X_i$ are regularly 
distributed over the unit interval; see \citep[Sec.\ 5.2.5]{Knuth98}
or \cite{Devroye86}.

\begin{rem}
In the one-way ANOVA model, \cite{Hirotsu86} proposed two 
statistics $\chi^{*2}$ and $\chi^{\dagger 2}$ for testing
the equality against the monotonicity of mean profile, and for
testing monotonicity of mean profile against the convexity,
respectively. These two tests correspond to $A^{[1]}_{N}$ and 
$A^{[2]}_{N}$ above.
\end{rem}

\subsection{The limiting GAD processes and their covariance 
kernels}\label{subsec:kl-cov-ad}

In this section, we examine the limiting distribution of the GAD  
statistic under the null hypothesis $H_0$.

\paragraph{The limiting GAD processes.} Let $W(x)$ be the Wiener
process on $[0,1]$, and let $B(x)=W(x)-x W(1)$ be the Brownian
bridge. Define
\begin{align}
    B^{[m]}(x)
    &:= \int_0^1 \tau_m(t;x) \dd B(t) \nonumber \\
    &= \mathop{\int\cdots\int}_{0<t_{1}<
       \cdots<t_{m}<x} B(t_{1}) \dd t_{1}\cdots \dd t_{m-1} 
       - \sum_{0\le k<m} P_k^{(-m)}(x) \int_0^1 P_k(t) \dd B(t),
\label{Bm}
\end{align}
where $\tau_m$ is defined in \eqref{ad-tau_mtx}. Since the Brownian
bridge is continuous on the unit interval, namely, $B(\cdot)\in
C^0[0,1]$, we see that $B^{[m]}(\cdot)\in C^{m-1}[0,1]$ with
probability one.

The following lemma follows from a weak convergence argument in $L^2$
(Sec.~1.8 and Example 1.8.6 of \cite{vanderVaart96}; see
also \cite{Tsukuda14}).

\begin{lm}\label{lm:L2} 
Let $\mu$ be any finite measure on $[0,1]$. Then, as $N\to\infty$,
$B^{[m]}_{N}(\cdot) \stackrel{d}{\to} B^{[m]}(\cdot)$ in
$L^2([0,1],\mu)$ and $A^{[m]}_{N}\stackrel{d}{\to} A^{[m]}$, where
\begin{equation}\label{Am}
    A^{[m]} 
    = \int_0^1 \frac{B^{[m]}(x)^2}{(x(1-x))^{m}} \dd x.
\end{equation}
\end{lm}
\begin{proof}
The stochastic map $x\mapsto B_{N}(x)=\sqrt{N}(F_{N}(x)-x)$ converges
weakly to $B(\cdot)$ in $L^2([0,1],\mu)$. Because of $\int_0^1 P_k(t)
\dd B(t) = - \int_0^1 P_k'(t) B(t) \dd t$, $B^{[m]}(\cdot)$ in
\eqref{Bm} is continuous in $B(\cdot)$ in $L^2([0,1],\mu)$.
Similarly, $A^{[m]}$ is continuous in $B(\cdot)$. The convergences
$B^{[m]}_{N}(\cdot) \stackrel{d}{\to} B^{[m]}(\cdot)$ and
$A^{[m]}_{N}\stackrel{d}{\to} A^{[m]}$ then follow from the
continuous mapping theorem.
\end{proof}

We refer to the signed square-root of the integrand function in
\eqref{Am} $x\mapsto \frac{B^{[m]}(x)}{(x(1-x))^{m/2}}$ as the
\emph{generalized Anderson-Darling (GAD) processes}. We now clarify
the stochastic structure of such processes, starting first with
defining the required eigenfunctions, sketching a heuristic
derivation of the KL expansions and then proving rigorously all the
claims later.

\paragraph{Associated Legendre polynomials.} To pave way for deriving 
the KL expansions, we introduce the following definition.
\begin{df} Define the normalized shifted associated Legendre 
polynomials on $[0,1]$ by
\begin{align}\label{Pkm}
    P^{[m]}_{k}(x)
    :=(-1)^m\sqrt{\frac{(k-m)!}{(k+m)!}}\,
    (x(1-x))^{m/2} P_k^{(m)}(x),
\end{align}
where $P_k^{(m)}(x)$ denotes the $m$-th derivative of $P_k(x)$ with 
respect to $x$.
\end{df}
For each $m$, $\{P^{[m]}_{k}(x)\}_{k\ge m}$ forms an orthonormal 
basis, that is, 
\[
    \int_0^1 P^{[m]}_{k}(x) P^{[m]}_{l}(x) \dd x 
    = \Ind_{\{k=l\}} \qquad(k,l\ge0). 
\]
These associated polynomials are connected to the integrated 
Legendre polynomials $P_k^{(-m)}(x)$ by the following relation.
\begin{lm}\label{lm:dublin}
For $0\le m\le k$
\begin{align}\label{dublin}
    P_{k}^{(-m)}(x) 
    = \sqrt{\frac{(k-m)!}{(k+m)!}} \,(x(1-x))^{m/2} 
    P^{[m]}_{k}(x).
\end{align}
\end{lm}
\begin{proof}
The Legendre polynomial $P_k(x)$ satisfies the following differential
equations, called Legendre's differential equation when $m=1$.
For 
$1\le m\le k$,
\begin{align}\label{diffeq}
    \Bigl(\Dx\Bigr)^m \lpa{(x(1-x))^m P_{k}^{(m)}(x)}
    - (-1)^m \frac{(k+m)!}{(k-m)!} P_{k}(x) = 0.
\end{align}
To prove \eqref{diffeq}, let $q_{k-1}(x)$ be an arbitrary polynomial 
of degree $k-1$. We show that
\begin{equation}
    \int_0^1 q_{k-1}(x) \Bigl(\Dx\Bigr)^m \lpa{(x(1-x))^m 
    P_{k}^{(m)}(x)} \dd x =0.
\label{I}
\end{equation}
In fact, by integration by parts,
\begin{align*}
    \mbox{LHS of \eqref{I}}
    &= (-1)^m \int_0^1 q_{k-1}^{(m)}(x) 
    (x(1-x))^m P_{k}^{(m)}(x) \dd x \\
    &= \int_0^1 \Bigl(\Dx\Bigr)^m 
    \lpa{q_{k-1}^{(m)}(x) (x(1-x))^m}
    P_{k}(x) \dd x \\
    &= 0.
\end{align*}
The last equality holds because the degree of the polynomial
$\bigl(\Dx\bigr)^m\lpa{q_{k-1}^{(m)}(x) (x(1-x))^m}$ is less than or
equal to $k-1$. Since $q_{k-1}(x)$ is arbitrary, there is a constant
$M$ such that
\[
    \Bigl(\Dx\Bigr)^m \lpa{(x(1-x))^m 
    P_{k}^{(m)}(x)} = M P_{k}(x).
\]
By checking the leading term, we see that $M=(-1)^m
\frac{(k+m)!}{(k-m)!}$. This proves \eqref{diffeq}, which,
together with the initial condition
\[
    \Bigl(\Dx\Bigr)^j \lpa{(x(1-x))^m 
    P_{k}^{(m)}(x)}\big|_{x=0} = 0 \qquad (0\le j<m),
\]
implies
\[
    (x(1-x))^m
    P_{k}^{(m)}(x) = 
   (-1)^m \frac{(k-m)!}{(k+m)!} P_{k}^{(-m)}(x),
\]
that is, \eqref{dublin} holds.
\end{proof}

Lemma \ref{lm:dublin}, coupling with \eqref{Pkm} and \eqref{dublin},  
yields
\begin{equation}\label{orthogonality}
    \int_0^1 P_{k}^{(m)}(x) P_{\ell}^{(-m)}(x) \dd x 
    = (-1)^m \int_0^1 P^{[m]}_{k}(x) P^{[m]}_{\ell}(x) \dd x 
    = (-1)^m \Ind_{\{k=\ell\}},
 \end{equation}
for $0\le m\le  \ell,k$.

\paragraph{Covariance kernel.} Let
\[
    K^{[m]}(x,y)
    = \Cov\biggl(\frac{B^{[m]}(x)}{(x(1-x))^{m/2}},
    \frac{B^{[m]}(y)}{(y(1-y))^{m/2}}\biggr)
\]
be the covariance functions of the GAD processes.
\begin{thm}
\label{thm:cov}
For $m\ge1$
\begin{align}
\label{Kmxy2}
    K^{[m]}(x,y)
    &= \frac{1}{(m-1)!^2}\sum_{0\le k<m} 
    \frac{1}{m+k} \binom{m-1}{k} 
    \frac{(x\wedge y - x y)^{m/2+k} 
    |x-y|^{m-1-k}}{(x\vee y - x y)^{m/2}}.
\end{align}
\end{thm}
When $x\ne y$, \eqref{Kmxy2} can alternatively be expressed as
\[
    K^{[m]}(x,y)
    = \frac{(-1)^m}{(m-1)!^2}\,
    \text{Beta}_{m,m}\biggl(-\frac{x\wedge y - x y}{|x-y|}\biggr) 
    \frac{|x-y|^{2m-1}}{(x(1-x)y(1-y))^{m/2}},
\]
where $\text{Beta}_{a,b}$ is the incomplete beta function with 
parameters $(a,b)$. In particular, when $x=y$,
\begin{align*}
    K^{[m]}(x,x)
    = \frac{(x(1-x))^{m-1}}{(2m-1) (m-1)!^2}.
\end{align*}

The covariance functions $K^{[m]}(x,y)$ for $m=1,2,3$ are listed 
as follows.
\begin{small}
\begin{align*}
    K^{[1]}(x,y) 
    &= \sqrt{\frac{x\wedge y -xy}{x\vee y -xy}}, \\
    K^{[2]}(x,y) 
    &= \frac{x\wedge y- xy}{x\vee y -xy}\cdot
    \frac{3(x\vee y)-(x\wedge y)+2xy}{6}, \\
    K^{[3]}(x,y) 
    &= \frac{(x\wedge y -xy)^{\frac{3}{2}}}
    {(x\vee y -xy)^{\frac{3}{2}}}
    \biggl(\frac{\lpa{(x\vee y) -\frac34 xy}^2}{12}+
    \frac{\lpa{(x\wedge y) +\frac32 xy}^2}{120} 
    -\frac{xy(3xy+8)}{192}\biggr).
\end{align*}
\end{small}

\begin{rem}
Since $|x-y|=x\vee y-x\wedge y$ and $xy=(x\wedge y)(x\vee y)$,
$K^{[m]}(x,y)$ in \eqref{Kmxy2} is a function of $x\wedge y$ and
$x\vee y$. More precisely, $K^{[m]}(x,y)$ is a symmetric function
in $x\wedge y$ and $1-x\vee y$ because of the invariance property
of $A^{[m]}_N$ given in Lemma \ref{lm:invariance}.
\end{rem}

We use a direct combinatorial proof, which is somewhat lengthy but 
self-contained.
\begin{proof}[Proof of Theorem \ref{thm:cov}]
It suffices to show that
\begin{align*}
    &\Cov\lpa{B^{[m]}(x),B^{[m]}(y)}\\
    &\qquad = \frac{1}{(m-1)!^2}\sum_{0\le k<m} 
    \frac{1}{m+k} \binom{m-1}{k} 
    (x\wedge y - x y)^{m+k} |x-y|^{m-1-k}.
\end{align*}
Note that when $m=1$, the right-hand side becomes $x\wedge y - x y$,
which is the covariance function of the Brownian bridge. By definition
\begin{align}
    &\Cov\lpa{B^{[m]}(x),B^{[m]}(y)}\nonumber\\
    &\qquad = \int_0^1 \tau_m(t;x) \tau_m(t;y) 
	\dd t\nonumber \\
    &\qquad = \int_0^{x\wedge y}
    \frac{(t-x)^{m-1}(t-y)^{m-1}}{(m-1)!^2}\dd t
    - \sum_{0\le k<m} P_{k}^{(-m)}(x) P_{k}^{(-m)}(y).
\label{covBm}
\end{align}
Without loss of generality, assume from now on $m\ge1$ and 
$0\le x\le y\le 1$. Let 
\begin{align}\label{Jmxy}
    J_m(x,y)
    := \int_0^x \frac{(x-v)^{m-1}
    (y-v)^{m-1}}{(m-1)!^2}\dd v-
    \sum_{0\le k<m} P_{k}^{(-m)}(x) 
    P_{k}^{(-m)}(y).
\end{align}
We show that 
\begin{align}\label{Qmx}
    J_m(x,y) 
    = \frac{(x(1-y))^{m}}{(m-1)!} \sum_{0\le k<m}
    \frac{1}{m+k} \binom{m-1}{k} (x(1-y))^k (y-x)^{m-1-k},
\end{align}
starting with the first term on the right-hand side of \eqref{Jmxy}, 
which can be written as
\[
    \int_0^x \frac{(x-v)^{m-1}
    (y-v)^{m-1}}{(m-1)!^2}\dd v
    =\frac{x^{m}}{(m-1)!^2} 
    \int_0^1 (1-v)^{m-1} (y-xv)^{m-1} \dd v.    
\]
For the second term on the right-hand side of \eqref{Jmxy}, 
we have, by \eqref{Pk-m}, 
\begin{align*}
    &\sum_{0\le k<m} 
    P_{k}^{(-m)}(x) P_{k}^{(-m)}(y) \\
    &\qquad
    = \sum_{0\le j,\ell<m} \frac{(-1)^{j+\ell}
    x^{m+j}{y^{m+\ell}}}{j!\ell!(m+j)!(m+\ell)!}
    \sum_{(j\vee \ell)\le k<m} 
    (2k+1)\frac{(k+j)!(k+\ell)!}{(k-j)!(k-\ell)!}.
\end{align*}
By a simple induction on $m$, the inner sum equals 
\begin{equation}
\label{inner_sum}
    \sum_{(j\vee \ell)\le k<m}(2k+1)
    \frac{(k+j)!(k+\ell)!}{(k-j)!(k-\ell)!}
    = \frac{(m+j)!(m+\ell)!}{(m-1-j)!(m-1-\ell)!(j+\ell+1)},
\end{equation}
so that
\begin{align*} 
    &\sum_{0\le k<m} P_{k}^{(-m)}(x) P_{k}^{(-m)}(y)\\
    &\qquad= \sum_{0\le j,\ell<m}
    \frac{(-1)^{j+\ell} x^{m+j} y^{m+\ell}}
    {j!\ell!(m-1-j)!(m-1-\ell)!(j+\ell+1)} \\
    &\qquad= \frac{1}{(m-1)!^2} 
    \sum_{0\le j,\ell<m} \binom{m-1}{j} 
    \binom{m-1}{\ell} \frac{(-1)^{j+\ell}}
    {j+\ell+1}\, x^{m+j}y^{m+\ell} \\
    &\qquad= \frac{x^{m}y^{m}}{(m-1)!^2} 
    \int_0^1 (1-vx)^{m-1} (1-vy)^{m-1} \dd v.
\end{align*}
On the other hand, 
\[
    \mbox{RHS of \eqref{Qmx}}
    = \frac{(x(1-y))^{m}}{(m-1)!^2}
    \int_0^1 v^{m-1} (x(1-y)v+y-x)^{m-1} \dd v.
\]
Thus \eqref{Qmx} will follow from the identity
\begin{align*}
    & \frac{x^{m}}{(m-1)!^2} 
    \int_0^1 (1-v)^{m-1} (y-xv)^{m-1} \dd v \\
    &\qquad\qquad - \frac{x^{m} y^{m}}{(m-1)!^2} 
    \int_0^1 (1-vx)^{m-1} (1-vy)^{m-1} \dd v \\
    &\qquad
    = \frac{(x(1-y))^{m}}{(m-1)!^2} 
    \int_0^1 v^{m-1} (x(1-y)v+y-x)^{m-1} \dd v,
\end{align*}
or, equivalently,
\begin{equation}\label{cov-id}
\begin{aligned}
    & \int_0^1 v^{m-1} (y-x+xv)^{m-1} \dd v - y^{m} 
    \int_0^1 (1-vx)^{m-1} (1-vy)^{m-1} \dd v \\
    & \qquad= (1-y)^{m} \int_0^1 v^{m-1} (x(1-y)v+y-x)^{m-1} \dd v.
\end{aligned}
\end{equation}
By multiplying both sides by $\frac{z^{m-1}}{(m-1)!}$ and then 
summing over all $m\ge1$, it suffices only to prove that 
\begin{align}\label{int-exp}
    \int_0^1 \lpa{e^{zv(y-x+xv)} - y e^{zy(1-vx)(1-vy)}
    -(1-y)e^{zv (x(1-y)v+y-x)} } \dd v = 0.
\end{align}
Let 
\[
     I(x) := \int_0^x e^{t^2}\dd t,
\]
which equals (up to a constant $\sqrt{\pi}/2$) the imaginary error 
function. With this integral, all integrals in \eqref{int-exp} are
of the form 
\[
    \int_0^1 e^{tv+wv^2}\dd v 
    = \frac{e^{-\frac{t^2}{4w}}}{\sqrt{w}}
    \llpa{I\llpa{\frac{t}{2\sqrt{w}}+\sqrt{w}} 
    - I\llpa{\frac{t}{2\sqrt{w}}}}.
\]
Applying this identity, we then have
\begin{align*}
    &\int_0^1 e^{zv(y-x+xv)} \dd v\\
    &\qquad= \frac{e^{-\frac{z(y-x)^2}{4x}}}{\sqrt{xz}} 
    I\llpa{\frac{(x+y)\sqrt{z}}{2\sqrt{x}}}
    - \frac{e^{-\frac{z(y-x)^2}{4x}}}{\sqrt{xz}} 
    I\llpa{\frac{(y-x)\sqrt{z}}{2\sqrt{x}}}, \\
    &-y\int_0^1 e^{zy(1-vx)(1-vy)} \dd v \\
    &\qquad= -\frac{e^{-\frac{z(y-x)^2}{4x}}}{\sqrt{xz}} 
    I\llpa{\frac{(y+x-2xy)\sqrt{z}}{2\sqrt{x}}}
    + \frac{e^{-\frac{z(y-x)^2}{4x}}}{\sqrt{xz}} 
    I\llpa{\frac{(y-x)\sqrt{z}}{2\sqrt{x}}},
\end{align*}
and
\begin{align*}
    & -(1-y)\int_0^1 e^{zv\pa{x(1-y)v+y-x}} \dd v \\
    & \qquad= -\frac{e^{-\frac{z(y-x)^2}{4x}}}{\sqrt{xz}} 
    I\llpa{\frac{(x+y)\sqrt{z}}{2\sqrt{x}}}
    +\frac{e^{-\frac{z(y-x)^2}{4x}}}{\sqrt{xz}} 
    I\llpa{\frac{(y+x-2xy)\sqrt{z}}{2\sqrt{x}}}.
\end{align*}
Since the sum of the right-hand sides of the three relations is zero, 
this proves \eqref{int-exp} and thus \eqref{cov-id}.
\end{proof}

\subsection{Eigenfunction expansions and KL expansions} 
\label{subsec:eigen-kl}

\begin{lm}
\label{lm:eigeneq}
For $m\ge1$, the following eigenequation  
\begin{align}\label{eigeneq}
    \int_0^1 K^{[m]}(x,y) P^{[m]}_{k}(y) \dd y 
    = \frac{(k-m)!}{(k+m)!} P^{[m]}_{k}(x)\qquad (1\le m\le k)
\end{align}
holds.
\end{lm}
\begin{proof}
Because of \eqref{Pkm} and \eqref{dublin}, it suffices to show that
\[
    \int_0^1 \Cov(B^{[m]}(x),B^{[m]}(y)) P_k^{(m)}(y) \dd y 
    = (-1)^{m} P_k^{(-m)}(x).
\]
By \eqref{covBm}, we consider first the double integral
\begin{align*}
    &\int_0^1 P_k^{(m)}(y) \int_0^{x\wedge y}
    \frac{(t-x)^{m-1}(t-y)^{m-1}}
    {(m-1)!^2}\dd t \dd y \\
    &\qquad= \int_0^{x} \frac{(t-x)^{m-1}}{(m-1)!}
    \int_{t}^1 \frac{(t-y)^{m-1}}{(m-1)!} 
    P_k^{(m)}(y) \dd y \dd t.
\end{align*}
Since the inner integral on the right-hand side is a polynomial in 
$t$ of degree $k$, we have 
\begin{align*}
    &\int_{t}^1 \frac{(t-y)^{m-1}}{(m-1)!} P_k^{(m)}(y) \dd y\\
    &\qquad= \sum_{0\le \ell\le k} P_{\ell}(t)
    \int_0^1 P_{\ell}(s) \int_{s}^1 \frac{(s-y)^{m-1}}{(m-1)!} 
    P_k^{(m)}(y) \dd y \dd s \\
    &\qquad= (-1)^{m-1}\sum_{0\le \ell\le k} P_{\ell}(t)
    \int_0^1 \int_0^{y} \frac{(y-s)^{m-1}}{(m-1)!} P_{\ell}(s) \dd s
    P_k^{(m)}(y) \dd y \\
    &\qquad= (-1)^{m-1} \sum_{0\le \ell\le k} P_{\ell}(t)
    \int_0^1 P_{\ell}^{(-m)}(y) P_k^{(m)}(y) \dd y \\
    &\qquad= -P_k(t),
\end{align*}
where we used \eqref{orthogonality}. Thus
\begin{align*}
    &\int_0^1 
	\Cov\lpa{B^{[m]}(x),B^{[m]}(y)} P_k^{(m)}(y) \dd y \\
    &\qquad= \int_0^{x} \frac{(t-x)^{m-1}}{(m-1)!}
    \bigl(-P_k(t) \bigr) \dd t 
	- \sum_{0\le \ell<m} P_{\ell}^{(-m)}(x) 
    \int_0^1 P_{\ell}^{(-m)}(y) P_k^{(m)}(y) \dd y \\
    &\qquad= (-1)^m P_k^{(-m)}(x).
\end{align*} 
This proves \eqref{eigeneq}.
\end{proof}

Since Legendre polynomial system $\{P_k(t)\}_{k\ge 0}$ is 
complete, we have 
\[
    \frac{(x-t)^{m-1}}{(m-1)!}\, \Ind_{\{t\le x\}}
    = \sum_{k\ge 0} P_k(t)
    \int_0^x \frac{(x-u)^{m-1}}{(m-1)!} P_k(u) \dd u
    = \sum_{k\ge 0} P_k(t) P_k^{(-m)}(x).
\]
Thus
\begin{equation}\label{tau-exp}
    \tau_m(t;x) 
    = \sum_{k\ge m} P_k(t) P_k^{(-m)}(x),
\end{equation}
which, together with Lemma \ref{lm:dublin}, yields 
\begin{equation}\label{tau_m-P}
    \frac{\tau_m(t;x)}{(x(1-x))^{m/2}} 
    = \sum_{k\ge m} \sqrt{\frac{(k-m)!}{(k+m)!}}\,
    P_k(t) P^{[m]}_{k}(x) \qquad(m\ge1).
\end{equation}
Let 
\begin{equation*}
    \xi_k 
    = \int_0^1 P_k(t) \dd B(t)
    = \int_0^1 P_k(t) \dd W(t), 
\end{equation*}
be a standard i.i.d.\ Gaussian sequence. Then $\mathbb{E}(\xi_k\xi_l)
=\int_0^1 P_k(t) P_l(t) \dd t = \Ind_{\{k=l\}}$. By definition 
\eqref{BNm} and \eqref{tau_m-P}, we expect that 
\begin{equation}\label{KL}
    \frac{B^{[m]}(x)}{(x(1-x))^{m/2}}
    = \sum_{k\ge m} \sqrt{\frac{(k-m)!}{(k+m)!}}\,
    P^{[m]}_{k}(x) \xi_k.
\end{equation}
On the other hand, by \eqref{eigeneq}, the following eigenfunction 
expansion is also expected to be true
\begin{align}\label{Kmxy}
    K^{[m]}(x,y)
    &= \sum_{k\ge m} \frac{(k-m)!}{(k+m)!} \,
    P^{[m]}_{k}(x) P^{[m]}_{k}(y).
\end{align}
It then follows that 
\begin{equation}\label{KL2}
    A^{[m]} 
    = \sum_{k\ge m} \frac{(k-m)!}{(k+m)!}\, \xi_k^2.
\end{equation}

We now justify all these expansions.
\begin{thm}[Expansions for $B^{[m]}(x)$, $K^{[m]}$ and $A^{[m]}$] 
\label{thm:KL} If $m=1$, then the KL expansion \eqref{KL} holds
uniformly in $x\in[\ve,1-\ve]$ with probability $1$ for any $\ve>0$,
the eigenfunction expansion \eqref{Kmxy} holds uniformly in
$x,y\in[\ve,1-\ve]$, and \eqref{KL2} converges with probability 1.

If $m\ge 2$, then \eqref{KL} holds uniformly in $x\in[0,1]$ with
probability 1, \eqref{Kmxy} holds uniformly in $x,y\in[0,1]$, and
\eqref{KL2} converges with probability 1.
\end{thm}
\begin{proof}
The results for $m=1$, namely, the original Anderson-Darling 
statistic, are well-known; 
see, for example, Chapter 5 of \cite{Shorack86} and \cite{Pycke03}.

Assume now $m\ge 2$. Since $B^{[m]}(x)\in C^{m-1}[0,1]$ and 
\[
    \frac{\text{d}^k}{\text{d}x^k} B^{[m]}(x)\Bigl|_{x=0}
    = \frac{\text{d}^k}{\text{d}x^k} B^{[m]}(x)\Bigl|_{x=1}
    = 0 \qquad (0\le k<m),
\]
we define $\frac{B^{[m]}(x)}{(x(1-x))^{m/2}}$ at the end points $x=0$
and $1$ to be $0$, so that $\frac{B^{[m]}(x)} {(x(1-x))^{m/2}}$
becomes a continuous function in $x$ on $[0,1]$. Also the covariance
function $K^{[m]}(x,y)$ of the GAD process $\frac{B^{[m]}(x)} 
{(x(1-x))^{m/2}}$ is continuous on $[0,1]\times[0,1]$ by using
\eqref{Kmxy2} of Theorem \ref{thm:cov}.

Now, all required results will follow from standard arguments of 
orthogonal decompositions of processes \cite[Ch.\ 5, Sec.\ 
2]{Shorack86}. First, the kernel $K^{[m]}(x,y)$ is symmetric 
and positive-definite, and satisfies 
\[
    \int_0^1\int_0^1 K^{[m]}(x,y)^2 \dd x \dd y <\infty.
\]
Then, by Lemma \ref{lm:eigeneq} and Mercer's theorem, (\ref{Kmxy})
holds uniformly in $x,y\in[0,1]$. On the other hand, by Kac and 
Siegert's theorem, the KL expansion \eqref{KL} also holds for each 
$x\in[0,1]$ in the mean-square convergence:
\[
    \mathbb{E}\Biggl(\frac{B^{[m]}(x)}{(x(1-x))^{m/2}} 
    - \sum_{m\le k\le M} \sqrt{\frac{(k-m)!}{(k+m)!}}
    \,P^{[m]}_{k}(x) \xi_k\Biggr)^2 \to 0\quad (M\to\infty).
\]
Moreover, \eqref{KL2} holds with probability 1 (see Theorem 2 of
\cite{Shorack86}, pp.\,210--211, for the details). The uniform
and almost sure convergence in \eqref{KL} is a consequence of
Theorems of 3.7 and 3.8 of \cite{Adler90}; see also \cite{Pycke03}.
\end{proof}

\subsection{Finite sample expansions and statistical power}
\label{subsec:finite}

The KL expansion \eqref{KL} and the eigenfunction expansions
\eqref{KL2} describe the limiting distribution when the sample size
$N$ goes to infinity. In this section we obtain the counterparts of 
these expansions when $N$ is finite.
This is obtained by replacing 
$B(x)$ by the ``sample Brownian bridge'' $B_{N}(x) =\sqrt{N} 
(F_{N}(x)-x)$.

First note that by direct calculations, $\tau_m(t;x)$ in 
(\ref{ad-tau_mtx}) satisfies 
\[
    \int_0^1 \tau_m(t;x) P_k^{(m)}(x) \dd x 
    = (-1)^m P_k(t),
\]
from which, together with the completeness of $P_k^{[m]}(x)$ and
\eqref{Pkm}, we see that (\ref{tau-exp}) holds for all $t$ in
$L^2(\dd x|_{[0,1]})$. The integration with respect to the measure
$\dd B_{N}(x)$ is nothing but a finite sum. Thus, substituting
$t:=X_i$ into (\ref{tau-exp}), and summing up for $1\le i\le N$,
we have the sample version of the expansion
\[
    \frac{B^{[m]}_{N}(x)}{(x(1-x))^{m/2}}
    = \sum_{k\ge m} \sqrt{\frac{(k-m)!}{(k+m)!}}\,
    P^{[m]}_{k}(x)\,\widehat\xi_k,
\]
where
\[
    \widehat\xi_k
    = \int_0^1 P_k(x) \dd B_{N}(x)
    = \sfrac{1}{\sqrt{N}} \sum_{1\le i\le N} P_k(X_i).
\]
This holds for each $(X_i)_{1\le i\le N}\in [0,1]^N$ in $L^2(\dd
x|_{[0,1]})$. For the generalized Anderson-Darling statistic
\eqref{Am}, we then have
\[
    A^{[m]}_{N} 
    = \sum_{k\ge m} \frac{(k-m)!}{(k+m)!} \,\widehat\xi_k^2
\]
for each $(X_i)_{1\le i\le N}\in [0,1]^N$. In particular,
$A^{[1]}_{N}$, $A^{[2]}_{N}$ and $A^{[3]}_{N}$, when expressed in terms of the
$\widehat\xi_k$'s, have the forms
\begin{align*}
    A^{[1]}_{N}
    &= \sum_{k\ge 1} \frac{\widehat\xi_k^2}{k(k+1)} 
    = \frac{\widehat\xi_1^2}{2} + \frac{\widehat\xi_2^2}{6}
    + \frac{\widehat\xi_3^2}{12}
    +\cdots, \\
    A^{[2]}_{N}
    &= \sum_{k\ge 2} \frac{\widehat\xi_k^2 }{(k-1)k(k+1)(k+2)} 
    = \frac{\widehat\xi_2^2}{24} 
    + \frac{\widehat\xi_3^2}{120} + \cdots, \\
    A^{[3]}_{N}
    &= \sum_{k\ge 3} \frac{\widehat\xi_k^2 }{(k-2)(k-1)k(k+1)(k+2)(k+3)} 
    = \frac{\widehat\xi_3^2}{720} +  \frac{\widehat\xi_4^2}{5040} + \cdots,
\end{align*}
where
$\widehat\xi_1 = 2\sqrt{3 N} m_1$, $\widehat\xi_2 = 6 \sqrt{5 N} \lpa{m_2 -\frac1{12}}$,
$\widehat\xi_3 = 20 \sqrt{7 N} \lpa{m_3 -\frac{3}{20}m_1}$
with $m_k = \frac{1}{N}
\sum_{1\le i\le N} \lpa{X_i-\frac12}^k$, the sample $k$th moment around $\frac12$.
Since $\widehat\xi_1$ and $\widehat\xi_2$ are 
linear functions of the sample mean and the sample variance,
respectively, they convey the information of the mean
and dispersion of the distribution, respectively.  Thus, 
one expects that $A^{[1]}_{N}$
whose dominant component is $\widehat\xi_1^2$ has
statistical power especially for mean-shift alternative, and
$A^{[2]}_{N}$ whose dominant component is $\widehat\xi_2^2$
has statistical power for dispersion-shift alternative.
Similarly, since $\widehat\xi_3$ contains the information of the skewness,
$A^{[3]}_{N}$ is expected to detect the skewness (or asymmetry)
of the distribution.

Note that when
$N$ is large and the distribution $F$ of $X_i$ is not too much away
from the null hypothesis $\mathrm{Unif}(0,1)$, $\widehat\xi_k^2$ are 
approximately distributed as the independent chi-square distributions with one degree of freedom and
the non-centrality parameter
$\mathbb{E}\bigl[\widehat\xi_k\bigr]^2$ with
$\mathbb{E}\bigl[\widehat\xi_k\bigr]=\sqrt{N}\bigl(\int_0^1 P_k(x) \dd F(x)\bigr)$.
For example, 
let $Y_i$ be i.i.d.\ Gaussian random variables  with
mean $\Delta\mu/\sqrt{N}$ and
variance $1+\Delta\sigma^2/\sqrt{N}$, 
and consider testing the null hypothesis that
the true distribution is $N(0,1)$. 
Then, $F$ is the distribution of $X_i=\Phi(Y_i)$, where $\Phi$ the distribution function of $N(0,1)$.
By simple calculations, 
\begin{equation}
\label{contiguous}
 \mathbb{E}\bigl[\widehat\xi_1\bigr] \approx
 0.977\Delta\mu, \qquad 
 \mathbb{E}\bigl[\widehat\xi_2\bigr] \approx
 0.616\Delta\sigma^2, \qquad 
 \mathbb{E}\bigl[\widehat\xi_3\bigr] \approx
 0.183\Delta\mu 
\end{equation}
up to an error of $O(1)$.
The local property (\ref{contiguous}) is consistent with the interpretations of $\widehat\xi_k$ stated above.

The decomposition of goodness-of-fit statistics is also discussed in
\cite{Durbin72} based on the Cram\'er-von Mises statistic. While the
eigenfunctions for the Cram\'er-von Mises statistic are trigonometric
functions and different from our associate Legendre polynomials, the
interpretation of our components $\widehat\xi_k^2$ is similar to that
given in \cite{Durbin72}.

In Figures \ref{fig:normal1}--\ref{fig:skew}, the powers of
the test statistics $A_N^{[1]}$, $A_N^{[2]}$ and $A_N^{[3]}$
are compared based on Monte Carlo simulations.
We use the null hypothesis $H_0$ that $F$ is the normal distribution
with zero mean and unit variance, and assume that $N=100$ samples
are available.
The size of the tests is 0.05, and the critical values were
estimated by simulations.  The number of replications is 100,000.

First, we assume that the true distribution is the normal
distribution with mean $\mu$ and variance $\sigma^2$.
The estimated power functions of $A_N^{[1]}$, $A_N^{[2]}$
and $A_N^{[3]}$ are summarized in Figures \ref{fig:normal1}
and \ref{fig:normal2}, where 
the power functions are plotted
as functions of the mean $\mu$ 
(Figure \ref{fig:normal1}) and of the standard deviation $\sigma$
(Figure \ref{fig:normal2}), respectively.

One sees that $A_N^{[1]}$ performs the best in many cases.
$A_N^{[2]}$ has more power that $A_N^{[1]}$ when $\mu$ is
close to 0 and $\sigma$ is away from 1, while $A_N^{[1]}$ has
more power than $A_N^{[2]}$ when $\mu$ is not close to
0 or $\sigma$ is not away from 1.  As suggested in
(\ref{contiguous}), $A_N^{[3]}$ has less power than
$A_N^{[1]}$ and $A_N^{[2]}$ in any configuration.

\newcommand\blueredgreen{blue circle: $m=1$, red square: $m=2$, green triangle: $m=3$}

\begin{figure}
\begin{center}
\begin{tabular}{cc}
\includegraphics[scale=0.25]{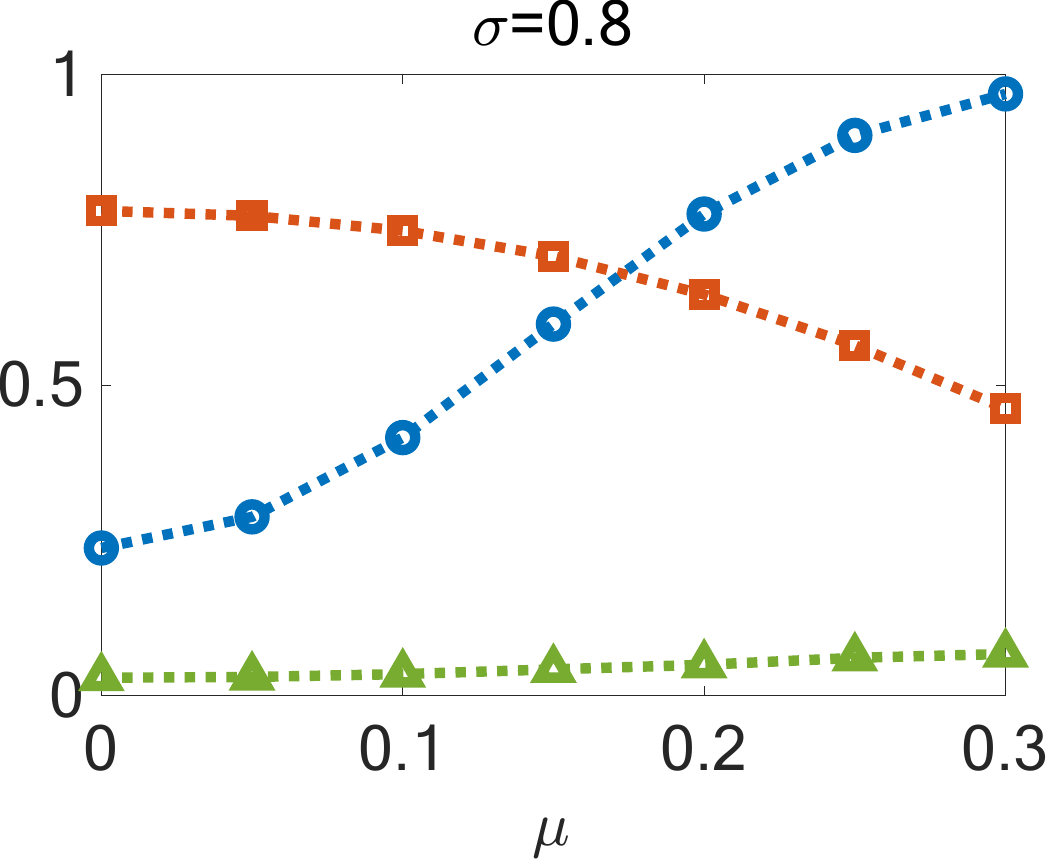} &
\includegraphics[scale=0.25]{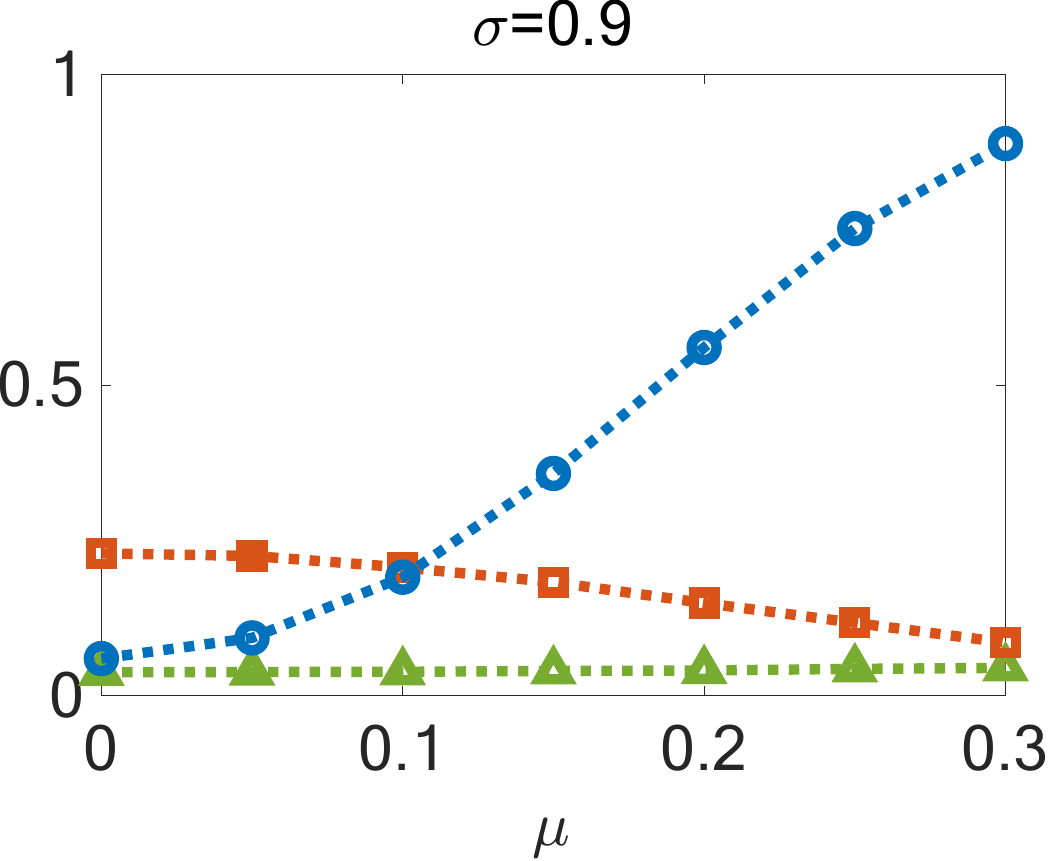} \\
\multicolumn{2}{c}{\includegraphics[scale=0.25]{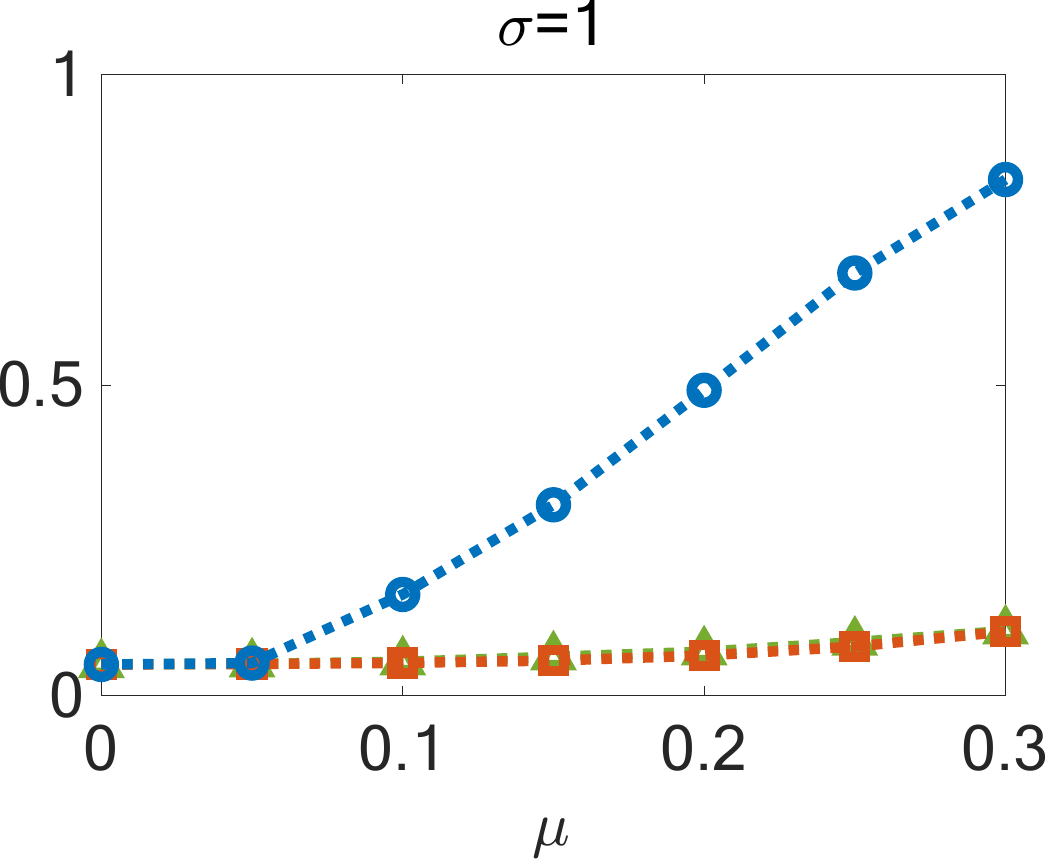}} \\
\includegraphics[scale=0.25]{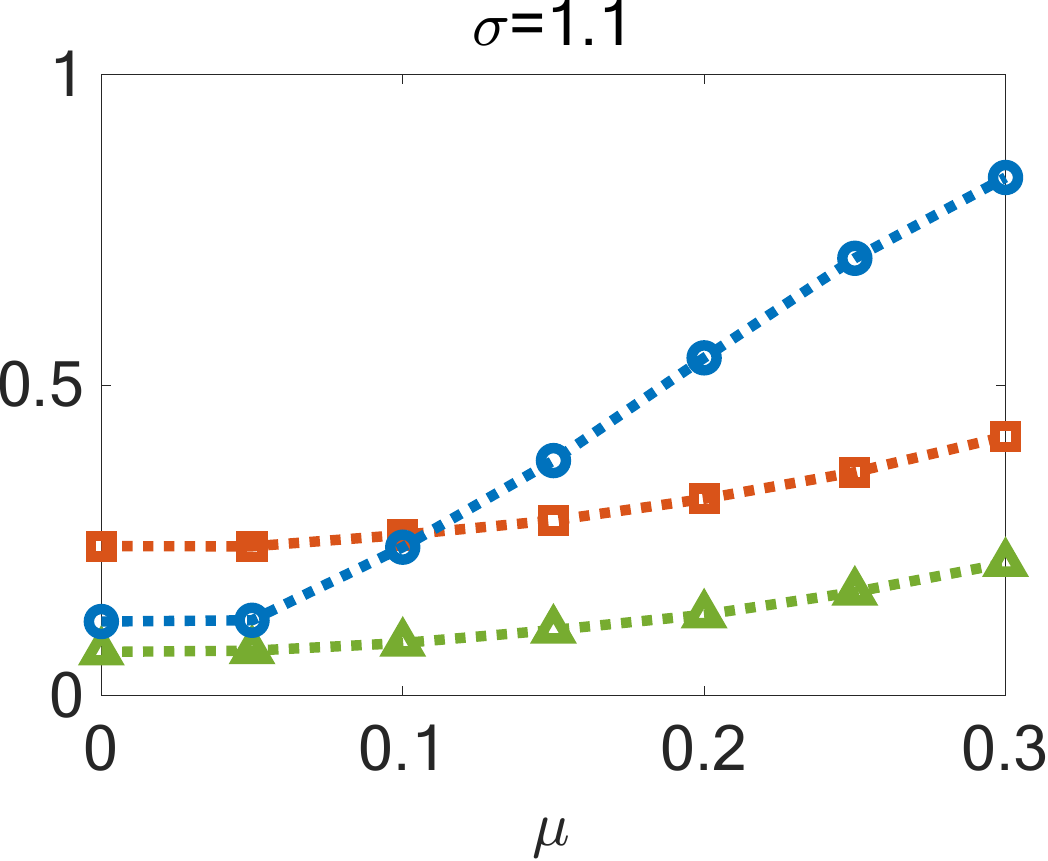} &
\includegraphics[scale=0.25]{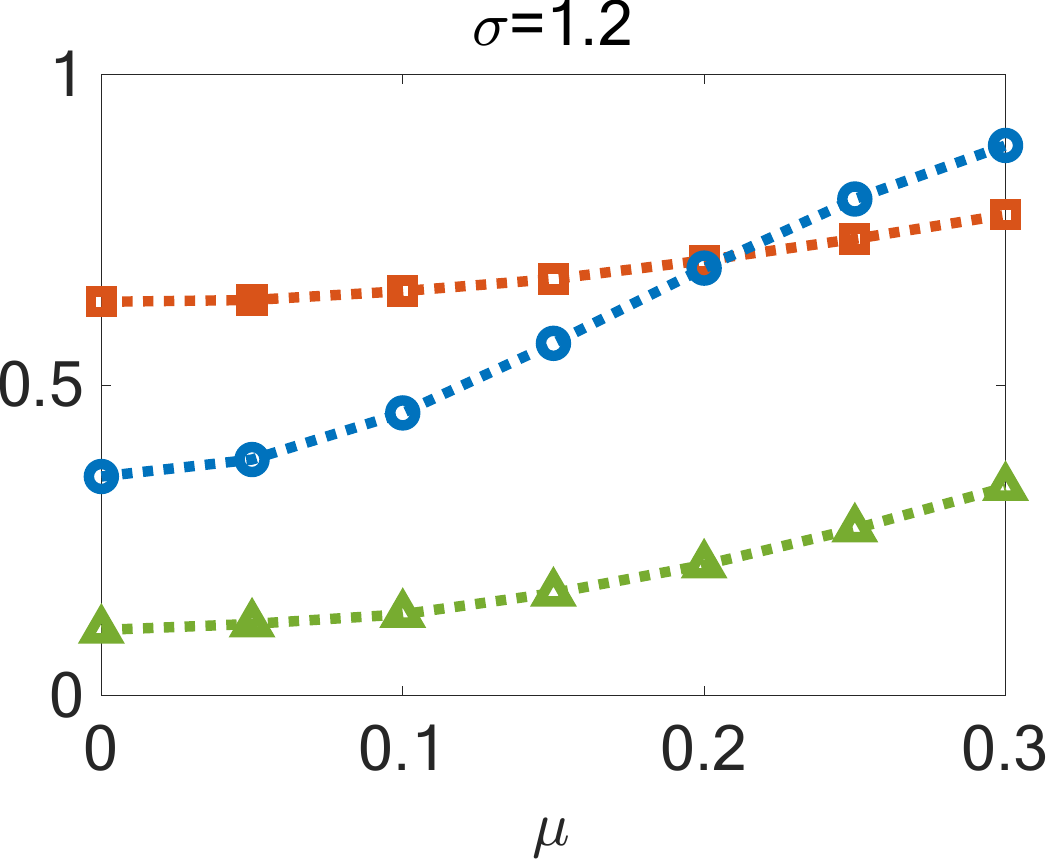}
\end{tabular}
\caption{Powers of $A^{[m]}_N$ for $\mu\in[0,0.3]$ and $\sigma\in\{0.8,0.9,1,1.1,1.2\}$.}
\smallskip (\blueredgreen)
\label{fig:normal1}
\end{center}
\end{figure}

\begin{figure}
\begin{center}
\begin{tabular}{cc}
\includegraphics[scale=0.25]{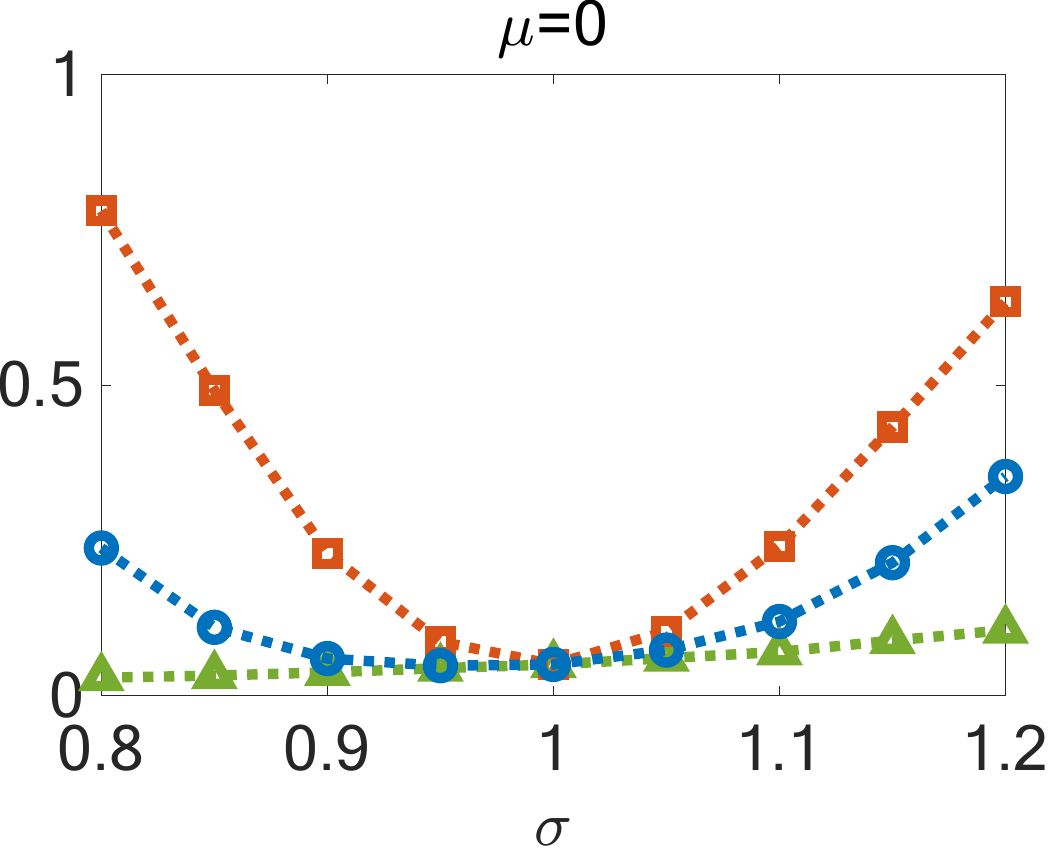} &
\includegraphics[scale=0.25]{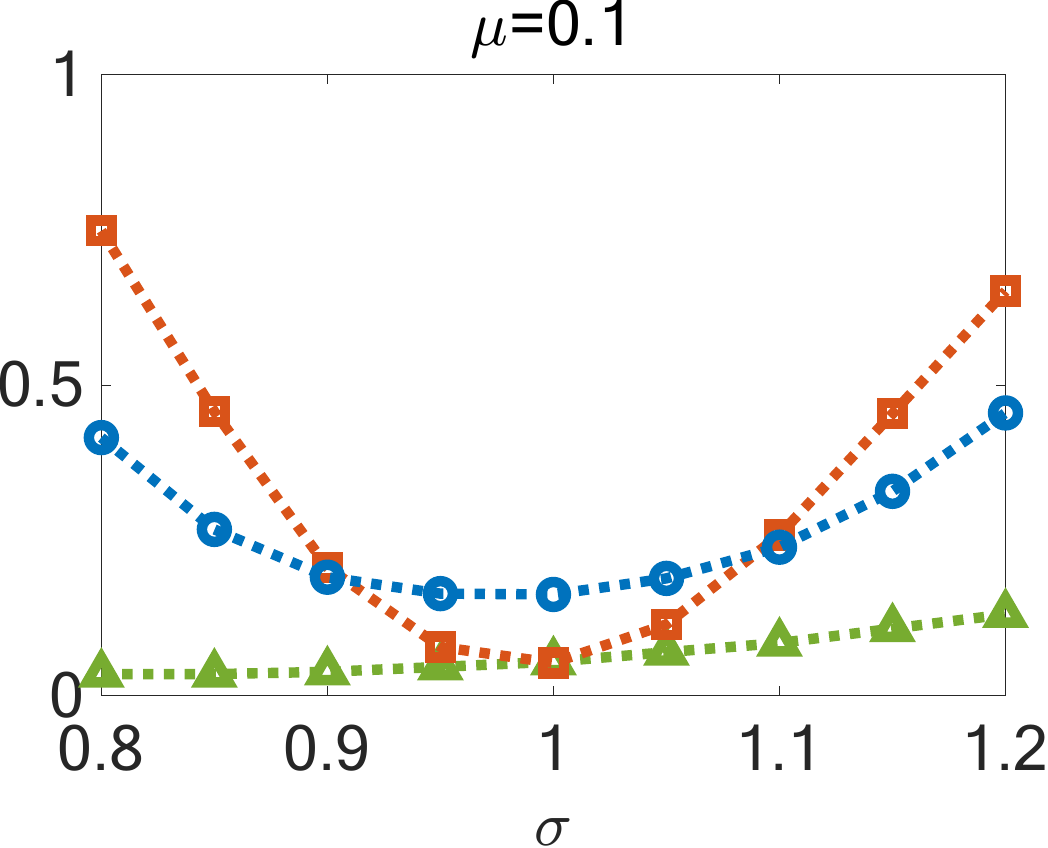} \\
\includegraphics[scale=0.25]{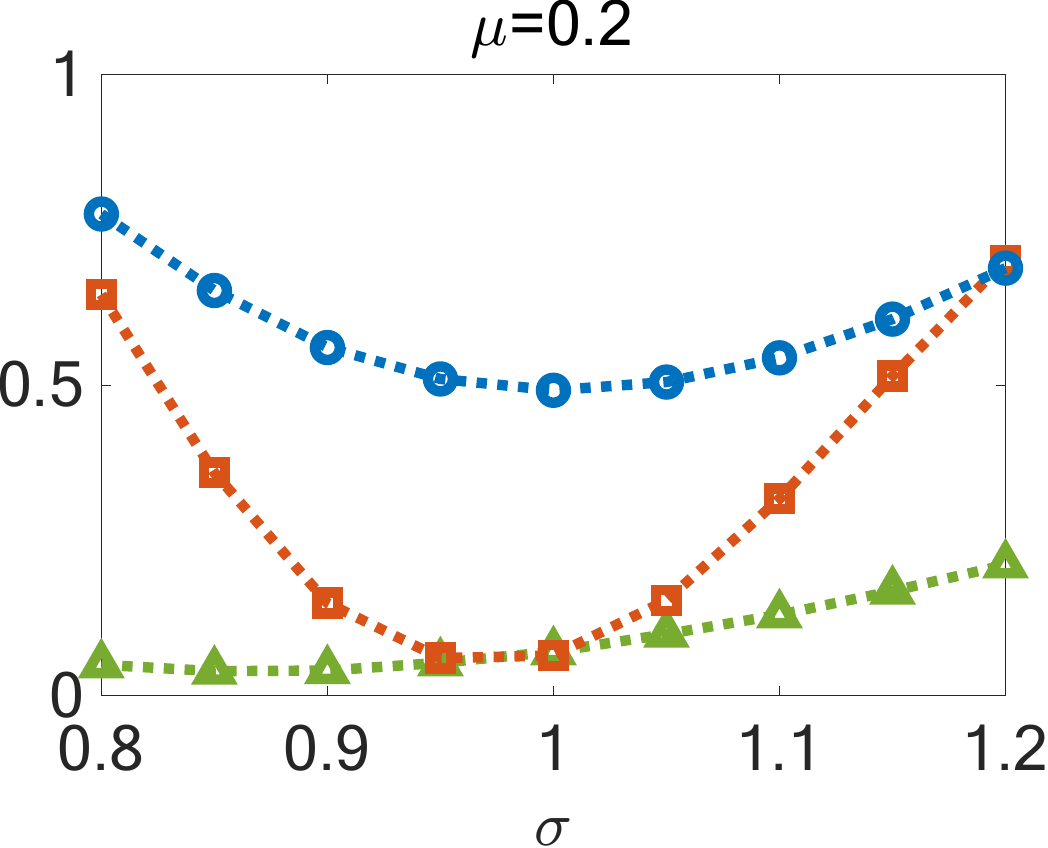} &
\includegraphics[scale=0.25]{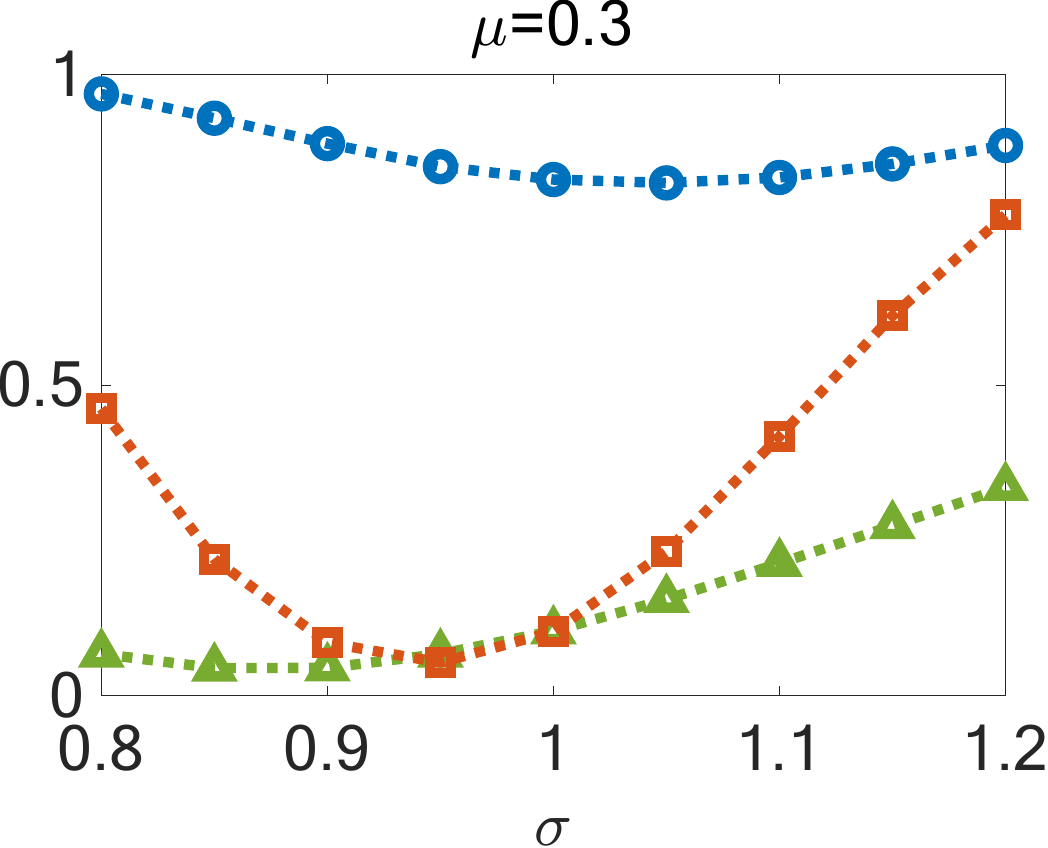}
\end{tabular}
\caption{Powers of $A^{[m]}_N$ for $\sigma\in[0.8,1.2]$ and $\mu\in\{0,0.1,0.2,0.3\}$.}
\smallskip
(\blueredgreen)
\label{fig:normal2}
\end{center}
\end{figure}

Second, we assume the skew-normal distribution with density
\begin{equation}
\label{skew-normal}
 \frac{2}{\omega}\phi\Bigl(\frac{x-\xi}{\omega}\Bigr)\Phi\Bigl(\alpha\frac{x-\xi}{\omega}\Bigr)
\end{equation}
as alternatives \citep{Azzalini85}, where $\phi$ and $\Phi$ are the probability
density function and the cumulative distribution function of
the standard normal distribution, respectively.
The mean and the variance of the skew-normal distribution are
\[
  \mu = \xi+\omega\delta\sqrt{\frac{2}{\pi}} \quad\mbox{and}\quad
  \sigma^2 = \omega^2\Bigl(1-\frac{2\delta^2}{\pi}\Bigr) \quad\mbox{with}\quad
  \delta=\frac{\alpha}{\sqrt{1+\alpha^2}},
\]
respectively.
The null hypothesis (the standard normal distribution) corresponds
to the case $\xi=0$, $\omega=1$, and $\alpha=0$.
For arbitrary $\alpha$, $(\mu,\sigma)=(0,1)$ is equivalent to
\[
    (\xi,\omega)=
    \llpa{-\omega\delta\sqrt{\frac2\pi},
    \frac1{\sqrt{1-2\delta^2/\pi}}}.
\]
The powers of $A_N^{[1]}$, $A_N^{[2]}$ and $A_N^{[3]}$ are
summarized in Figure \ref{fig:skew}.  Note that for skewed normal distribution the power
functions $(\alpha,\mu)$ and $(-\alpha,-\mu)$ are identical.

From the figures, we see that, when $\sigma=1$ and the absolute value of
$\mu$ is large, $A_N^{[1]}$ is superior to $A_N^{[2]}$, and
when $\mu=0$ and $\sigma$ is away from 1, $A_N^{[2]}$ is
superior to $A_N^{[1]}$.  This is similar to what has already
been observed in Figures \ref{fig:normal1} and \ref{fig:normal2},
and is consistent with the findings in \cite{Durio16}, where they 
showed that, even after adjusting the mean and the variance,
integrated statistics have considerable statistical power.

Moreover, when $|\alpha|$ increases, $A_N^{[3]}$ begins to
perform better, and can detect $\alpha\ne 0$ even
when $(\mu,\sigma)=(0,1)$.  Indeed, when
$(\mu,\sigma)\approx(0,1)$, $A_N^{[3]}$ has considerable power
when compared with $A_N^{[1]}$ and $A_N^{[2]}$.

\begin{figure}
\begin{center}
\begin{tabular}{ccc}
\includegraphics[scale=0.21]{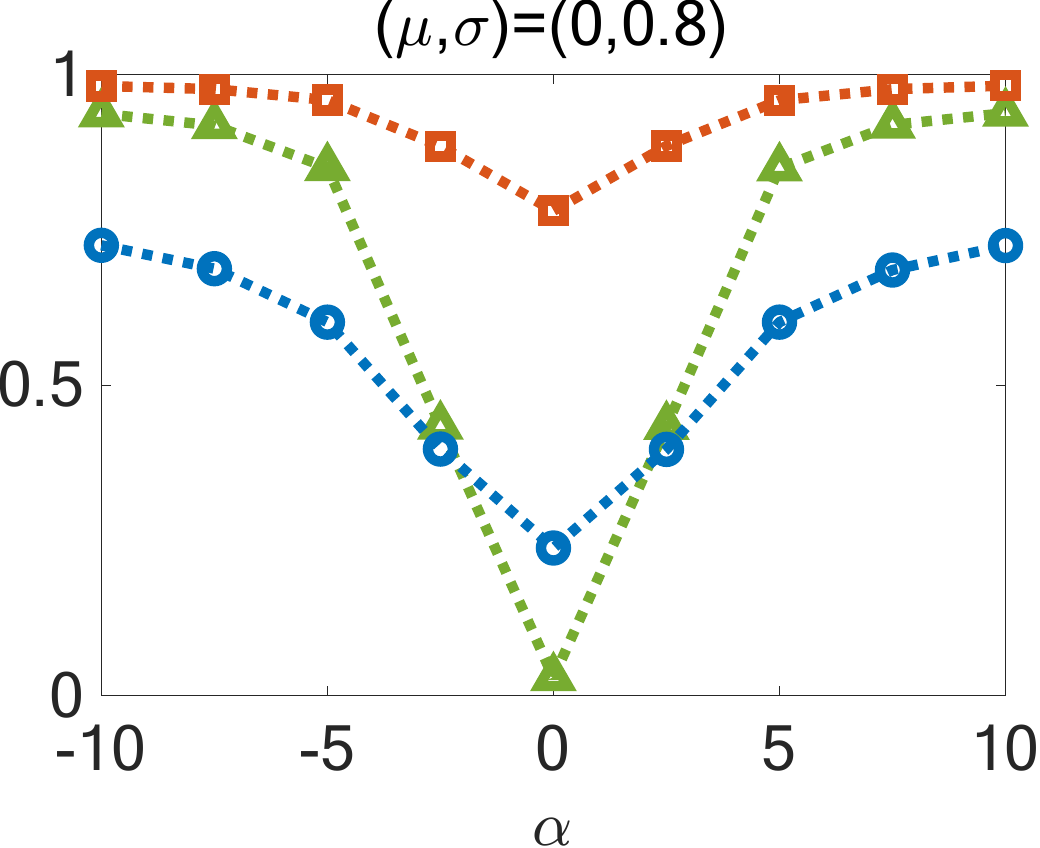} &
\includegraphics[scale=0.21]{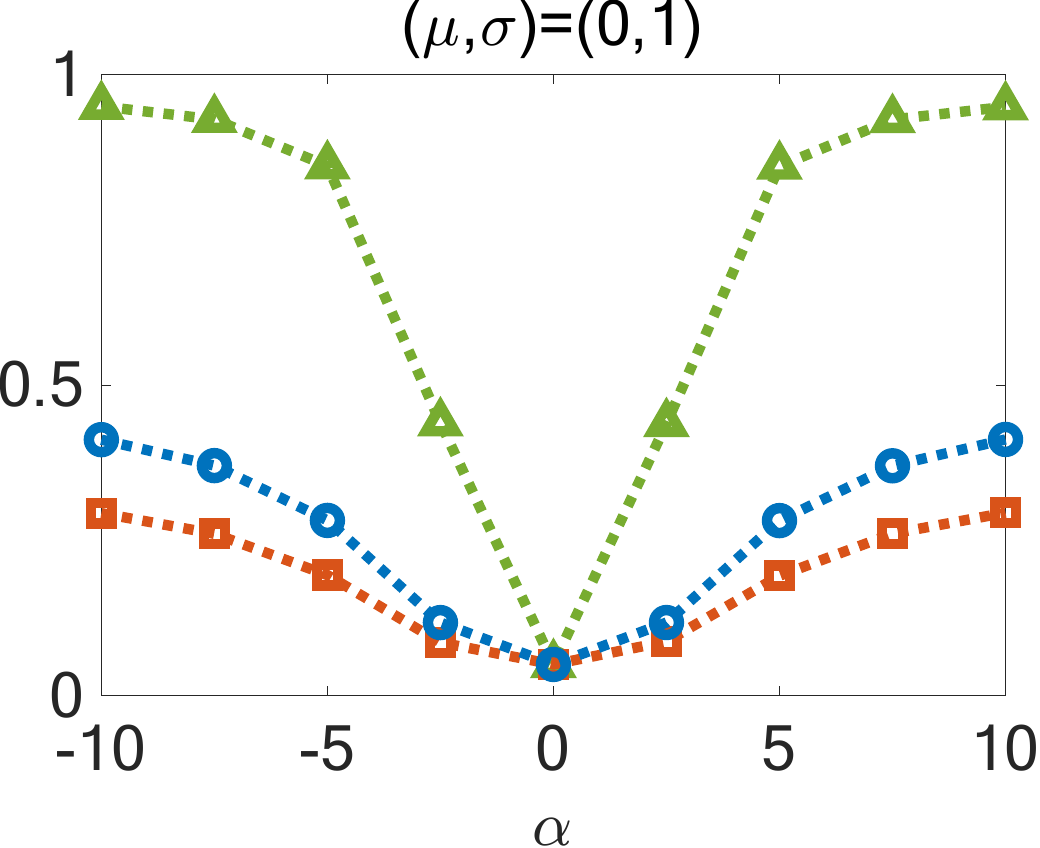} &
\includegraphics[scale=0.21]{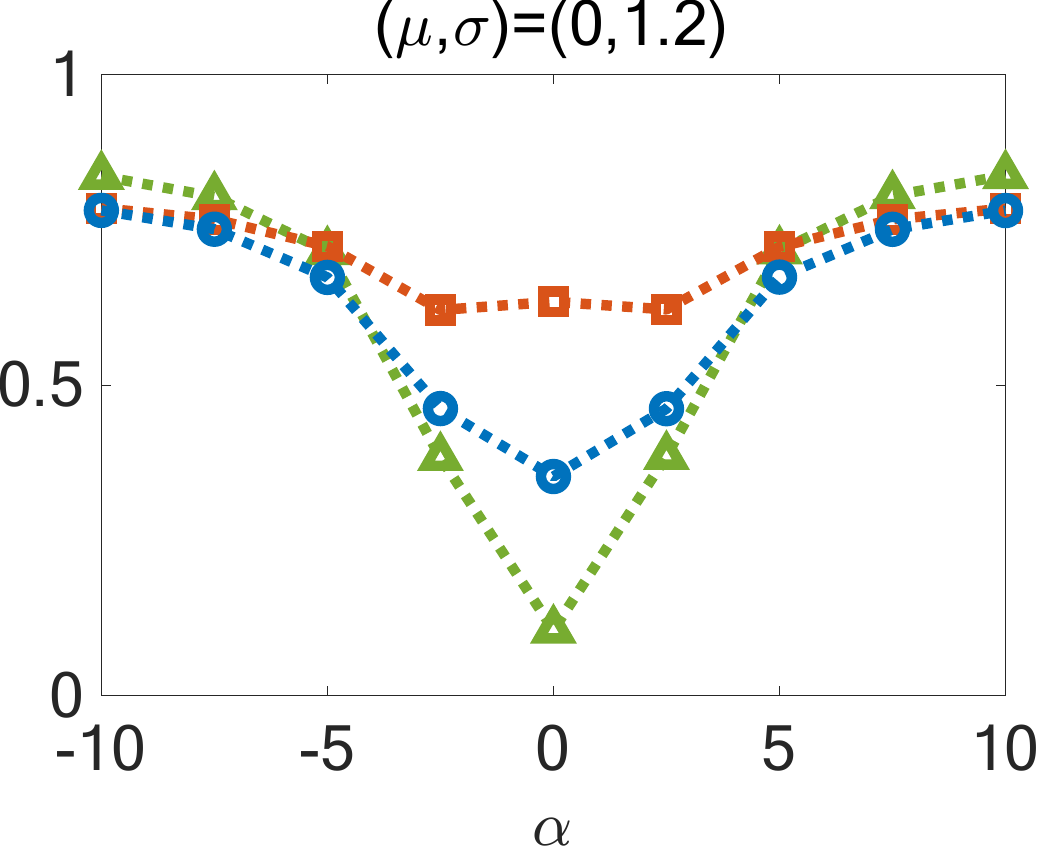} \\
\includegraphics[scale=0.21]{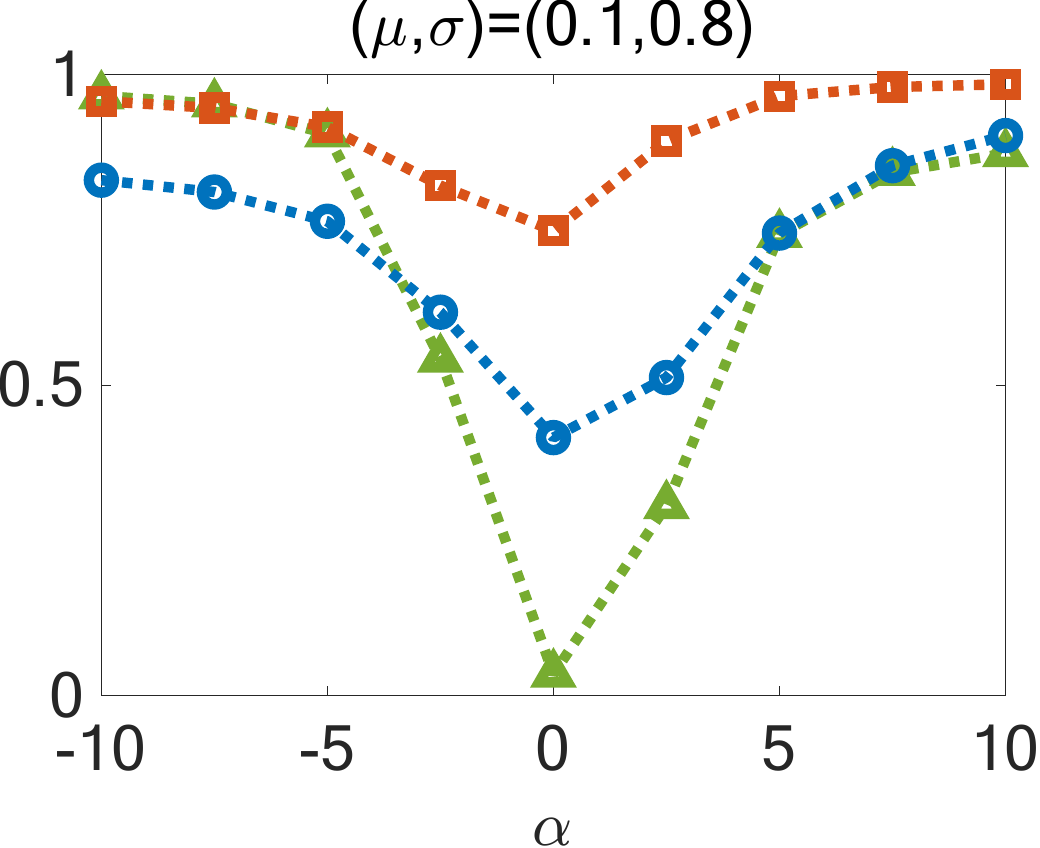} &
\includegraphics[scale=0.21]{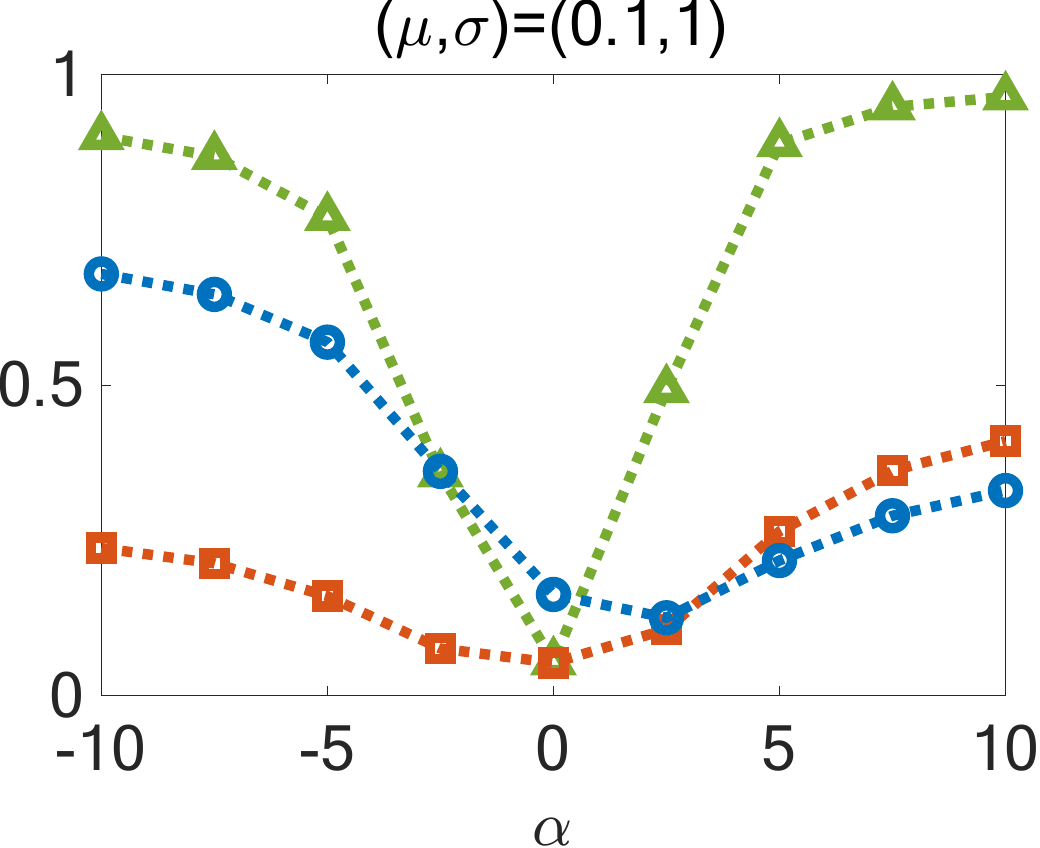} &
\includegraphics[scale=0.21]{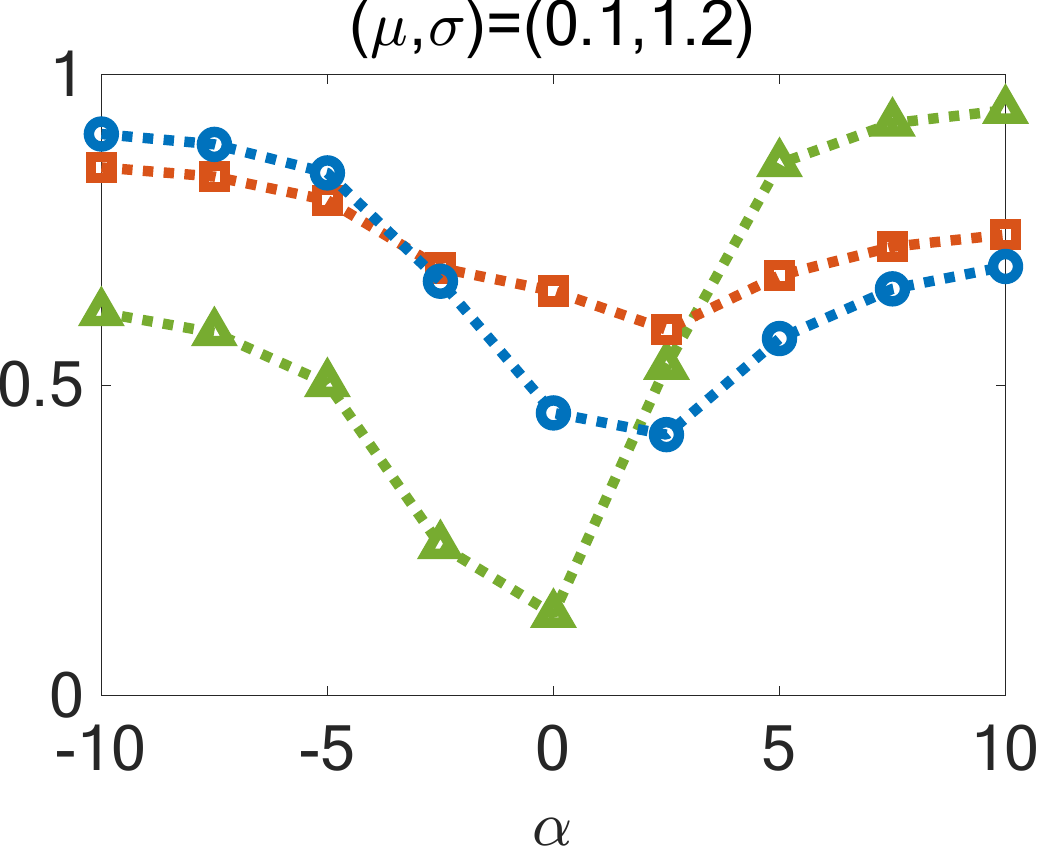} \\
\includegraphics[scale=0.21]{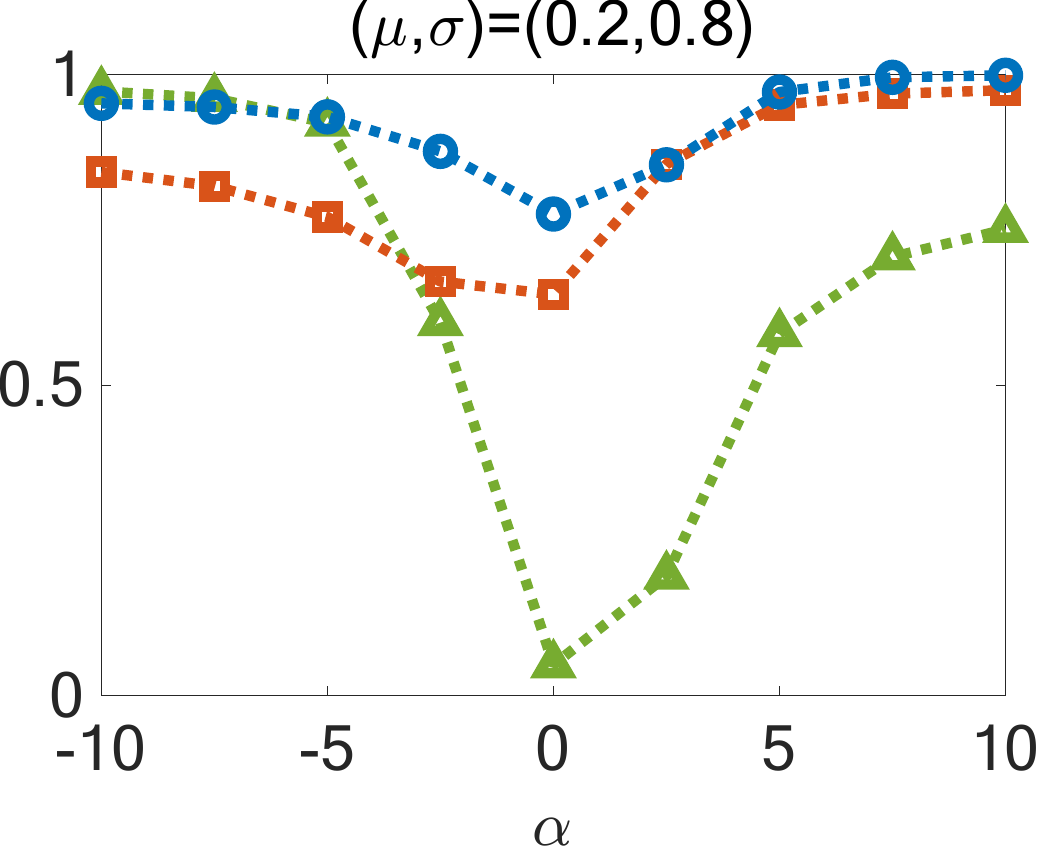} &
\includegraphics[scale=0.21]{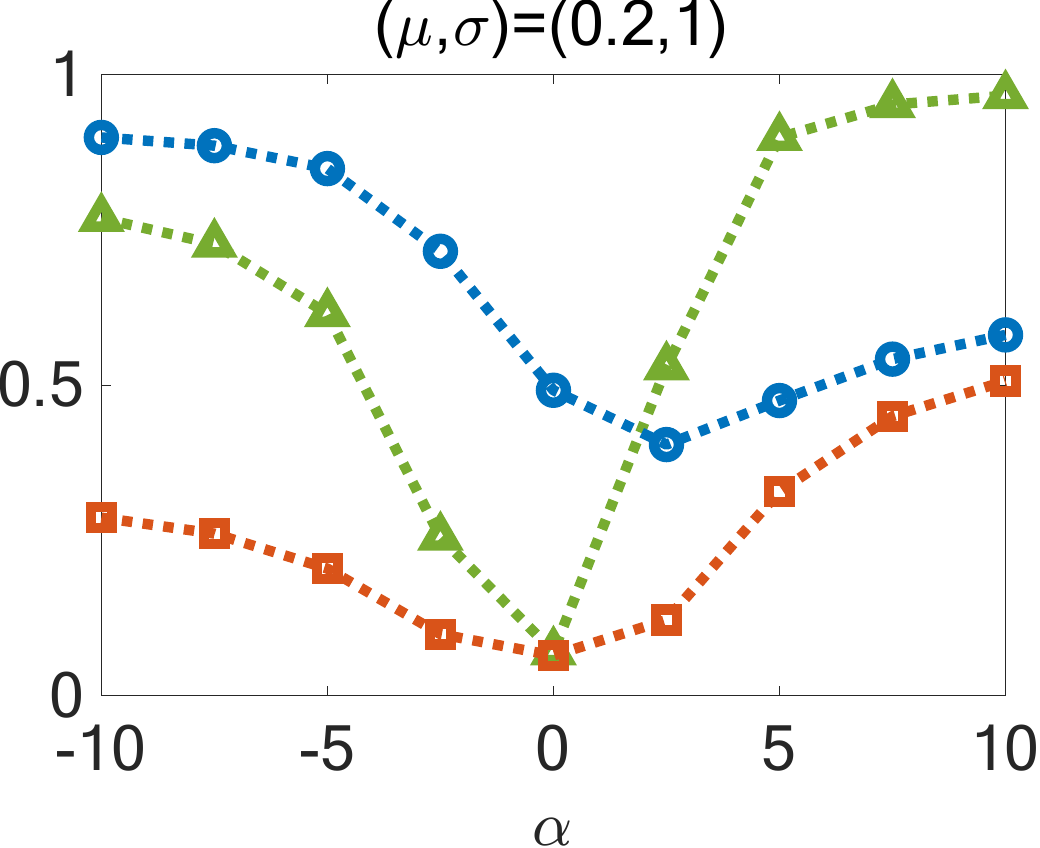} &
\includegraphics[scale=0.21]{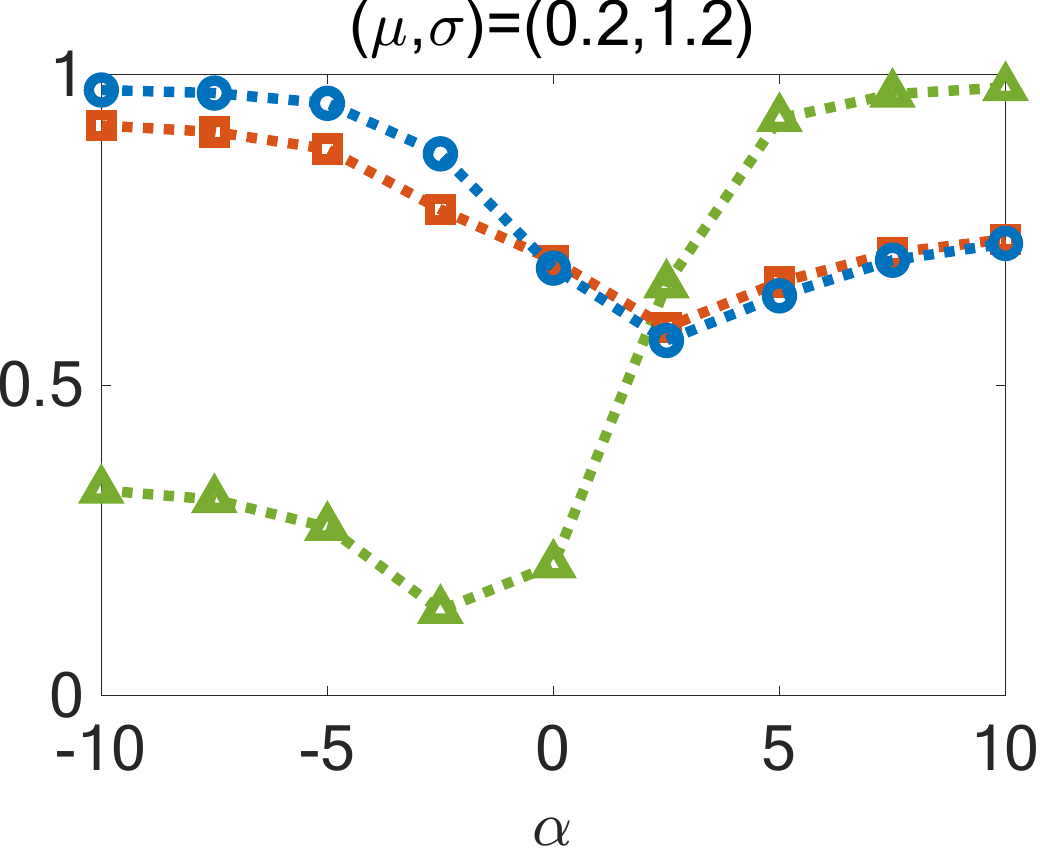}
\end{tabular}
\caption{Powers of $A^{[m]}_N$ under skew-normal distributions
with index $\alpha$, $\mu\in\{0,0.1,0.2\}$ and $\sigma\in\{0.8,1,1.2\}$.}
\smallskip
(\blueredgreen)
\label{fig:skew}
\end{center}
\end{figure}

\subsection{Two-sample test statistic}
\label{subsec:two-sample}

For the original Anderson-Darling statistic, \cite{Pettitt76}
proposed a two-sample statistic for testing equality of two empirical
distributions. Given two distribution functions $F^{(\ell)}$,
$\ell=1,2$, assume that two i.i.d.\ samples $X^{(\ell)}_1,\ldots,
X^{(\ell)}_{N_\ell}$ from $F^{(\ell)}$ are available. The merged 
dataset of $X^{(\ell)}_j$'s is denoted by $Y_1,\ldots,Y_N$, 
$N=N_1+N_2$. Let
\[
    R^{(\ell)}_j 
    = \sharp\left(i \mid Y_i\le X^{(\ell)}_j \right)
\]
be the rank of $X^{(\ell)}_j$ in the whole dataset $Y_1,\ldots,Y_N$.

Then, the $m$-fold integrated two-sample Anderson-Darling statistic 
for testing $H_0: F^{(1)}=F^{(2)}$ against $H_1: F^{(1)}\ne
F^{(2)}$ is given by
\begin{equation}\label{two-sample}
    \widetilde A^{\,\,[m]}_{N_1,N_2}
    = \frac{N_1 N_2}{N^2} \sum_{1\le i<N}
    \frac{\lpa{B^{(1)}_i-B^{(2)}_i}^2}
    {\left( \frac iN\lpa{1-\frac iN} \right)^m},
\end{equation}
where
\begin{align*}
    B^{(\ell)}_i
    &= \frac{1}{N_\ell}\sum_{1\le j\le N_\ell} 
    \frac{\lpa{\frac iN-\frac1NR^{(\ell)}_j}^{m-1}}{(m-1)!} 
    \Ind\lpa{R^{(\ell)}_j\le i} \\
    &\qquad- \frac{1}{N_\ell}
    \sum_{0\le k<m}P_k^{(-m)}\lpa{\tfrac iN}
    \sum_{1\le j\le N_\ell} P_k\Lpa{\tfrac{R^{(\ell)}_j}N}.
\end{align*}

To prove \eqref{two-sample}, note first that the $m$-fold GAD 
statistic for testing
\[
    H_0: F=F_0 \quad \mbox{vs.}\quad H_1: F\ne F_0
\]
without assuming that $F_0$ is the uniform distribution on $[0,1]$
is given as
\[
    A^{[m]}_{N} 
    = N \int_{-\infty}^\infty
    \frac{\lpa{\frac{1}{\sqrt{N}} B^{[m]}_{N}(x;F_0;F_N)}^2}
    {\left( F_0(x)(1-F_0(x))\right)^m}\dd F_0(x),
\]
where
\begin{align*}
    \frac{1}{\sqrt{N}}\, B^{[m]}_{N}(x;F_0;F_N)
    =& \int_{-\infty}^x 
    \frac{(F_0(x)-F_0(t))^{m-1}}{(m-1)!} \dd F_N(t) \\
    & - \sum_{0\le k<m}P_k^{(-m)}(F_0(x))
    \int_{-\infty}^\infty P_k(F_0(x)) \dd F_N(x).
\end{align*}

Let $F_{N_\ell}^{(\ell)}$ be the underlying empirical distribution 
function of $X^{(\ell)}_1,\ldots,X^{(\ell)}_{N_\ell}$ for $\ell=1,2$.
Denote by $F_N$ the empirical distribution function of the merged
data $Y_1,\ldots,Y_N$, $N=N_1+N_2$.
Obviously,
\[
    F_N(x) 
    = \frac{N_1 F^{(1)}_{N_1}(x) 
    + N_2 F^{(2)}_{N_2}(x)}{N}.
\]
The two-sample Anderson-Darling statistic is defined by
\[
    \widetilde A^{\,\,[m]}_{N_1,N_2}
    = \frac{N_1 N_2}{N} \int_{-\infty}^\infty
    \frac{\left(\frac{1}{\sqrt{N_1}} 
    B^{[m]}_{N_1}\lpa{x;F_N;F^{(1)}_{N_1}}-\frac{1}{\sqrt{N_2}}
    B^{[m]}_{N_2}\lpa{x;F_N;F^{(2)}_{N_2}}\right)^2}{\left( 
    F_N(x)(1-F_N(x))\right)^m}\dd F_N(x).
\]
Since $F_N(X^{(\ell)}_j) = \frac{1}{N} R^{(\ell)}_j$, a simple 
calculation then yields the result \eqref{two-sample}.

As \cite{Pettitt76} proved, the original Anderson-Darling
statistic has the same limiting null distribution as
the one-sided case. This holds for the general $m$ as well, but
the details are omitted here.

\subsection{The MGF of the limit law $A^{[m]}$}
\label{subsec:mgf-ad}

By the expansion \eqref{KL2}, we see that the MGF of the limit law 
$A^{[m]}$ has the infinite product representation
\[
    \mathbb{E}\lpa{e^{s A^{[m]}}}
    = \prod_{k\ge1} \frac1{\sqrt{1-\frac{2s}{\lambda_k}}}
    \quad\text{where}\quad
    \lambda_k := k(k+1)\cdots(k+2m+1).
\]
This expression holds for complex $s$ except for the branch-cut 
$[\frac12\lambda_1,\infty)$. It is known that 
\begin{align}\label{EesA1}
    \mathbb{E}\lpa{e^{sA^{[1]}}} 
    = \sqrt{\frac{-2\pi s}{\cos\frac{\pi}{2}\sqrt{1+8s}}}.
\end{align}
We derive a similar representation for all $m\ge1$.

\begin{thm}\label{thm:gad-cos}
Let $\eta_j$ ($0\le j<m$) represent the zeros of the 
polynomial equation (in $z$)
\begin{align}\label{poly-eq}
    \prod_{0\le j<m} \left(z-(j+\tfrac12)^2\right)
    -2s =0.
\end{align}
Then for $m\ge1$
\begin{align}\label{EesAm}
    \mathbb{E}\lpa{e^{sA^{[m]}}}
    = \sqrt{\frac{(-2\pi s)^{m}}{1!2!\cdots(2m-1)!}
    \prod_{0\le j<m}
    \frac1{\cos\lpa{\pi\sqrt{\eta_j^{}}}}},
\end{align}
where the square-root denotes the principal branch.
\end{thm}
The finite product representation \eqref{EesAm} is preferable for 
most analytic or numerical purposes.  For example, the distribution 
function of $A^{[m]}$ can be computed by the following formula
\begin{align}\label{SSF}
    \mathbb{P}\lpa{A^{[m]}>x} 
    = \sfrac1{\pi}\sum_{k\ge1} (-1)^{k-1}
    \int_{\frac12\lambda_{2k-1}}^{\frac12\lambda_{2k}} 
    \frac{e^{-xs}}
    {s\sqrt{\bigl|\mathbb{E}\lpa{e^{s A^{[m]}}}\bigr|}}\dd s,
\end{align}
which may be called \emph{the Smirnov-Slepian formula}; see 
\cite{Smirnov37,Slepian58} for the original papers where such 
a technique was introduced. See also \cite{Matsui08} and 
\cite{Deheuvels08} for other applications. In the above 
formula, we see that effective numerical procedures such as 
\eqref{EesAm} are crucial in computing the tails to high precision.

In particular, when $m=1$, the solution to $\eqref{poly-eq}$ is 
obviously given by $\eta_0 = \frac14+2s$, so we obtain \eqref{EesA1}.
When $m=2$, the solutions to \eqref{poly-eq} are 
$\frac54\pm\sqrt{1+2s}$, so we obtain the identity
\begin{equation}\label{adm2}
    \begin{split}
   	&\prod_{k\ge1}\frac1{\sqrt{1-\frac{2s}{k(k+1)(k+2)(k+3)}}}\\
    &\qquad= \frac{\pi s}{\sqrt{3\cos\left(\frac\pi2
    \sqrt{5-4\sqrt{1+2s}}\right)
    \cos\left(\frac\pi2
    \sqrt{5+4\sqrt{1+2s}}\right)}}.
    \end{split}
\end{equation}
A very different construction based on Jacobi polynomials leading
to an MGF of the same form but with $s\mapsto 6s$ can be found in 
\cite{Deheuvels08}. For $m=3$, 
\begin{equation}\label{adm3}
    \begin{split}
   	&\prod_{k\ge1}
    \frac1{\sqrt{1-\frac{2s}{k(k+1)(k+2)(k+3)(k+4)(k+5)}}}\\
    &\qquad= \frac{(\pi s)^{3/2}}{\sqrt{-4320\cos\left(\pi
    \sqrt{\eta_1}\right)
    \cosh\left(\pi\sqrt{\eta_2}\right)
    \cosh\left(\pi\sqrt{\eta_3}\right)}}
	\end{split}
\end{equation}
where $\eta := \sqrt[3]{27s+80+3\sqrt{81s^2+480s-1728})}$ and 
\begin{align*}
    \eta_1 &:= \frac1{12\eta}\left(4\eta^2+35\eta+112\right), \\
    \eta_2 &:= \frac{1}{12\eta}
    \left(4e^{-\pi i/3}\eta^2-35\eta+112e^{\pi i/3}\right), \\
    \eta_3 &:= \frac{1}{12\eta}
    \left(4e^{\pi i/3}\eta^2-35\eta+112e^{-\pi i/3}\right).
\end{align*}

While the exact expressions readily become lengthy for higher
values of $m$, numerical evaluation is rather straightforward
once the value of $s$ is specified.

\begin{lm}\label{thm:mgf-ad}
Let $\zeta_j(s)$, $0\le j<2m$, denote the zeros of the polynomials 
$z(z+1)\cdots(z+2m-1)-2s$: 
\begin{align}\label{ad-ie}
    z(z+1)\cdots(z+2m-1)-2s = \prod_{0\le j<2m}
    \left(z-\zeta_j(s)\right).
\end{align}  
Then
\begin{align}\label{gad-1st-prod}
    \mathbb{E}\lpa{e^{s A^{[m]}}}
    =\sqrt{\prod_{0\le j<2m} \frac{\Gamma(1-\zeta_j(s))}
    {j!}}.
\end{align}
\end{lm}
\begin{proof}
By \eqref{ad-ie}
\[
    1-\frac{2s}{k(k+1)\cdots(k+2m-1)} = \prod_{0\le j<2m}
    \frac{k-\zeta_{j}(s)}{k+j} = 
    \prod_{0\le j<2m}\frac{1-\frac{\zeta_{j}(s)}{k}}
    {1+\frac jk}.
\]
Note that 
\begin{align}\label{id}
    -\sum_{0\le j<2m}\zeta_{j}(s) 
    = [z^{2m-1}]\prod_{0\le j<2m}(z+j)
    = m(2m-1) = \sum_{0\le j<2m}j,
\end{align}
where $[z^m]f(z)$ denotes the coefficient of $z^m$ in the Taylor
expansion of $f(z)$. Using the identity \eqref{id} and
the infinite product representation of the Gamma function
\[
    \Gamma(1+s)e^{\gamma s} 
    = \prod_{k\ge1}\frac{e^{\frac sk}}{1+\frac sk},
\]
we then have
\begin{align*}
    \prod_{k\ge1}
    \frac1{1-\frac{2s}{k(k+1)\cdots(k+2m-1)}}
    &= \prod_{k\ge1}\prod_{0\le j<2m}
    \frac{\left(1+\frac jk\right)e^{-\frac jk}}
    {\lpa{1-\frac{\zeta_{j}(s)}k}
    e^{-\frac{\zeta_{j}(s)}k}} \\
    &= \prod_{0\le j<2m}
    \frac{\Gamma(1-\zeta_{j}(s)) 
    e^{-\gamma \zeta_{j}(s)}}{\Gamma(j+1)e^{j\gamma}} \\
    &= \prod_{0\le j<2m}
    \frac{\Gamma(1-\zeta_{j}(s))}{j!}.
\end{align*}
This proves \eqref{gad-1st-prod}.
\end{proof}

\begin{proof}[Proof of Theorem~\ref{thm:gad-cos}.]
Note that the roots of the equation \eqref{ad-ie} are symmetric
with respect to $-m+\frac12$, namely, if $-m+\frac12+\zeta$ is a
root of the equation \eqref{ad-ie}, then $-m+\frac12-\zeta$ is
also a root. Thus we can regroup the zeros in pairs and use 
the factorization
\begin{align*}
    &\Gamma\lpa{m+\tfrac12-\zeta}
    \Gamma\lpa{m+\tfrac12+\zeta} \\
    &\qquad= \Gamma\lpa{-m+\tfrac12-\zeta}
    \Gamma\lpa{m+\tfrac12+\zeta}
    \prod_{0\le \ell<2m} \left(-m+\tfrac12
    -\zeta +\ell\right),
\end{align*}
which equals, by Euler's reflection formula
\[
    \Gamma(s)\Gamma(1-s) = \frac{\pi}{\sin \pi s},
\]
and \eqref{ad-ie}, to 
\[
    \Gamma\lpa{m+\tfrac12-\zeta}
    \Gamma\lpa{m+\tfrac12+\zeta}
    = \frac{2\pi s}{\sin\lpa{\pi(m+\tfrac12+\zeta)}}
    = \frac{(-1)^m 2\pi s}{\cos(\pi\zeta)}.
\]
Now if we replace $x$ in the equation $x\cdots(x+2m-1)=2s$ by 
$x\mapsto-m+\tfrac12+z$, then equation \eqref{ad-ie} becomes 
\[
    \prod_{0\le j<m} \left(z^2-(-m+\tfrac12+j)^2\right)
    = 2s,
\]
which is equivalent to 
\[
    \prod_{0\le j<m} \left(z^2-(j+\tfrac12)^2\right)
    = 2s.
\]
Thus it suffices to consider the equation \eqref{poly-eq}
\[
    \prod_{0\le j<m} \left(z-(j+\tfrac12)^2\right)
    = 2s.
\]
It then follows that 
\[
    \prod_{0\le j<2m} \Gamma(1-\zeta_j(s))
    = \prod_{0\le j<m} \frac{(-1)^m 2\pi s}
    {\cos\lpa{\pi\sqrt{\eta_j}}}.
\] 
This proves \eqref{EesAm}.
\end{proof}

\section{The generalized Watson statistics}
\label{sec:w}

We follow the same approach developed above and examine a 
generalization of Watson's statistic by the same iterated empirical 
measures.

\subsection{A new class of Watson statistics}
\label{subsec:proposal-w}

Let $X_1,\ldots,X_N$ be an i.i.d.\ sequence on
$[0,1)=\{ \{x\} \,|\,x \in \mathbb{R} \}$, where $\{x\}= x \bmod 1$ denotes the fractional part of 
$x$. Here the $X_i$'s are referred to as ``circular data'' or 
``directional data''.
\cite{Watson61}'s statistic to test the 
uniformity of $X_i$'s on $[0,1)$ is defined as
\begin{equation}\label{UN}
    U_{N} 
    = \int_0^1 \overbar{B}_{N}(x)^2 \dd x, \quad \overbar{B}_{N}(x) 
    = B_{N}(x)-\int_0^1 B_{N}(x) \dd x,
\end{equation}
where $B_{N}(x) = \sqrt{N}(F_{N}(x)-x) = \sqrt{N} \int_0^1
(\Ind_{\{t\le x\}}-x) \dd F_{N}(t)$ as before.

The generalized Watson (GW) statistics we propose here are defined 
similarly as the GAD statistics. In particular, the correction term 
$-\int_0^1 B_{N}(x) \dd x$ in \eqref{UN} makes the statistic 
invariant with respect to the choice of the origin of the data 
$X_i\mapsto\{X_i+a\}$ ($a$ is a constant), and the reflection of the 
data $X_i\mapsto\{-X_i\}$; see Corollary~\ref{co:invariance}. Such a 
correction will also be incorporated in our generalized Watson 
statistics.

Let $\overbar{B}^{[1]}_{N}(x) = \overbar{B}_{N}(x)$ and 
\begin{align}
    \overbar{B}^{[m+1]}_{N}(x) 
    &= \int_0^x \overbar{B}^{[m]}_{N}(t) \dd t - \int_0^1 
    \int_0^x \overbar{B}^{[m]}_{N}(t) \dd t \dd x \nonumber \\
    &= \int_0^x \overbar{B}^{[m]}_{N}(t) \dd t 
    - \int_0^1 (1-t) \overbar{B}^{[m]}_{N}(t) \dd t \qquad(m\ge1).
\label{rec-C}
\end{align}
Then $\int_0^1 \overbar{B}^{[m]}_{N}(t) \dd t = 0$ and
$\overbar{B}^{[m]}_{N}(x)$ is an $m$-fold iterated integral of the
empirical measure $\dd F_{N}$ with mean zero. Thus we define the GW statistic as
\begin{align*}
    U^{[m]}_{N} := \int_0^1 \overbar{B}^{[m]}_{N}(x)^2 \dd x.
\end{align*}

\paragraph{Normalized Bernoulli polynomials.}
As in the case of the GAD statistic, $\overbar{B}^{[m]}_{N}(x)$ can be
expressed in terms of a class of template functions. For that 
purpose, we introduce the following polynomials. For $m\ge 1$, let 
$\b_{m}(y)$ be a polynomial in $y$ of degree $m$ satisfying
\begin{equation}\label{bernoulli}
    \b_{m}(y+1)
    =\b_{m}(y)+\frac{y^{m-1}}{(m-1)!}\quad \text{and}\quad
    \int_0^1 \b_{m}(y) \dd y 
    = 0.
\end{equation}
It turns out that the polynomial $m! \b_m(y)$ is identical to
Bernoulli polynomial of order $m$. Alternatively, $\b_m(y)$ can be
computed recursively by
\begin{equation}\label{recursive}
    \b_{m+1}(y)
    = \int_0^y \b_{m}(t) \dd t 
    - \int_0^1 (1-t) \b_{m}(t) \dd t,
\end{equation}
with $b_0(y)=1$. In particular, 
\begin{center}
  \begin{tabular}{ccccc} 
    $b_0(y)$ & $b_1(y)$ & $b_2(y)$ & $b_3(y)$ 
    & $b_4(y)$ \\ \hline
    1 & $y-\frac12$ & $\frac{y^2}2-\frac{y}2+\frac1{12}$ &
    $\frac{y^3}6-\frac{y^2}4+\frac{y}{12}$ &
    $\frac{y^4}{24}-\frac{y^3}{12}+\frac{y^2}{24}
    -\frac1{720}$ \\ 
  \end{tabular}
\end{center}
The generating function of $b_m(y)$ is
\begin{equation}\label{gen_func}
    V(z;x) = \sum_{m\ge 0}\b_{m}(x) z^m
    = \frac{z e^{xz}}{e^z-1} 
    \qquad (|z|<2\pi),
\end{equation}
and the Fourier series is given, for $m\ge1$, by 
\begin{align}
    \b_{m}(x)
    = -2 \sum_{k\ge 1}\frac{\cos\bigl(2k\pi x 
    -\frac{m\pi}{2}\bigr)}{(2k\pi)^{m}}\qquad (0<x<1);
\label{fourier}
\end{align}
see \cite[p.\,805]{Abramowitz64}. Another property we need is 
$b_m(1-x) = (-1)^mb_m(x)$.
Note that, from (\ref{bernoulli}) and (\ref{recursive}),
the generating function satisfies
\begin{equation}\label{rec-gen}
\left\{
\begin{aligned}
    & \int_0^1 V(z;y) \dd y = 1, \\
    & z^{-1}(V(z;y)-1) = \int_0^y V(z;t) \dd t - \int_0^1 (1-t) V(z;t) \dd t.
\end{aligned}
\right.
\end{equation}

\paragraph{The template functions.}
Let
\begin{align}
    \bar{\tau}_m(t;x)
    &= \frac{(x-t)^{m-1}}{(m-1)!}\Ind_{\{t\le x\}} 
    + (-1)^{m-1} \b_{m}(t-x)\nonumber \\
    &= (-1)^{m-1} \b_{m}(t-x+1) \Ind_{\{t\le x\}} 
    + (-1)^{m-1} \b_{m}(t-x) \Ind_{\{t>x\}}.
\label{tau_m-w}
\end{align}

The $\overbar{B}_{N}^{[m]}(x)$ are linked to the template functions 
by the following integral representation.
\begin{lm}\label{lm:hm-w} For $m\ge1$
\begin{equation}\label{CNm}
    \overbar{B}^{[m]}_{N}(x) 
    = \int_0^1 \bar{\tau}_m(t;x) \dd F_{N}(x).
\end{equation}
\end{lm}

\begin{proof}
Substituting \eqref{CNm} into the right-hand side of \eqref{rec-C}, 
we obtain
\begin{align*}
    \overbar{B}^{[m+1]}_{N}(x) 
    &=\int_0^x \overbar{B}^{[m]}_{N}(u) \dd u 
    - \int_0^1 (1-u) \overbar{B}^{[m]}_{N}(u) \dd u \\ 
    &= \int_{0}^x 
    \int_{0}^1 \bar{\tau}_m(t;u) \dd F_{N}(t) \dd u 
    - \int_{0}^1 (1-u) 
    \int_{0}^1 \bar{\tau}_m(t;u) \dd F_{N}(t) \dd u \\ 
    & = \int_{0}^1 \dd F_{N}(t) \int_{0}^1 
    \left(\Ind_{\{u\le x\}} -(1-u) \right) \bar{\tau}_m(t;u) \dd u \\ 
    & = \int_{0}^1 \dd F_{N}(t) 
    \int_{0}^1 \left( u \Ind_{\{u\le x\}} 
    - (1-u) \Ind_{\{u>x\}}\right) \bar{\tau}_m(t;u) \dd u.
\end{align*}
It suffices, by induction, to prove that
\begin{equation}\label{indep}
    (-1)^{m-1} \int_0^1 
    \lpa{ u \Ind_{\{u\le x\}} - (1-u) \Ind_{\{u>x\}} } 
    \bar{\tau}_m(t;u) \dd u 
    = -(-1)^m \bar{\tau}_{m+1}(t;x).
\end{equation}
By the definition of the generating function \eqref{gen_func},
$\b_{m}(x) = [z^m]V(z;x)$ and thus $\bar{\tau}_m(t;x)$
in \eqref{tau_m-w} can be written as
\begin{equation}\label{tau-V}
    \bar{\tau}_m(t;x)
    = (-1)^{m-1} [z^m] \bigl(V(z;t-x+1) \Ind_{\{t\le x\}} 
    + V(z;t-x) \Ind_{\{t>x\}}\bigr).
\end{equation}
With this relation, together with (\ref{rec-gen}), the
left-hand side of \eqref{indep} becomes
\begin{align*}
    & [z^m] \int_0^1 (u \Ind_{\{u\le x\}} 
    - (1-u) \Ind_{\{u>x\}}) 
    \bigl(V(z;t-u+1) \Ind_{\{t\le u\}} 
    + V(z;t-u) \Ind_{\{t>u\}}\bigr) \dd u \\
    &= [z^m] \biggl\{ \Ind_{\{t\le x\}} 
    \int_t^x u V(z;t-u+1) \dd u 
    - \int_{x\vee t}^1 (1-u) V(z;t-u+1) \dd u \\
    &\qquad\quad\ \ + \int_0^{x\wedge t} u V(z;t-u) \dd u 
    - \Ind_{\{t>x\}} \int_x^t (1-u) V(z;t-u) \dd u \biggr\} \\
    &=\begin{cases}\displaystyle
        [z^m] \left\{ -z^{-1} V(z;t-x+1) + t 
        + z^{-1} \right\} & (t\le x), \\
        \displaystyle
        [z^m] \left( -z^{-1} V(z;t-x) 
        - (1-t) + z^{-1} \right) & (t>x)
    \end{cases} \\
    & = -[z^{m+1}] \bigl(V(z;t-x+1) \Ind_{\{t\le x\}} 
    + V(z;t-x) \Ind_{\{t>x\}}\bigr) \\
    & = -(-1)^m \bar{\tau}_{m+1}(t;x),
\end{align*}
which proves \eqref{indep} and completes the proof of the lemma.
\end{proof}

Following a similar construction principle used for GAD statistic 
where the template functions are orthogonal to eigenfunctions of 
lower degrees, we propose the truncated version of the template 
function
\begin{equation}\label{tau_m*-w}
    \bar{\tau}^*_m(t;x) 
    = \bar{\tau}_m(t;x) - 2 \sum_{1\le k< m} \frac{1}{(2k\pi)^m} 
    \cos\Bigl(2 k\pi (t-x)+\sfrac{m\pi}{2}\Bigr),
\end{equation}
(compare \eqref{ad-tau_mtx}) and define the truncated version of the 
generalized Watson statistics
\begin{align*}
    U^{*[m]}_{N} 
    = \int_0^1 \overbar{B}^{*[m]}_{N}(x)^2 \dd x,\quad
    \overbar{B}^{*[m]}_{N}(x) 
    = \int_0^1 \bar{\tau}^*_m(t;x) \dd F_{N}(x);
\end{align*}
see also \eqref{eigen_exp}.

The first few $\bar{\tau}^*_m$ are plotted in 
Figure \ref{fig:template2}.
\begin{figure}[h]
\begin{center}
\begin{tabular}{cc}
$\bar{\tau}^*_1(t;x_0)$ & $\bar{\tau}^*_2(t;x_0)$ \\
\includegraphics[height=3.2cm]{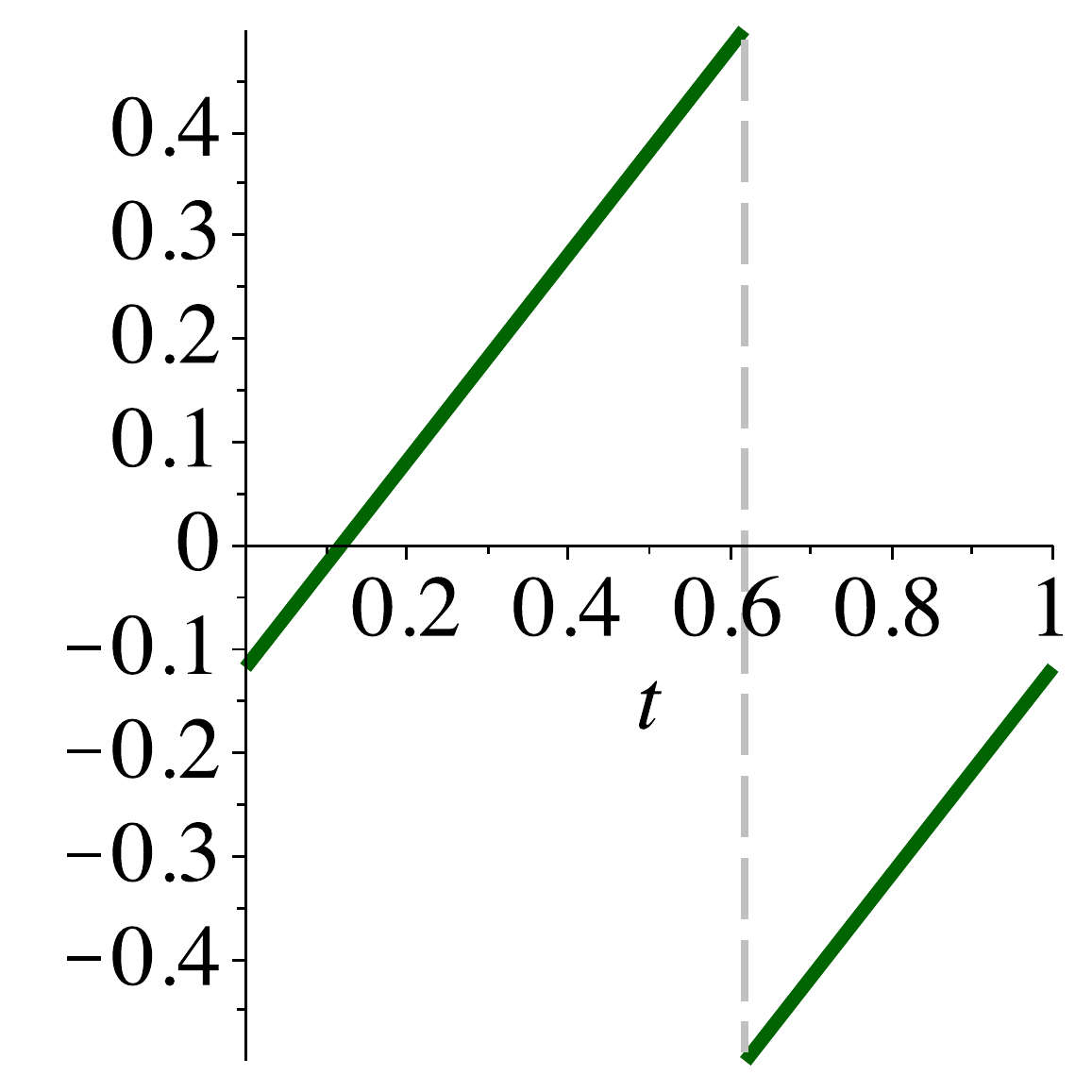} &
\includegraphics[height=3.2cm]{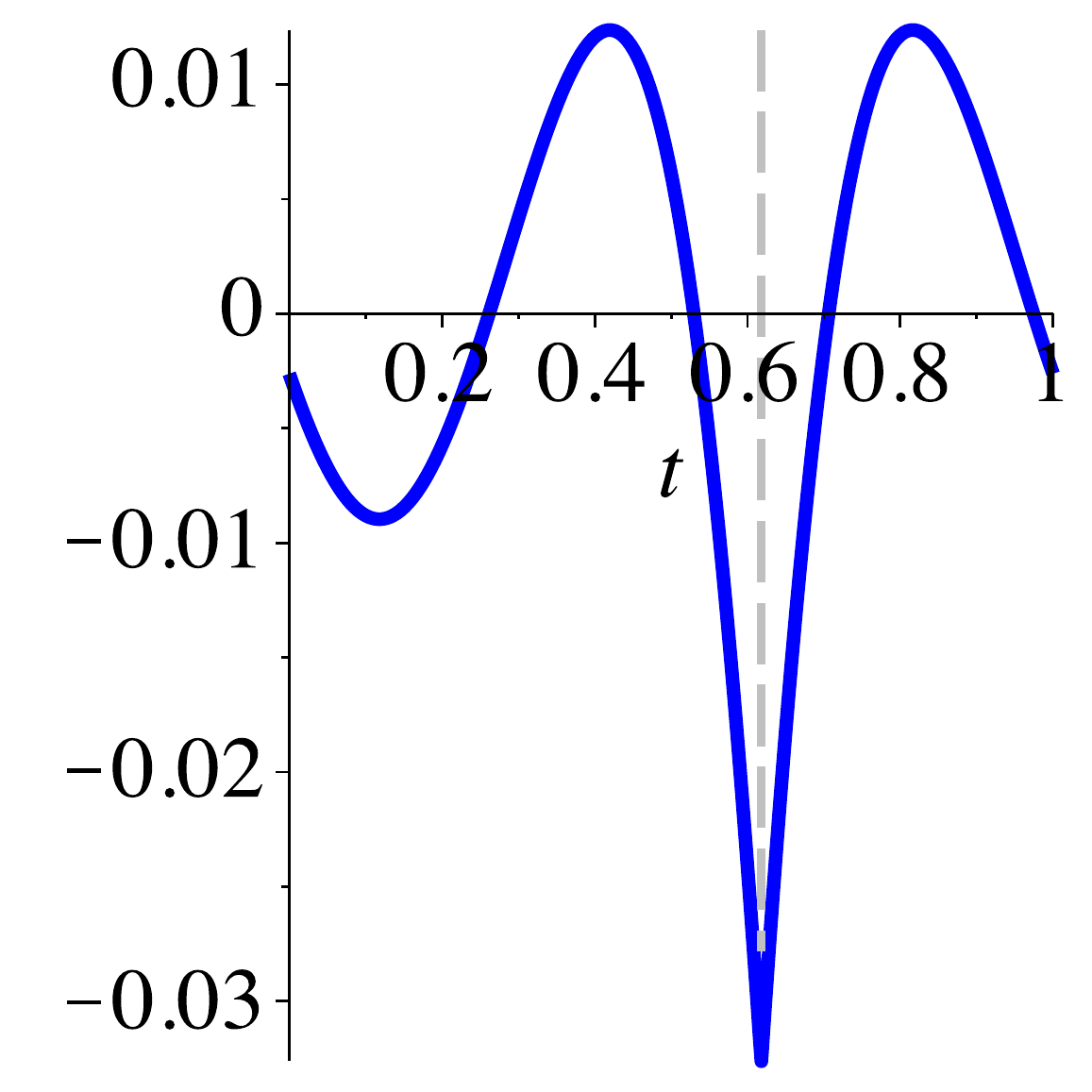} \\
$\bar{\tau}^*_3(t;x_0)$ & $\bar{\tau}^*_4(t;x_0)$ \\
\includegraphics[height=3.2cm]{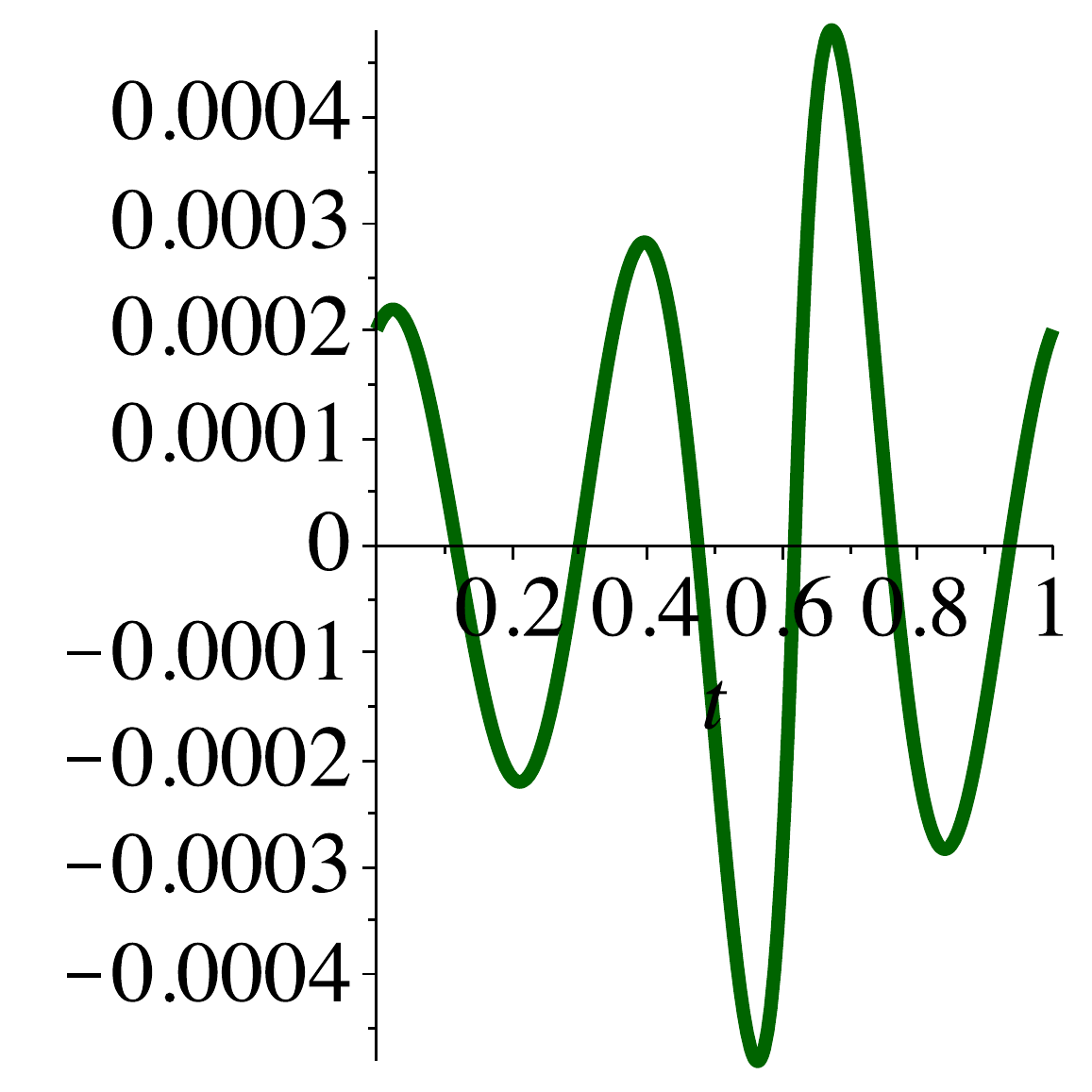} &
\includegraphics[height=3.2cm]{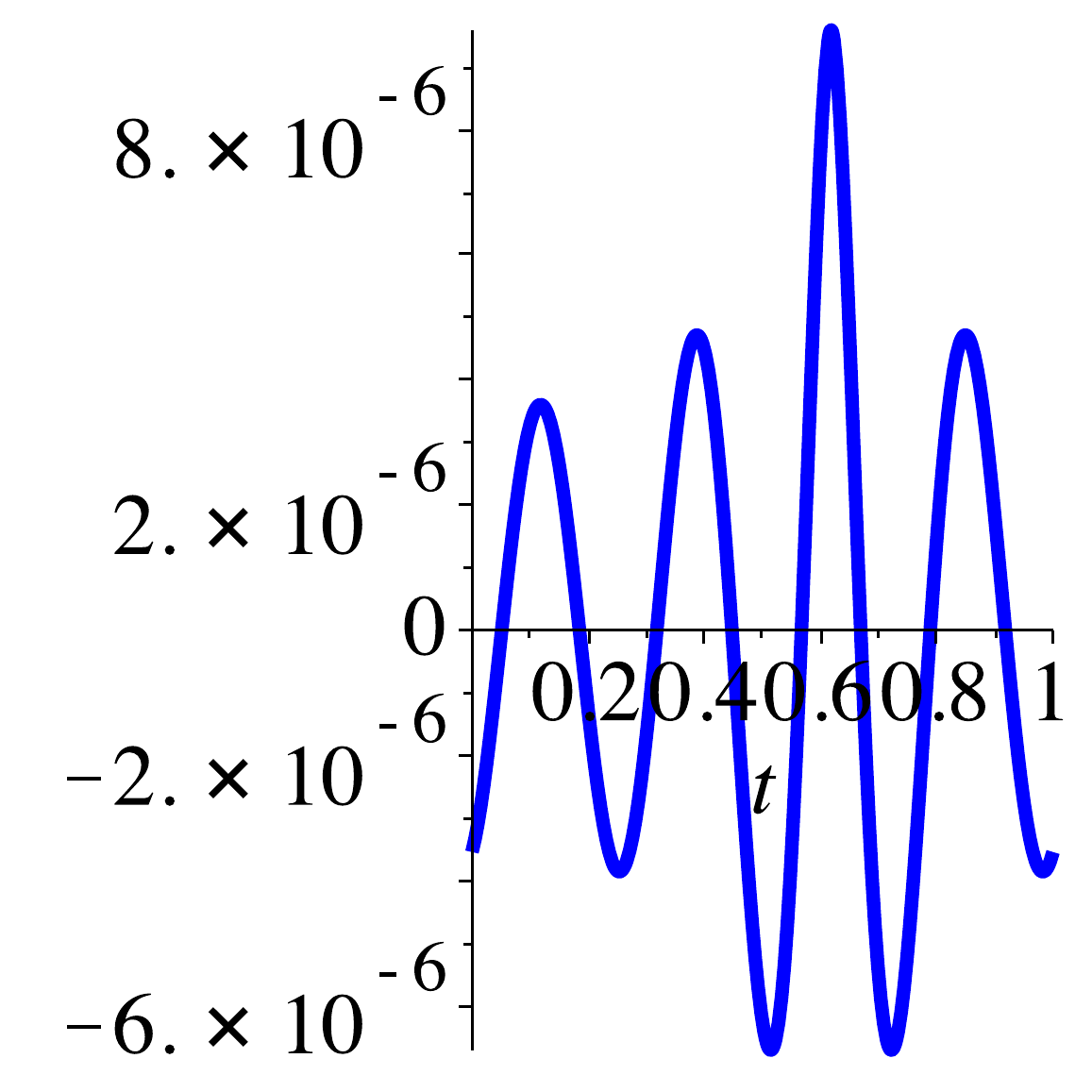}\\
\end{tabular}
\end{center}
\vspace*{-.3cm}
\caption{\emph{Template functions 
$\bar{\tau}^*_m(\cdot;x_0)$ for the generalized Watson 
statistics for $1\le m\le 4$; here $x_0= \frac{\sqrt{5}-1}2$, the 
reciprocal of the golden ratio.}}
\label{fig:template2}
\end{figure}

\begin{co}
\label{co:invariance}
The statistics $U_N^{[m]}$ and $U_N^{*[m]}$ are well-defined and
independent of the choice of the origin. They are invariant with
respect the sign change $X_i\mapsto \{-X_i\}=1-X_i$.
\end{co}
\begin{proof}
Consider first $U_N^{[m]}$. If $X_i \mapsto \{X_i+a\}$, 
$i=1,\ldots,N$, then 
\begin{align*}
    \overbar{B}^{[m]}_{N}(x)
    = \sfrac{1}{\sqrt{N}}\sum_{1\le i\le N} 
    \bar{\tau}_m(X_i;x)\mapsto 
    & \sfrac{1}{\sqrt{N}}\sum_{1\le i\le N} 
    \bar{\tau}_m(\{X_i+a\};x) \\
    & = \sfrac{1}{\sqrt{N}}\sum_{1\le i\le N} 
    \bar{\tau}_m(X_i;\{x-a\}) \\
    & = \overbar{B}^{[m]}_{N}(\{x-a\}).
\end{align*}
Thus
\[
    U^{[m]}_{N}
    = \int_0^1 \overbar{B}^{[m]}_{N}(x)^2 \dd x 
    \mapsto \int_0^1 \overbar{B}^{[m]}_{N}(\{x-a\})^2 \dd x 
    = U^{[m]}_{N},
\]
where we used the relation
\begin{align*}
    (\{X_i+a\}-x) - (X_i-\{x-a\}) 
    &= - (X_i+a-\{X_i+a\}) - (x-a-\{x-a\}) \\
    &= 0 \bmod 1.
\end{align*}
On the other hand, if $X_i \mapsto \{-X_i\}=1-X_i$, $i=1,\ldots,N$, 
then 
\begin{align*}
    \overbar{B}^{[m]}_{N}(x)
    = \sfrac{1}{\sqrt{N}}\sum_{1\le i\le N} 
    \bar{\tau}_m(X_i;x)\mapsto 
    & \sfrac{1}{\sqrt{N}}\sum_{1\le i\le N} 
    \bar{\tau}_m(1-X_i;x) \\
    & = \sfrac{1}{\sqrt{N}}\sum_{1\le i\le N} 
    (-1)^m \bar{\tau}_m(X_i;1-x) \\
    & = (-1)^m \overbar{B}^{[m]}_{N}(1-x).
\end{align*}
It follows that 
\[
    U^{[m]}_{N} 
    = \int_0^1 \overbar{B}^{[m]}_{N}(x)^2 \dd x 
    \mapsto \int_0^1 \overbar{B}^{[m]}_{N}(1-x)^2 \dd x 
    = U^{[m]}_{N}.
\]
The proof for $U_N^{*[m]}$ is similar.
\end{proof}

\subsection{The limiting GW processes: KL expansions and covariance kernel}
\label{subsec:kl-cov-w}

In this section, we examine the limiting distribution of the GW 
statistic under the null hypothesis $H_0$.

From their constructions, $U^{[m]}_{N}$ and $U^{*[m]}_{N}$ have
almost the same eigenstructure. In this subsection, we mainly deal
with the non-truncated version $U^{[m]}_{N}$. As in
Sec.~\ref{subsec:kl-cov-ad}, we denote by $W(x)$ and $B(x)$
the Wiener process on $[0,1]$ and the Brownian bridge,
respectively. Using the template functions $\bar{\tau}_m(t;x)$
in \eqref{tau_m-w} and $\bar{\tau}^*_m(t;x)$ in \eqref{tau_m*-w},
we now define the limiting GW processes
\begin{equation}\label{Cm}
\begin{aligned}
    & \overbar{B}^{[m]}(x)
    = \int_0^1 \bar{\tau}_m(t;x) \dd W(t)
    = \int_0^1 \bar{\tau}_m(t;x) \dd B(t), \\
    & \overbar{B}^{*[m]}(x)
    = \int_0^1 \bar{\tau}^*_m(t;x) \dd W(t)
    = \int_0^1 \bar{\tau}^*_m(t;x) \dd B(t).
\end{aligned}
\end{equation}

Since the Brownian bridge $B(\cdot)$ is continuous on $[0,1]$,
both $\overbar{B}^{[m]}(\cdot)$ and $\overbar{B}^{*[m]}(\cdot)$ are 
of the class $C^{m-1}$ on $[0,1]$.

We now describe the weak convergence of these processes in the
functional spaces $L^2([0,1],\mu)$ and $L^\infty([0,1])$, the space
of all essentially bounded functions on $[0,1]$.

\begin{lm}\label{lm:Um}
Let $\mu$ be any finite measure on $[0,1]$. As $N\to\infty$,
$\overbar{B}^{[m]}_{N}(\cdot) \stackrel{d}{\to}
\overbar{B}^{[m]}(\cdot)$ and $\overbar{B}^{*[m]}_{N}(\cdot)
\stackrel{d}{\to} \overbar{B}^{*[m]}(\cdot)$ in $L^2([0,1],\mu)$ and
in $L^\infty([0,1])$, and $U^{[m]}_{N}\stackrel{d}{\to} U^{[m]}$ and
$U^{*[m]}_{N}\stackrel{d}{\to} U^{*[m]}$, where
\[
    U^{[m]} 
    = \int_0^1 \overbar{B}^{[m]}(x)^2 \dd x \quad
    \text{and}\quad
    U^{*[m]} 
    = \int_0^1 \overbar{B}^{*[m]}(x)^2 \dd x.
\]
\end{lm}
\begin{proof}
The stochastic map $x\mapsto B_{N}(x)=\sqrt{N}(F_{N}(x)-x)$ converges
weakly to $B(\cdot)$ in both $L^2([0,1],\mu)$ and $L^\infty([0,1])$.
$B^{[m]}(\cdot)$ and $B^{*[m]}(\cdot)$ in \eqref{Cm} are continuous
in $B(\cdot)$ in both $L^2([0,1],\mu)$ and $L^\infty([0,1])$. The
rest of the proof is similar to that of Lemma \ref{lm:L2}.
\end{proof}

We now clarify the covariances of the $\overbar{B}^{[m]}(\cdot)$
and $\overbar{B}^{*[m]}(\cdot)$, which are centered Gaussian
processes on $[0,1]$.
\begin{thm}\label{thm:cov-w}
For $m\ge1$
\begin{align}
\label{Kmxy2-w}
    & \Cov\lpa{\overbar{B}^{[m]}(x),\overbar{B}^{[m]}(y)}
    = (-1)^{m-1} \b_{2m}(|x-y|)
    = 2 \sum_{k\ge 1}\frac{\cos(2k\pi(x-y))}{(2k\pi)^{2m}}, \\
\label{Kmxy2*-w}
    & \Cov\lpa{\overbar{B}^{*[m]}(x),\overbar{B}^{*[m]}(y)}
    = (-1)^{m-1} \b_{2m}(|x-y|) 
    -2 \sum_{1\le k<m}\frac{\cos(2k\pi(x-y))}{(2k\pi)^{2m}}.
\end{align}
In particular,
\begin{equation}
\label{Kmxx2-w}
\begin{aligned}
    & \Var\lpa{\overbar{B}^{[m]}(x)}
    = (-1)^{m-1} \b_{2m}(0), \\
    & \Var\lpa{\overbar{B}^{*[m]}(x)}
    = (-1)^{m-1} \b_{2m}(0) 
    - 2 \sum_{1\le k< m}\frac{1}{(2k\pi)^{2m}},
\end{aligned}
\end{equation}
where the $(2m)!b_{2m}(0)$ are Bernoulli numbers.
\end{thm}
\begin{proof}
Assume without loss of generality that $0\le x\le y\le 1$. We use
again the generating function \eqref{gen_func} of $\b_{m}(x)$.
By the representation for $\bar{\tau}_m(t;x)$ in \eqref{tau-V}, we
see that the covariance function satisfies
\begin{align*}
    \Cov\lpa{\overbar{B}^{[m]}(x),\overbar{B}^{[m]}(y)}
    &= \int_0^1 \bar{\tau}_m(t;x) \bar{\tau}_m(t;y) \dd t \\
    &= [z^m w^m] \biggl(\int_0^x V(z;t-x+1)V(w;t-y+1) \dd t \\
    &\qquad\qquad
    + \int_x^y V(z;t-x) V(w;t-y+1) \dd t \\
    &\qquad\qquad
    + \int_y^1 V(z;t-x) V(w;t-y) \dd t\biggr) \\
    &= [z^m w^m]\frac{zw}{z+w}\left( \frac{e^{\Delta z}}{e^z-1}
    + \frac{e^{(1-\Delta) w}}{e^w-1} \right),
\end{align*}
where $\Delta=y-x$. Now 
\begin{align*}
    \frac{zw}{z+w}\left( \frac{e^{\Delta z}}{e^z-1}
    + \frac{e^{(1-\Delta) w}}{e^w-1} \right) 
    &= \sum_{k\ge 0} \frac{w \b_{k}(\Delta) z^{k} 
    + z \b_k(1-\Delta) w^{k}}{z+w} \\
    &= zw \sum_{k\ge 1} \b_{k}(\Delta) \frac{z^{k-1} 
    - (-w)^{k-1}}{z-(-w)} +1 \\
    &= zw \sum_{k\ge 2} \b_{k}(\Delta) \frac{z^{k-1} 
    + (-w)^{k-1}}{z-(-w)} +1 \\
    &= zw \sum_{k\ge 2} \b_{k}(\Delta) 
    \sum_{0\le j\le k-2} z^j (-w)^{k-j-2} +1.
\end{align*}
It follows that when $x\le y$, $\Cov\lpa{\overbar{B}^{[m]}(x),
\overbar{B}^{[m]}(y)} = (-1)^{m-1} \b_{2m}(\Delta)$. Similarly,
when $x<y$,
$\Cov\lpa{\overbar{B}^{[m]}(x), \overbar{B}^{[m]}(y)} =
(-1)^{m-1} \b_{2m}(-\Delta)$. We then conclude \eqref{Kmxy2-w}.

The proof of \eqref{Kmxy2*-w} is based on the identities 
\begin{equation*}
    \int_0^1 \bar{\tau}_m(t;x)
    \cos\Bigl(2k\pi(t-y)+\sfrac{m\pi}{2} \Bigr) \dd t
    = \frac{\cos(2 k\pi (x-y))}{(2 k\pi)^m},
\end{equation*}
and
\begin{equation}\label{cos-cos}
	\begin{split}
	&\int_0^1 \cos\Bigl(2k\pi(t-x)+\sfrac{m\pi}{2} \Bigr)
    \cos\Bigl(2\ell\pi(t-y)+\sfrac{m\pi}{2} \Bigr) \dd t\\
    &\qquad= \begin{cases} 
	    \sfrac{1}{2}\cos(2 k\pi (x-y)), & k=\ell, \\
        0, & k\ne\ell, 
	\end{cases}
	\end{split}
\end{equation}
whose proofs are similar and omitted.
\end{proof}

Let $\psi_{0}(x)=1$, and $\psi_{2k-1}(x)=\sqrt{2}\sin(2k\pi x)$,
$\psi_{2k}(x)=\sqrt{2}\cos(2k\pi x)$ for $k\ge 1$. Then,
$\{\psi_{k}(x)\}_{k\ge 0}$ is a complete orthonormal basis of
$L^2[0,1]$. For each $m\ge 1$, let
\[
    \psi^{[m]}_{2k-1}(x) 
    = \sqrt{2} \sin\Bigl(2k\pi x-\sfrac{m\pi}{2} \Bigr) \quad
    \text{and} \quad
    \psi^{[m]}_{2k}(x) 
    = \sqrt{2} \cos\Bigl(2k\pi x-\sfrac{m\pi}{2} \Bigr).
\]
Then $\{\psi^{[m]}_{k}(x)\}_{k\ge 0}$ also form an orthonormal basis.
By expressing $\bar{\tau}_m(\cdot;x)$ in \eqref{tau_m-w} as
a Fourier series \eqref{fourier}, and by expanding it on the basis
$\psi_k(\cdot)$, we obtain the expansion
\begin{align}
    \bar{\tau}_m(t;x) 
    &= \sum_{k\ge 1} \frac{1}{(2k\pi)^m} 
    \left(\psi^{[m]}_{2k-1}(x) \psi_{2k-1}(t) 
    + \psi^{[m]}_{2k}(x) \psi_{2k}(t) \right) \nonumber \\
    &=
    2 \sum_{k\ge 1} \frac{1}{(2k\pi)^m}
    \cos\Bigl(2 k\pi (t-x)+\sfrac{m\pi}{2}\Bigr),
\label{eigen_exp}
\end{align}
and
\[
    \bar{\tau}^*_m(t;x)
    = \sum_{k\ge m} \frac{1}{(2k\pi)^m} 
    \left(\psi^{[m]}_{2k-1}(x) \psi_{2k-1}(t) 
    + \psi^{[m]}_{2k}(x) \psi_{2k}(t) \right).
\]

By taking the stochastic integral $\int_0^1 \dd B(t)$ on both sides, 
we get the KL expansions
\begin{align}
\label{KL-w}
    & \overbar{B}^{[m]}(x) 
    = \sum_{k\ge 1}\frac{1}{(2k\pi)^{m}}
    \left( \psi^{[m]}_{2k-1}(x) \xi_{2k-1} 
    + \psi^{[m]}_{2k}(x) \xi_{2k} \right), \\
\label{KL*-w}
    & \overbar{B}^{*[m]}(x) 
    = \sum_{k\ge m}\frac{1}{(2k\pi)^{m}}
    \left( \psi^{[m]}_{2k-1}(x) \xi_{2k-1} 
    + \psi^{[m]}_{2k}(x) \xi_{2k} \right),
\end{align}
where
\begin{equation}\label{xi-w}
    \xi_j 
    = \int_0^1 \psi_{j}(y) \dd B(y) 
    = \int_0^1 \psi_{j}(y) \dd W(y) \sim N(0,1)\ \ \rm i.i.d.
\end{equation}
These KL expansions can again be justified by Mercer's theorem.

\begin{thm}\label{thm:KL-w}
Let $\xi_k$, $k\ge 0$, be i.i.d. standard Gaussian random variables
defined by \eqref{xi-w}. For $m\ge 1$, \eqref{KL-w} and \eqref{KL*-w}
hold uniformly in $x\in[0,1]$ with probability 1, and both
\begin{equation}\label{KL2-w}
    U^{[m]} 
    = \sum_{k\ge 1}\frac{\xi_{2k-1}^2 + \xi_{2k}^2}
    {(2k\pi)^{2m}}\quad
    \text{and}\quad
    U^{*[m]} 
    = \sum_{k\ge m}\frac{\xi_{2k-1}^2 + \xi_{2k}^2}
    {(2k\pi)^{2m}}
\end{equation}
converge with probability 1.
\end{thm}
\begin{proof}
Note that $\overbar{B}^{[m]}(\cdot)$ is continuous on $[0,1]$. The
covariance function of $\overbar{B}^{[m]}(\cdot)$ is given in
\eqref{Kmxy2-w}, which is continuous on $[0,1]\times [0,1]$.
Moreover, by \eqref{Kmxx2-w},
\[
    \int_0^1 \Cov(\overbar{B}^{[m]}(x),
    \overbar{B}^{[m]}(x)) \dd x <\infty.
\]
In addition, $\Cov\lpa{\overbar{B}^{[m]}(x),\overbar{B}^{[m]}(y)}$
satisfies the eigenequation 
\[
     \int_0^1 \Cov\lpa{\overbar{B}^{[m]}(x),\overbar{B}^{[m]}(y)}
     \psi^{[m]}_{k}(y) \dd y 
     = \frac{\psi^{[m]}_{k}(x)}
     {(2\lceil k/2\rceil\pi)^{2m}}, \qquad k\ge 1.
\]
The remaining proof then follows the same line as that of Theorem
\ref{thm:KL}. The required results for $\overbar{B}^{*[m]}(\cdot)$ and
$U^{*[m]}$ are proved similarly.
\end{proof}

The covariance function for $B^{[m]}(\cdot)$ when $m=1$ is given 
explicitly by 
\[
    \Cov\lpa{\overbar{B}^{[1]}(x),\overbar{B}^{[1]}(y)}
    = \sfrac{1}{12} -\sfrac{1}{2}|x-y| +\sfrac{1}{2}(x-y)^2,
\]
(see also \cite[p.\,111]{Watson61}), and when $m=2$
\[
    \Cov\lpa{\overbar{B}^{[2]}(x),\overbar{B}^{[2]}(y)}
    = \sfrac{1}{720}-\sfrac{1}{24}(x-y)^2
    +\sfrac{1}{12}|x-y|^3-\sfrac{1}{24}(x-y)^4;
\]
see also \cite{Henze02}.

\subsection{Statistic in terms of the samples}
\label{subsec:sample-w}

We now derive an expression for GW statistic in terms of the sample 
$X_i$'s, which turns out to be simpler than that of the GAD 
statistic. Since
\begin{equation}\label{UNm2}
    U^{[m]}_{N}
    = N \int_0^1 \biggl( \int_0^1 
    \bar{\tau}_m(t;x) \bar{\tau}_m(u;x) \dd x \biggr) \dd F_{N}(t) 
    \dd F_{N}(u),
\end{equation}
it suffices to evaluate the integral within the parentheses above, 
which can be done in essentially the same way as the proof of Theorem 
\ref{thm:cov-w}.
\begin{thm}\label{thm:sample-w}
The generalized Watson statistics are expressed in terms of samples 
$X_i$'s as
\begin{equation}\label{UNm-sample}
    U^{[m]}_{N}
    = \frac{1}{N} \sum_{1\le i,j\le N}  
    (-1)^{m-1} \b_{2m}(|X_i-X_j|),
\end{equation}
and
\begin{equation}
\label{UN*m-sample}
    U^{*[m]}_{N}
    = \frac{1}{N} \sum_{1\le i,j\le N} \biggl\{
    (-1)^{m-1} \b_{2m}(|X_i-X_j|)
    -2 \sum_{1\le k<m}\frac{\cos(2k\pi(X_i-X_j))}
    {(2k\pi)^{2m}}\biggr\}.
\end{equation}
\end{thm}
\begin{proof}
Let $\Delta=u-t$.
First assume that $0\le t\le u\le 1$.  The integral between
the large parentheses in \eqref{UNm2} is
\begin{align}
    \int_0^1 \bar{\tau}_m(t;x) \bar{\tau}_m(u;x) \dd x
    &= [z^m w^m] \biggl(\int_u^1 V(z;t-x+1)V(w;u-x+1)\dd t \nonumber \\
    &\qquad\qquad
    + \int_t^u V(z;t-x+1) V(w;u-x) \dd t \nonumber \\
    &\qquad\qquad + \int_0^t V(z;t-x) V(w;u-x) \dd t\biggr) \nonumber \\
    &= [z^m w^m]\frac{zw}{z+w}\left( \frac{e^{(1-\Delta) z}}{e^z-1}
    + \frac{e^{\Delta w}}{e^w-1} \right).
\label{zw}
\end{align}
By the same calculation in the proof of Theorem \ref{thm:cov-w}, it holds that $(\ref{zw})=  b_m(1-\Delta)= (-1)^m b_m(\Delta)$.
Similarly for $0\le u\le t\le 1$,
$(\ref{zw})=  b_m(1+\Delta)= (-1)^m b_m(-\Delta)$.
In both cases, $(\ref{zw}) = (-1)^m b_m(|\Delta|)$.
We then conclude \eqref{UNm-sample}.

For the proof of \eqref{UN*m-sample}, we use the identities
\[
    \int_0^1 \bar{\tau}_m(t;x)
    \cos\Bigl(2k\pi(u-x)+\sfrac{m\pi}{2} \Bigr) \dd x
    = \frac{\cos(2 k\pi (t-u))}{(2 k\pi)^m}
\]
and (\ref{cos-cos}).
\end{proof}

When $m=1$, the statistic is
\begin{align*}
    U^{[1]}_{N}
    &= \sfrac{1}{2N} \sum_{1\le i,j\le N} \Bigl(
    \sfrac{1}{6} -|X_i-X_j|
    + (X_i-X_j)^2 \Bigr) \\
    &= \sfrac{1}{2N} \sum_{1\le i,j\le N} 
    \Bigl(-\sfrac{1}{12} + \Lpa{|X_i-X_j|-\sfrac{1}{2}}^2 \Bigr).
\end{align*}
This is equivalent to $N$ times
\[
    \sum_{1\le i\le N} 
    \Bigl( X_{(i)}-\sfrac{2i-1}{2 N}-\bar X+\sfrac{1}{2} \Bigr)^2 
    + \sfrac{1}{12 N},
\]
where $\bar X = \frac{1}{N}\sum_{1\le i\le N} X_i$ represents the 
sample mean; see \cite{Watson61}.

When $m=2$,
\[
    U^{[2]}_{N}
    = \sfrac{1}{N} \sum_{1\le i,j\le N}  
    \Bigl(\sfrac{1}{720} -\sfrac{(X_i-X_j)^2}{24} 
    +\sfrac{|X_i - X_j|^3}{12} -\sfrac{(X_i-X_j)^4}{24} \Bigr),
\]
and
\[
    U^{*[2]}_{N} 
    = U^{[2]}_{N} - \sfrac{2}{N} 
    \sum_{1\le i,j\le N} \frac{\cos(2\pi (X_i-X_j))}{(2\pi)^4};
\]
see also \cite{Henze02}.

\subsection{Power comparisons}
\label{subsec:power-w}

Figure \ref{fig:U} summarizes the the powers of the truncated 
statistics $U_N^{*[1]}=U_N^{[1]}$, $U_N^{*[2]}$ and $U_N^{*[3]}$ for 
testing the uniformity estimated by numerical simulations.
We generate the random variables in $[0,1)$ from a mixture of three 
von Mises distributions:
\[
    w_0 f(x;\theta_0,\kappa) + w_1 f(x;\theta_1,\kappa) 
    + w_2 f(x;\theta_2,\kappa)
\]
with the density
\[
    f(x;\theta,\kappa) 
    = \frac{e^{\kappa\cos(x-\theta)}}{2\pi I_0(\kappa)}
    \qquad (x\in[0,2\pi)),
\]
where $I_0(\kappa)$ is the modified Bessel function of order 0.

By changing the parameters $\theta_i$ and $w_i$, this model can 
describe unimodal, bimodal, and trimodal distributions. The 
configurations we used in our numerical simulations are summarized 
below:
\begin{center}
\begin{tabular}{ccccl}
Model & ($\theta_0,\theta_1,\theta_2)$ & $(w_0,w_1,w_2)$ 
& \# of modes \\ \hline
I & $(0,*,*)$ & $(1,0,0)$ & $1$ & \\[0.5mm]
II & $(0,\pi,*)$ & $(\frac{1}{2},\frac{1}{2},0)$ & 2 \\[1mm]
III & $(0,\frac{2\pi}{3},\frac{4\pi}{3})$ 
& $(\frac{1}{3},\frac{1}{3},\frac{1}{3})$ & 3 \\[1mm]
I+II & $(0,\pi,*)$ & $(\frac{2}{3},\frac{1}{3},0)$ & 2 \\[1mm]
I+III & $(0,\frac{2\pi}{3},\frac{4\pi}{3})$ 
& $(\frac{1}{2},\frac{1}{4},\frac{1}{4})$ & 3 \\[1mm]
I'+II & $(0,\frac{\pi}{3},\pi)$ 
& $(\frac{1}{4},\frac{1}{2},\frac{1}{4})$ & 3 \\[1mm] 
\end{tabular}
\end{center}

Model I is unimodal, Models II is anti-modal with the modes located
at $0$ and $\pi$, and Models III is equally-spaced trimodal; the 
modes in each case are of the same magnitude. The other models
I+II, I+III, I'+II are mixture models of I and II, I and III,
and I (shifted) and II, respectively. They all become the uniform
distribution when $\kappa=0$.

All simulation parameters are as in Sec.~\ref{subsec:finite}: the
sample size $N$ is $100$, the size of the tests is 0.05.  Monte Carlo
simulations to estimate the powers of $U_N^{*[m]}$ ($m=1,2,3$) are
set to run with 300,000 replications, and the critical values are 
estimated by simulations in advance.
The results are summarized in Figure \ref{fig:U} as in
Sec.~\ref{subsec:finite}.

One sees that, the statistics $U_N^{*[1]}$, $U_N^{*[2]}$, and
$U_N^{*[3]}$ perform the best in Models I, II, and III, respectively.
This is reasonable because the finite sample expansions for
$U_N^{*[m]}$ are given by
\[
    U^{*[m]}_{N}
    = \sum_{k\ge m}
    \frac{\widehat\xi_{2k-1}^2+\widehat\xi_{2k}^2}{(2k\pi)^{2m}} 
    = \frac{\widehat\xi_{2m-1}^2+\widehat\xi_{2m}^2}{(2m\pi)^{2m}}
    +\frac{\widehat\xi_{2m+1}^2+\widehat\xi_{2m+2}^2}{(2(m+1)\pi)^{2m}}
    +\cdots
\]
where
\[
    \widehat\xi_{2k-1} 
    = \frac{\sqrt{2}}{\sqrt{N}}\sum_{1\le \ell \le N} 
    \sin(2k\pi X_{\ell}) \quad\text{and}\quad
    \widehat\xi_{2k} 
    = \frac{\sqrt{2}}{\sqrt{N}}\sum_{1\le \ell \le N}
    \cos(2k\pi X_{\ell}).
\]
In particular, $\widehat\xi_{2m-1}$ and $\widehat\xi_{2m}$,
being the leading terms of $U_N^{*[m]}$ and
representing cyclic components with cycle $2m\pi$, are
expected to detect the shape with $m$ modes. For the other
three models, there is a dominating mode, and thus $U_N^{[1]}$ 
has better power than the others.

\begin{figure}
\begin{center}
\begin{tabular}{cc}
\includegraphics[scale=0.25]{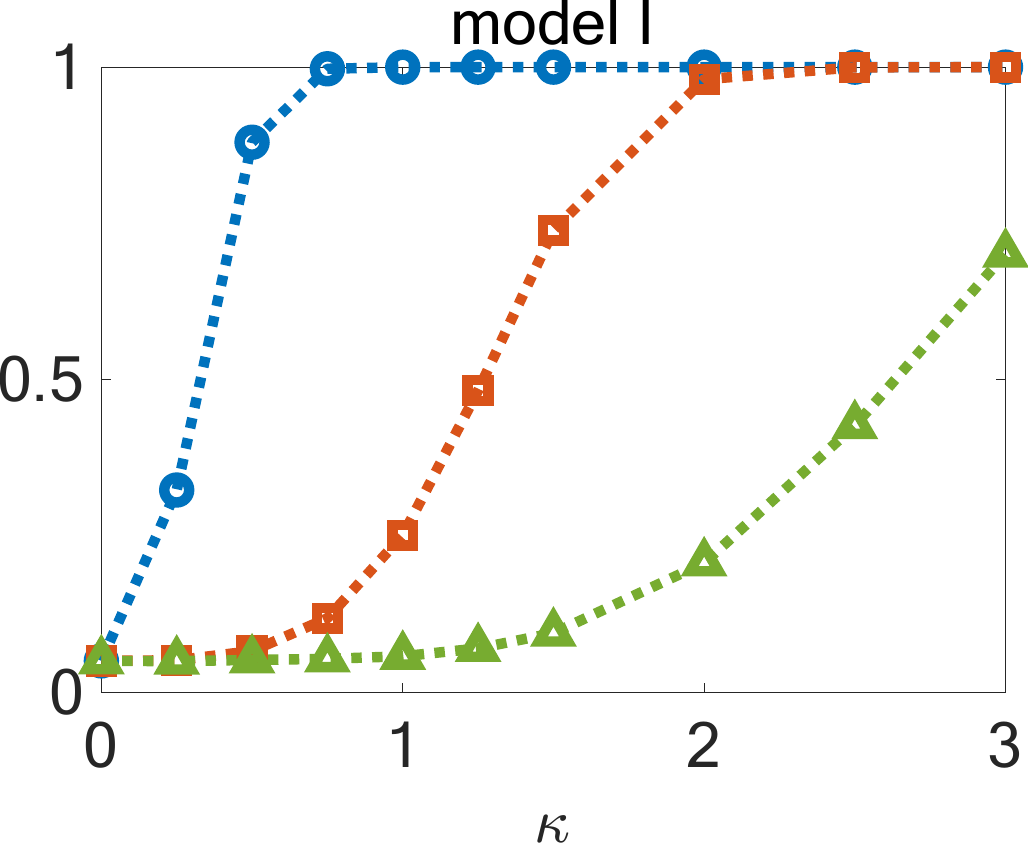} &
\includegraphics[scale=0.25]{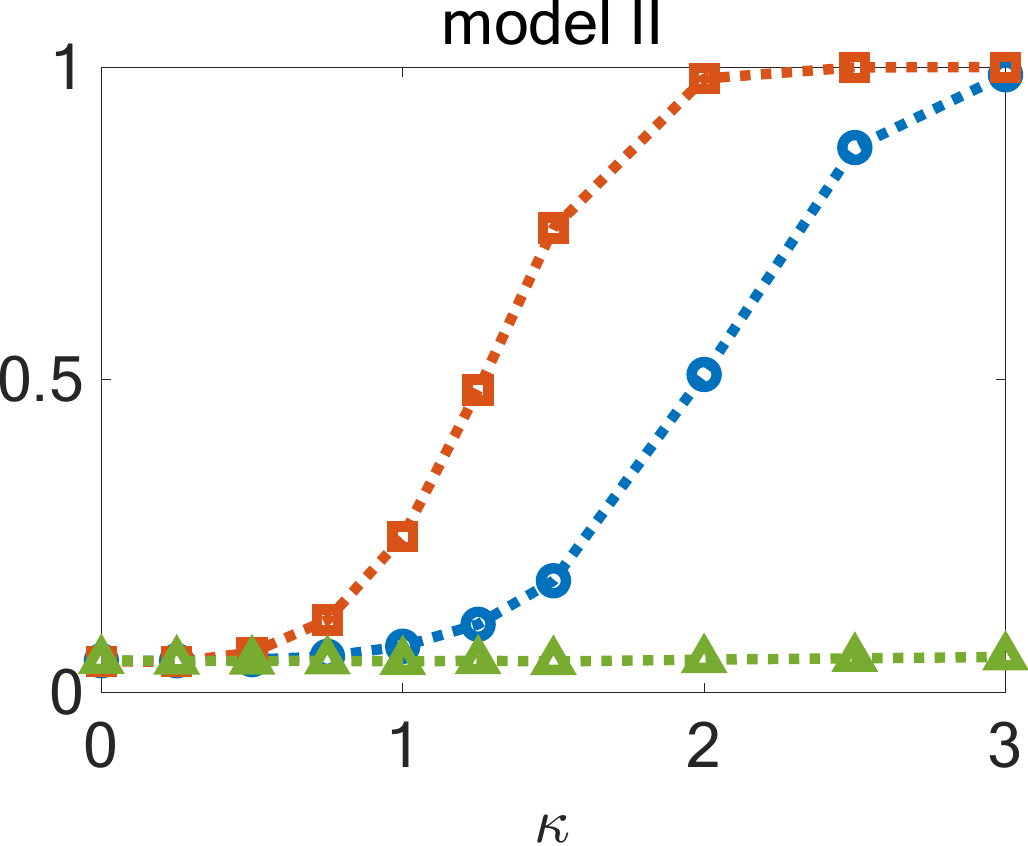} \\
\includegraphics[scale=0.25]{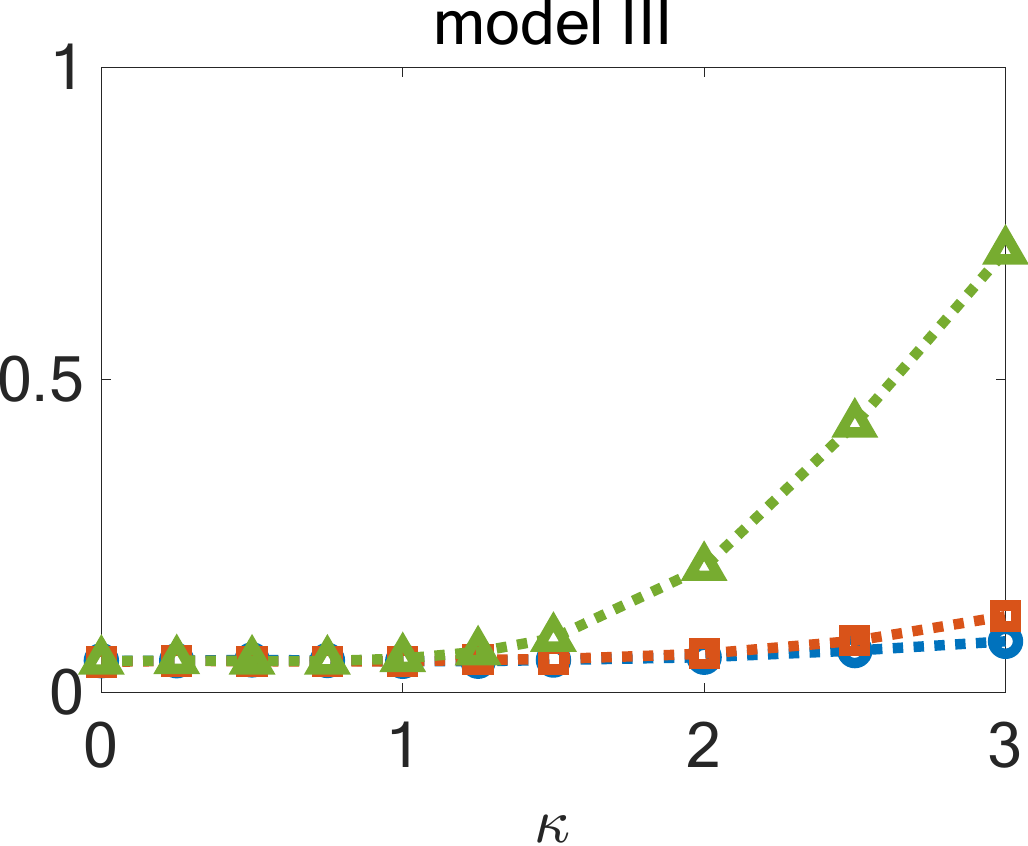} &
\includegraphics[scale=0.25]{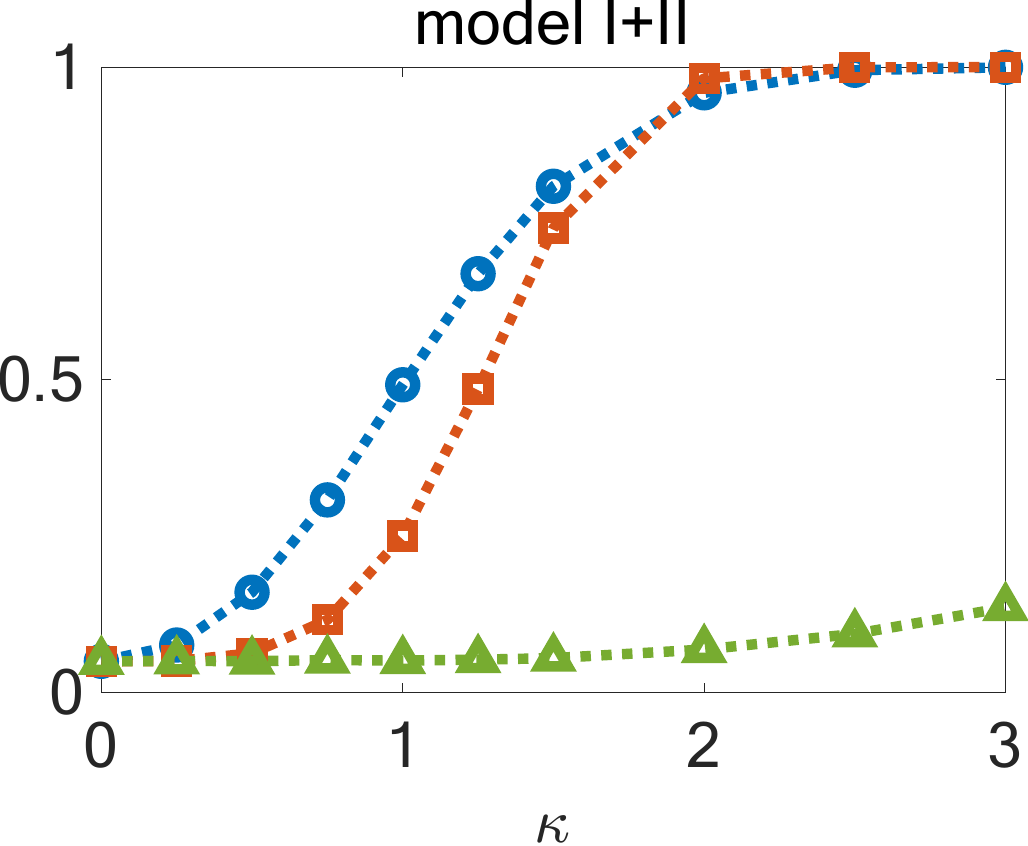} \\
\includegraphics[scale=0.25]{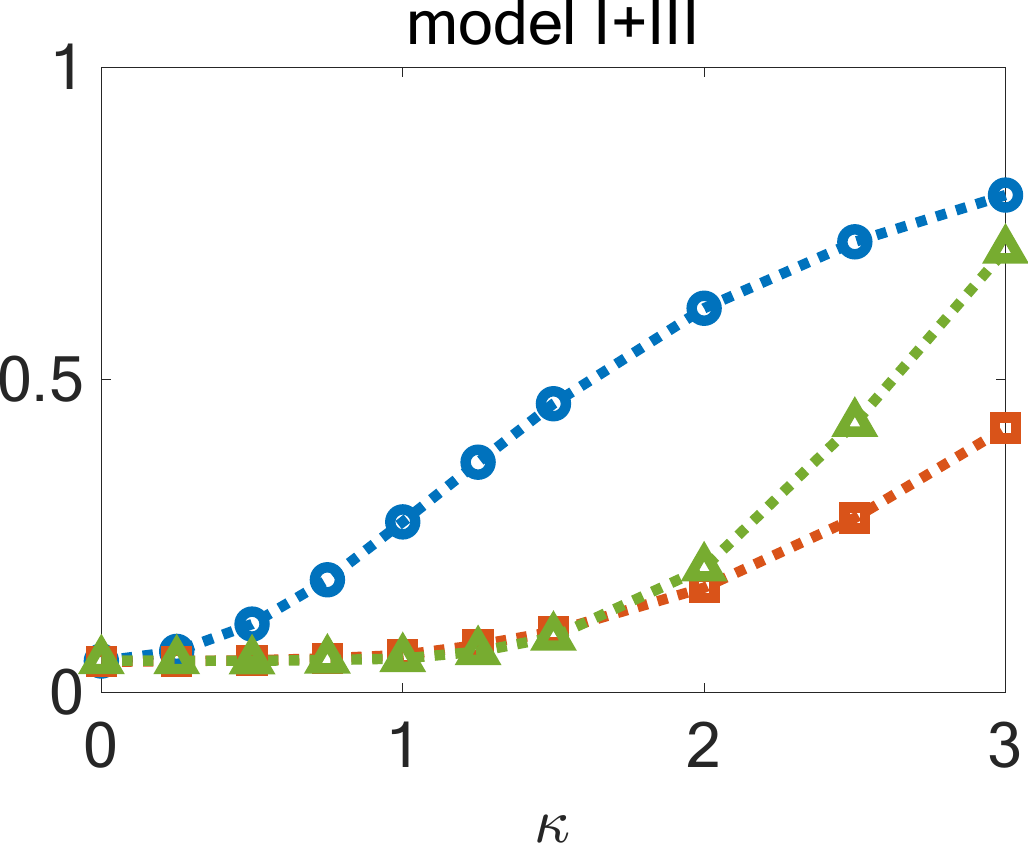} &
\includegraphics[scale=0.25]{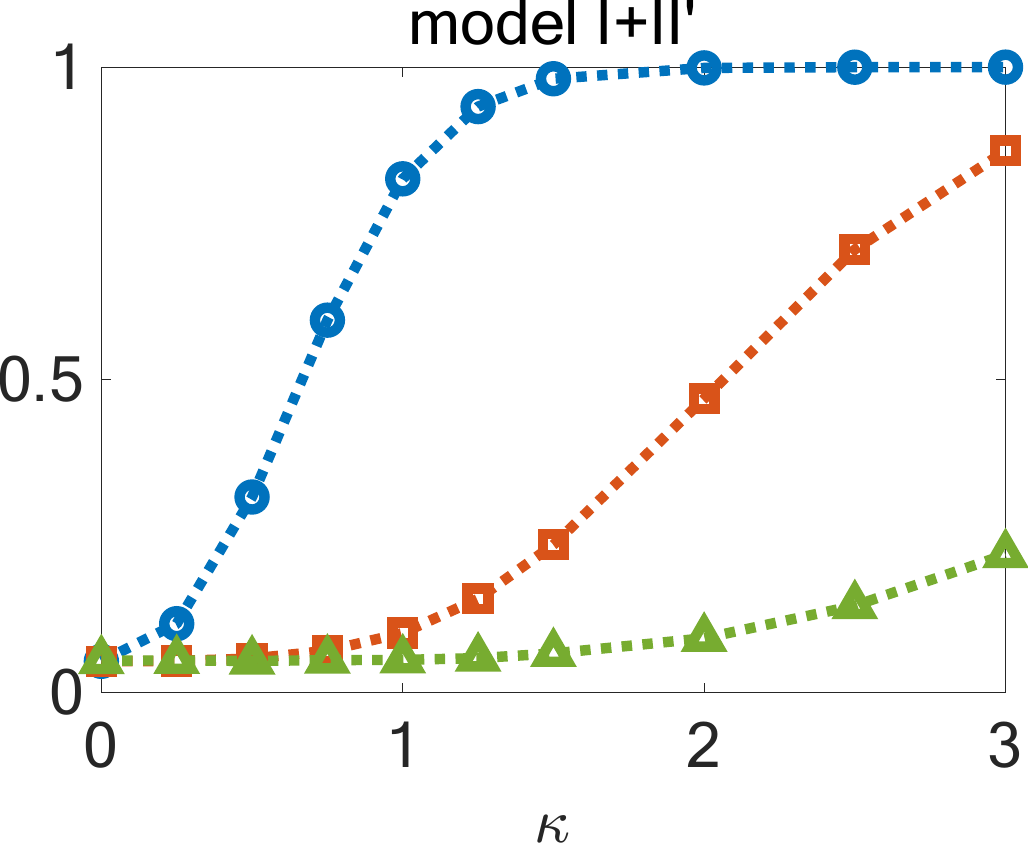}
\end{tabular}
\caption{Powers of $U^{*[m]}_N$ under the mixture models of von Mises distributions.}
\smallskip
(\blueredgreen)
\label{fig:U}
\end{center}
\end{figure}

\subsection{Moment generating function}
\label{subsec:mgf-w}

By the series representation \eqref{KL2-w}, the MGF of the limiting 
GW statistics now have the representations 
\begin{align}\label{EesUm}
    \mathbb{E}\lpa{e^{s U^{[m]}}}
    = \prod_{k\ge1} \frac1{1-\frac{2s}{\lambda_k}}
    \quad\text{and}\quad
    \mathbb{E}\lpa{e^{s U^{*[m]}}}
    = \mathbb{E}\lpa{e^{s U^{[m]}}} 
    \prod_{1\le k<m} \Bigl(1-\frac{2s}{\lambda_k}\Bigr),
\end{align}
where $\lambda_k:=(2k\pi)^{2m}$. The infinite series represents a 
meromorphic function in the complex $s$-plane with simple poles at 
$s=\frac12\lambda_k$ for $k=1,2,\dots$. Similar to 
Theorem~\ref{thm:gad-cos}, the infinite product in 
$\mathbb{E}\lpa{e^{s U^{[m]}}}$ can be simplified.
\begin{thm}\label{thm:mgf-w}
For $m\ge1$
\begin{equation}\label{gw-fp}
    \mathbb{E}\lpa{e^{s U^{[m]}}}
    = \frac{\sqrt{2s}\, e^{\frac{m-1}{2}\pi i} }  
    {2^m\prod_{0\le j<m}\sin\lpa{\frac12 
    (2s)^{\frac{1}{2m}}e^{\frac{j\pi i}{m}}}}.
\end{equation}
\end{thm}
\begin{proof}
The proof is similar to but simpler than that of
Theorem~\ref{thm:gad-cos} because the zeros of the polynomial
equation $(2\pi z)^{2m} = 2s$ have the closed-forms $\zeta_j(s)
=\frac1{2\pi} (2s)^{\frac1{2m}} e^{\frac{j\pi i}m}$ for $0\le j<2m$.
So we first have the finite-product expression
\[
    \mathbb{E}\lpa{e^{s U^{[m]}}}
    =\prod_{0\le j<2m}\Gamma\lpa{1-\zeta_j(s)}.
\]
Then we obtain \eqref{gw-fp} by rearranging the terms by 
grouping the zeros in pairs and by Euler's reflection formula.
Details are omitted here.
\end{proof}
In particular, with $\sigma_m := \frac12(2s)^{\frac1{2m}}$, we have 
\begin{align*}
    \mathbb{E}\lpa{e^{sU^{[1]}}} 
    &=\frac{\sigma_1}{\sin(\sigma_1)}, \\
    \mathbb{E}\lpa{e^{sU^{[2]}}} 
    &= \frac{\sigma_2^2}{\sin(\sigma_2)\sinh(\sigma_2)}, \\
    \mathbb{E}\lpa{e^{sU^{[3]}}}
    &= \frac{2\sigma_3^3}
    {\sin(\sigma_3)\lpa{\cosh(\sqrt{3}\sigma_3)-\cos(\sigma_3)}}, \\
    \mathbb{E}\lpa{e^{sU^{[4]}}}
    &= \frac{2\sigma_4^4}
    {\sin(\sigma_4)\sinh(\sigma_4)\lpa{\cosh(\sqrt{2}\sigma_4)
    -\cos(\sqrt{2}\sigma_4)}}.
\end{align*}

Since only simple poles instead of branch singularities are involved 
in the infinite product representation \eqref{EesUm} of 
$\mathbb{E}\lpa{e^{sU^{[m]}}}$, we can easily invert the MGF by using 
Laplace transform and obtain the following expressions.
\begin{thm} For $m\ge1$, the density functions $f_m(x)$ of $U^{[m]}$ 
are given by
\[
    f_m(x) = m4^m\sum_{k\ge1} 
    \frac{(-1)^{k-1}(k\pi)^{3m-1}e^{\frac{m-1}2\pi i}}
    {\prod_{1\le j<m}\sin \lpa{k\pi e^{\frac{j\pi i}m}}}
    \,e^{-\frac12(2k\pi)^{2m}x},
\]  
and the distribution functions by
\[
    1 - 2m\sum_{k\ge1} 
    \frac{(-1)^{k-1}(k\pi)^{m-1}e^{\frac{m-1}2\pi i}}
    {\prod_{1\le j<m}\sin \lpa{k\pi e^{\frac{j\pi i}m}}}
    \,e^{-\frac12(2k\pi)^{2m}x},
\]
for $x>0$.
\end{thm}
\begin{proof}
We start with the partial fraction expansion 
\begin{align}\label{gw-parfrac}
    \prod_{k\ge1} \frac1{1-\frac{2s}{(2km)^{2m}}}
    = 2m\sum_{k\ge1}\frac{(-1)^{k-1}}{1-\frac{2s}{(2km)^{2m}}}
    \cdot\frac{(k\pi)^{m-1}e^{\frac{m-1}2\pi i}}
    {\prod_{1\le j<m}\sin \lpa{k\pi e^{\frac{j\pi i}m}}},
\end{align}
which is absolutely convergent for $m\ge2$. This is proved by a 
standard argument
\[
    \prod_{k\ge1}\frac1{1-\frac{2s}{(2km)^{2m}}}
    = \sum_{k\ge1}\frac1{1-\frac{2s}{(2km)^{2m}}}
    \prod_{j\ne k}\frac1{1-\frac{k^{2m}}{j^{2m}}},
\]
where the inner product can be simplified by a similar method of 
proof used in proving Theorem~\ref{thm:mgf-w}
\[
    \prod_{j\not=k} \frac1{1-\frac{k^{2m}}{j^{2m}}}
    = (-1)^{k-1}\frac{2m(k\pi)^{m-1} e^{\frac{m-1}2\pi i}}
    {\prod_{1\le j<m}\sin\bigl(e^{\frac{j\pi i}m}k\pi\bigr)}
    \qquad(k\ge1; m\ge2).
\]
For large $k$ and $m\ge2$, the terms in \eqref{gw-parfrac} decrease 
in the order
\[
    2^m m (k\pi)^{m-1} 
    \exp\llpa{-k\pi \sum_{1\le j<m}
    \sin\frac{j\pi}m}
    = 2^m m (k\pi)^{m-1} 
    \exp\left(-k\pi \cot\frac{\pi}{2m} \right).
\]
Note that $\cot \frac{\pi}{2m}\sim \frac{2m}{\pi}$, we see that the 
partial fraction expansion \eqref{gw-parfrac} converge absolutely 
and very fast for $m\ge2$.

Thus the decomposition \eqref{gw-parfrac} implies that $U^{[m]}$ is 
an infinite sum of exponential distributions, and the expressions for 
the density and the distribution functions follow when $m\ge2$, which 
are also easily checked to hold when $m=1$ (see \cite{Watson61}).
\end{proof}

In particular, the density function of $U^{[1]}$ is given by 
\[
    f_1(x) = 4\sum_{k\ge1}(-1)^{k-1}
    (k\pi)^2 e^{-2(k\pi)^2x}\qquad(x>0);
\]
as in \cite{Watson61}, and that of $U^{[2]}$ by 
\[
    f_2(x) = 32\sum_{k\ge1}\frac{(-1)^{k+1} (k\pi)^5}
    {\sinh(k\pi)}\,e^{-8(k\pi)^4x}\qquad(x\ge0).
\]

\section{The generalized Cram\'er-von Mises statistics}
\label{sec:cvm}

In this section, we examine the corresponding generalized 
Cram\'er-von Mises (GCvM) test statistics. 

\subsection{A new class of GCvM goodness-of-fit statistics}
\label{subsec:proposal-cvm}

Recall that $F_N(x)$ denotes the empirical distribution function of
the i.i.d.\ sequence of random variables $\{X_1,\ldots,X_N\}$. The
Cram\'er-von Mises statistics is defined by
\[
    \omega_{N} 
    = \int_0^1 \widetilde{B}_{N}(x)^2 \dd x,
\]
where $\widetilde{B}_{N}(x) = \sqrt{N}(F_{N}(x)-x) = \sqrt{N}
\int_0^1 (\Ind_{\{t\le x\}}-x) \dd F_{N}(t)$ is as before. We
generalize the GCvM statistic by defining
\[
    \omega^{[m]}_{N} 
    = \int_0^1 \widetilde{B}^{[m]}_{N}(x)^2 \dd x
    \qquad(m\ge1),
\]
where $\widetilde{B}^{[1]}_{N}(x)=\widetilde{B}_{N}(x)$ and
\[
    \widetilde{B}^{[m]}_{N}(x)=
    \begin{cases}
        \displaystyle
        \int_0^x \widetilde{B}^{[m-1]}_{N}(t) \dd t, 
        & \mbox{if $m$ is odd}, \\
        \displaystyle
        \int_0^x \widetilde{B}^{[m-1]}_{N}(t) \dd t 
        - \int_0^1 (1-t) \widetilde{B}^{[m-1]}_{N}(t) \dd t, 
        & \mbox{if $m$ is even}.
    \end{cases}
\]
As we will see in detail later, the correction term $-\int_0^1 (1-t) 
\widetilde{B}^{[m-1]}_{N}(t)\dd t$ in the even case has the effect of 
keeping the template functions invariant in the sense of 
\eqref{cvm-invariance}, which induces the invariance of the 
statistics under the reflection of the data $X_i\mapsto 1-X_i$.

The iterated empirical measures $\widetilde{B}^{[m]}_{N}(x)$ are also
expressible in terms of template functions. Let $\b_m(y)$ be a
normalized Bernoulli polynomial defined in \eqref{bernoulli}. 

\begin{lm}\label{lm:hm-cvm}
For $m\ge1$
\[
    \widetilde{B}^{[m]}_{N}(x) 
	= \int_0^1 \widetilde{\tau}_m(t;x) \dd F_{N}(t),
\]
where
\begin{align}
    \widetilde{\tau}_m(t;x)
    &= \frac{(x-t)^{m-1}}{(m-1)!} \Ind_{\{t\le x\}} 
    - 2^{m-1} \left( \b_m\left(\sfrac{t+x}{2}\right)
    + (-1)^m \b_m\left(\sfrac{t-x}{2}\right) \right)
    \nonumber \\
    &= (-1)^{m-1} 2^{m-1} \b_{m}\left(\sfrac{t-x}{2}+1\right) 
    \Ind_{\{t\le x\}}
    + (-1)^{m-1} 2^{m-1} \b_{m}\left(\sfrac{t-x}{2}\right) 
    \Ind_{\{t>x\}} \nonumber \\
    & \qquad - 2^{m-1} \b_m\left(\sfrac{t+x}{2}\right).
\label{tau_m-cvm}
\end{align}
\end{lm}

Following the same construction of the template functions used in 
the generalizations of the AD and Watson statistics, we require the 
template functions to be orthogonal to the eigenfunctions of lower 
degrees. From the eigenfunction expansion \eqref{cvm-eigen_expansion} 
given later, we also propose the truncated version of the template 
functions
\begin{equation*}
    \widetilde{\tau}^*_m(t;x) = \widetilde{\tau}_m(t;x)
    - 2 \sum_{1\le k<m} \frac{1}{(\pi k)^m} 
    \cos\Bigl(k\pi x - \sfrac{m\pi}{2}\Bigr) \cos(k\pi t)
\end{equation*}
and define the truncated version of the generalized CvM statistics
\begin{equation*}
    \omega^{*[m]}_{N} 
    = \int_0^1 \widetilde{B}^{*[m]}_{N}(x)^2 \dd x
    \quad\text{where}\quad
    \widetilde{B}^{*[m]}_{N}(x) 
    = \int_0^1 \widetilde{\tau}^*_m(t;x) \dd F_{N}(x).
\end{equation*}

A graphical rendering of the first four template functions of  
$\widetilde{\tau}^*_m(t;x)$ is given in Figure~\ref{fig:template3}.
\begin{figure}[h]
\begin{center}
\begin{tabular}{cc}
\includegraphics[height=3.2cm]{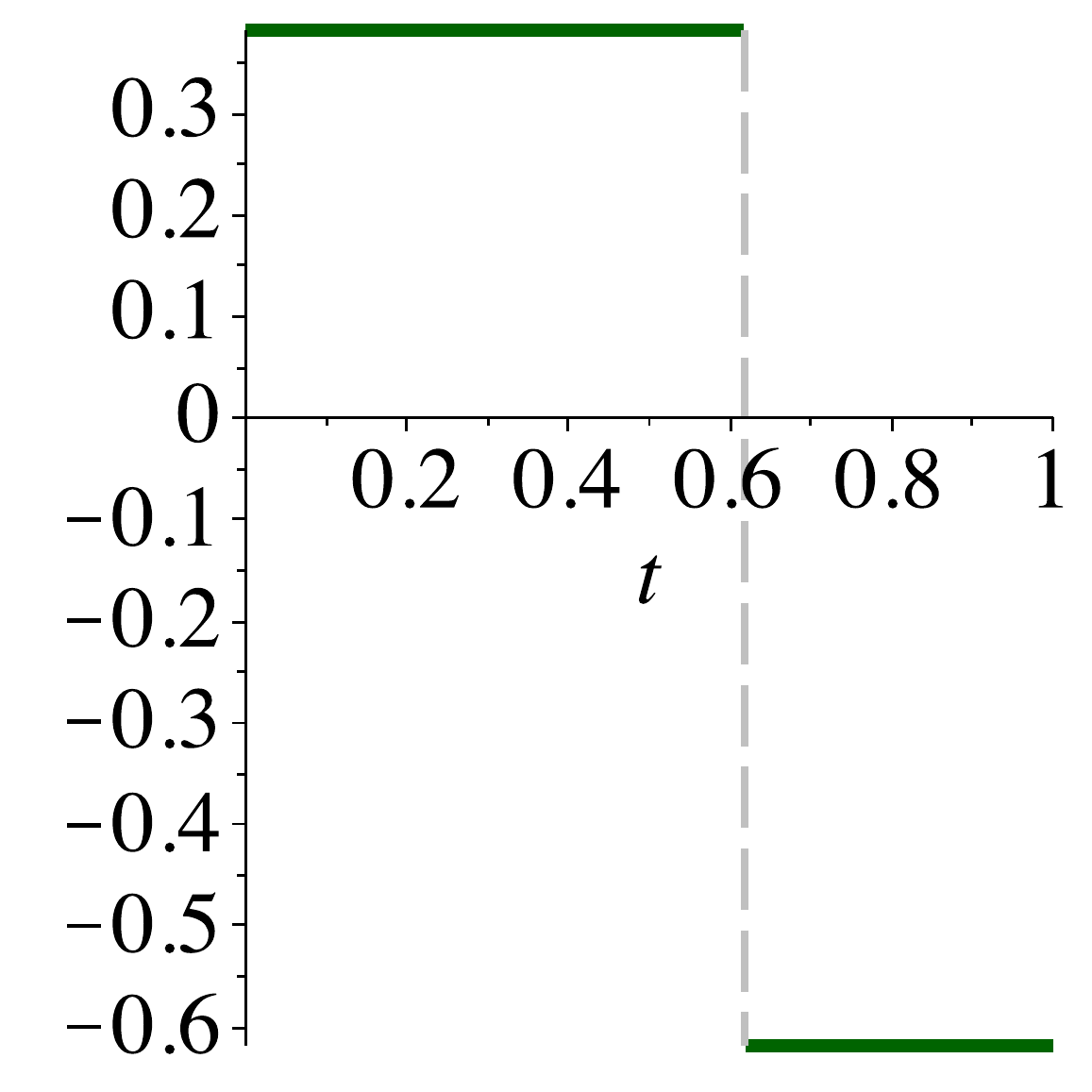} &
\includegraphics[height=3.2cm]{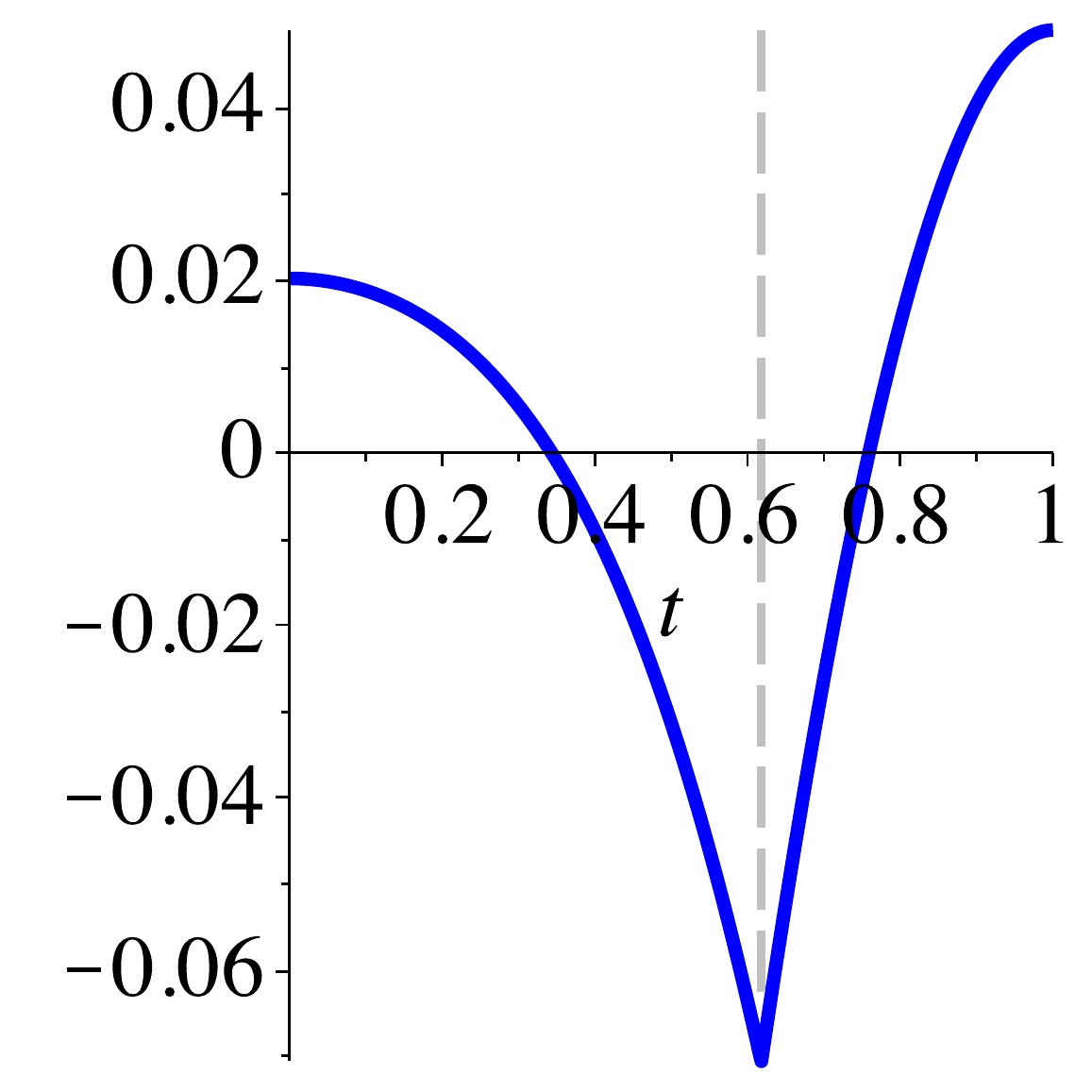} \\
\includegraphics[height=3.2cm]{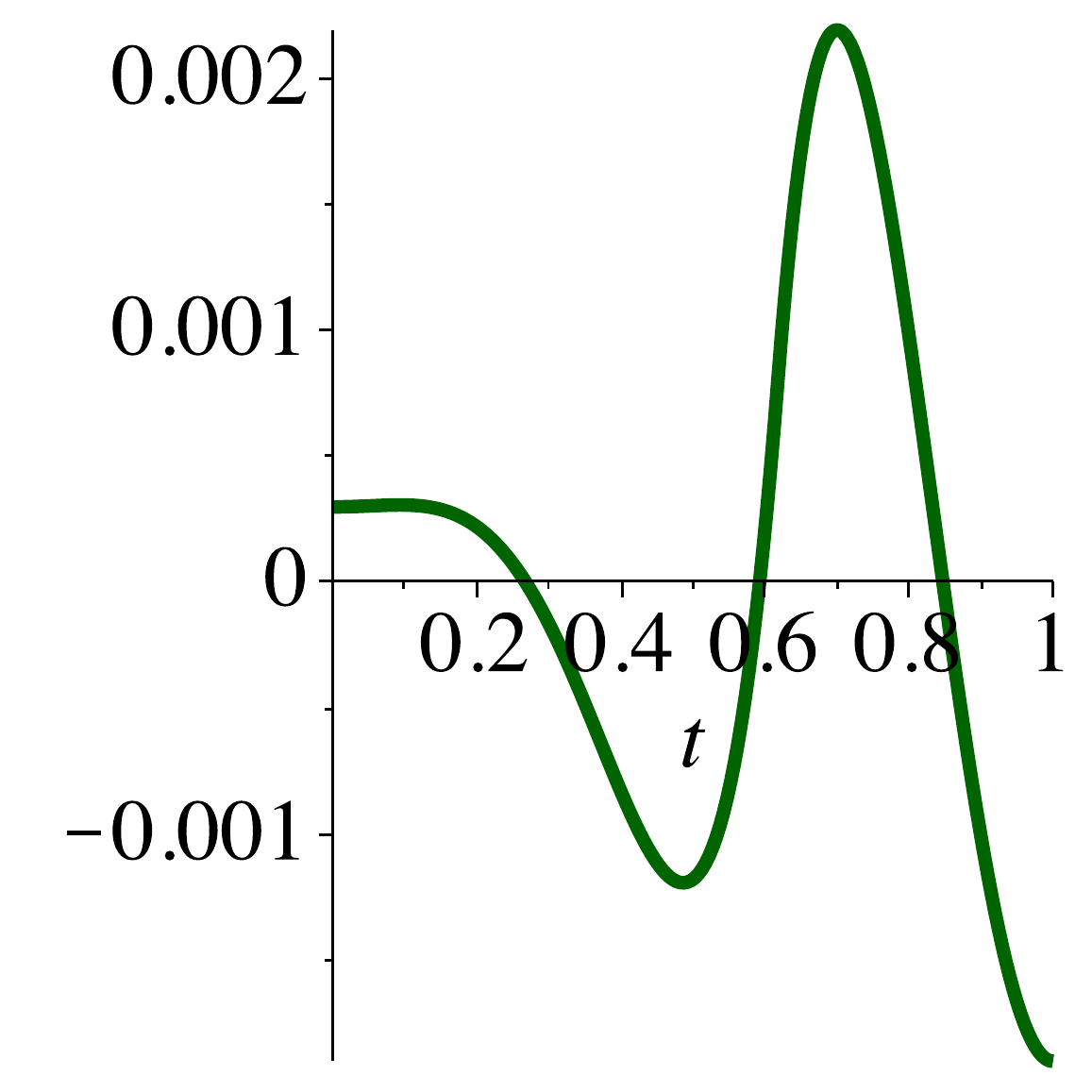} &
\includegraphics[height=3.2cm]{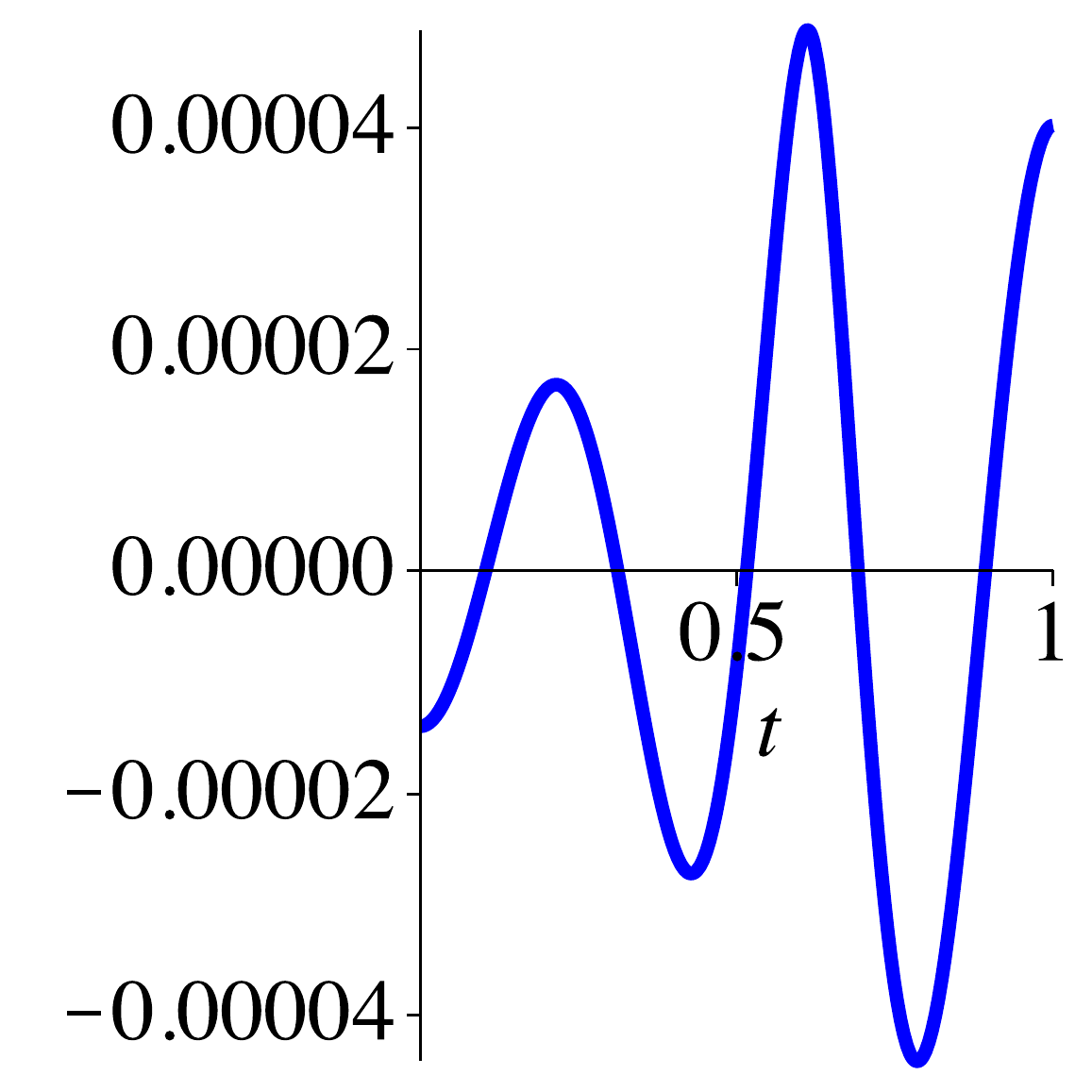}\\
\end{tabular}
\end{center}
\vspace*{-.3cm}
\caption{\emph{Template functions $\widetilde{\tau}^*_m (\cdot;x_0)$
for the generalized Cram\'er-von Mises statistics; here
$x_0=\frac{\sqrt{5}-1}2$, the reciprocal of the golden ratio.}}
\label{fig:template3}
\end{figure}

From the definitions, we obtain the following invariance properties.
\begin{lm}
The template functions satisfy 
\begin{equation}\label{cvm-invariance}
    \widetilde{\tau}_m(1-t;1-x) = (-1)^m \widetilde{\tau}_m(t;x) 
    \quad\text{and}\quad
    \widetilde{\tau}^*_m(1-t;1-x) = (-1)^m \widetilde{\tau}^*_m(t;x).
\end{equation}
The generalized CvM statistics are invariant under the reflection of 
the data $X_i\mapsto 1-X_i$.
\end{lm}

\begin{proof}[Proof of Lemma \ref{lm:hm-cvm}]
Write
\[
    \widetilde{\tau}_m(t;x) 
    = \frac{(x-t)^{m-1}}{(m-1)!}\,\Ind_{\{0\le t\le x\}} 
    - g_m(t;x),
\]
where
\[
    g_m(t;x) 
    = 2^{m-1} \left( \b_m\left(\sfrac{t+x}{2}\right)
    + (-1)^m \b_m\left(\sfrac{t-x}{2}\right) \right).
\]
Let $g_0(t;x)=x$.  Then, $g_m(t;x)$ satisfies, for $m\ge2$,
\begin{equation}\label{gm-tx}
    \begin{split}
		g_m(t;x)=
	    \begin{cases}\displaystyle
	        \int_0^x g_{m-1}(t;v)\dd v, & \mbox{if $m$ is odd}, \\
	        \displaystyle
	        \int_0^x g_{m-1}(t;v)\dd v +\frac{(1-t)^m}{m!}
	        \\ \displaystyle\qquad
			-\int_0^1 (1-v)g_{m-1}(t;v)\dd v, & \mbox{if $m$ is even}.
	    \end{cases}
    \end{split}
\end{equation}
To solve this recurrence relation, we write 
\[
    g_m(t;x) 
    = \sum_{0\le k\le m} g_{m,k}(t) x^k 
    \qquad(m\ge 1).
\]
Then we have $g_{1,1}=1$, and 
\[
    g_{m,k} 
    = \frac{g_{m-1,k-1}}{k}
    = \cdots 
    = \frac{g_{m-k,0}}{k!}
    \qquad (m\ge 2; k\ge 0).
\]
Thus it suffices to examine the sequence $c_m := g_{m,0}$ in more
detail. By \eqref{gm-tx}, we see that $c_m = 0$ if $m$ is odd, and
\[
    c_m 
    = \frac{(1-t)^m}{m!} 
    - \sum_{0\le j<m}\frac{c_{m-1-j}}{(j+2)!}
    \qquad (m\ge 2),
\]
if $m$ is even, with $c_0 := 1$. This sequence is easily solved by
considering the generating function $c(z) := \sum_{m\ge0}c_{2m} z^m$,
which satisfies
\[ 
    c(z) = \frac{\sqrt{z}(e^{(2-t)\sqrt{z}}
    +e^{t\sqrt{z}})}{e^{2\sqrt{z}}-1},
\]
and it follows that for even $m\ge0$
\[
    c_m = 2^{m-1}\left(\b_m\Lpa{1-\sfrac t2}
    +\b_m\Lpa{\sfrac t2}\right).
\]
Note that this holds also for odd $m$ for which $c_m=0$. By the
known relation 
\begin{align}\label{Bmx-dual}
    \b_m(1-x) = (-1)^m \b_m(x), 
\end{align}
we see that $c_m = 2^{m}\,\b_m\left(\sfrac{t}{2}\right)$ when $m$ is 
even. Consequently,
\[
    g_m(t;x) = \sum_{0\le j\le \left\lfloor m/2 \right\rfloor}
    \frac{2^{2j}x^{m-2j}}{(m-2j)!}\,
    \b_{2j}\left(\sfrac t2\right) \qquad (m\ge 0).
\]
This convolution sum can be simplified as follows.
\begin{align}
    g_m(t;x)
    &= [z^m] e^{xz}\left(\frac{ze^{tz}}{e^{2z}-1}
    +\frac{z e^{-tz}}{1-e^{-2z}}\right)
    \nonumber \\
    &= [z^m] \frac{ze^{xz} \left(e^{tz}
    + e^{(2-t)z} \right)}{e^{2z}-1} \nonumber \\
    &= 2^{m-1} \left(\b_m\left(\sfrac{x+t}2\right)
    +\b_m \left(1+\sfrac{x-t}2\right)\right) \nonumber \\
    &= 2^{m-1} \left(\b_m\left(\sfrac{x+t}2\right)
    +(-1)^m \b_m\left(\sfrac{t-x}2\right)\right),
\label{gmtx-Bm}
\end{align}
which completes the proof of the lemma.
\end{proof}

\subsection{The limiting GCvM processes: KL expansions and covariance kernel}
\label{subsec:kl-cov-cvm}

In this section, we examine the limiting distribution of the GCvM
statistic under the null hypothesis $H_0$. The arguments are almost
identical to those used for the GW statistics in Sec.~\ref{subsec:kl-cov-w}; thus some details are omitted.

Again the Wiener process on $[0,1]$ and the Brownian bridge are 
denoted by $W(x)$ and $B(x)$, respectively. Using the template
$\widetilde{\tau}_m(t;x)$ given in \eqref{tau_m-cvm}, we define
the limiting GCvM processes
\begin{align*}
    & \widetilde{B}^{[m]}(x)
    = \int_0^1 \widetilde{\tau}_m(t;x) \dd W(t)
    = \int_0^1 \widetilde{\tau}_m(t;x) \dd B(t), \\
    & 
    \widetilde{B}^{*[m]}(x)
    = \int_0^1 \widetilde{\tau}^*_m(t;x) \dd W(t)
    = \int_0^1 \widetilde{\tau}^*_m(t;x) \dd B(t).
\end{align*}
Since the Brownian bridge $B(\cdot)$ is of the class $C^0$ on
$[0,1]$, we have by induction that $\widetilde{B}^{[m]}(\cdot)$ and 
$\widetilde{B}^{*[m]}(\cdot)$ are of the class $C^{m-1}$ on $[0,1]$.

\begin{lm}
Let $\mu$ be any finite measure on $[0,1]$. Then, as $N\to\infty$,
$\widetilde{B}^{[m]}_{N}(\cdot) \stackrel{d}{\to}
\widetilde{B}^{[m]}(\cdot)$ and $\widetilde{B}^{*[m]}_{N}(\cdot)
\stackrel{d}{\to} \widetilde{B}^{*[m]}(\cdot)$ in $L^2([0,1],\mu)$
and in $L^\infty([0,1])$, and $\omega^{[m]}_{N}\stackrel{d}{\to}
\omega^{[m]}$ and $\omega^{*[m]}_{N}\stackrel{d}{\to} \omega^{*[m]}$,
where
\[
    \omega^{[m]} = \int_0^1 \widetilde{B}^{[m]}(x)^2 \dd x
    \quad\text{and}\quad
    \omega^{*[m]} = \int_0^1 \widetilde{B}^{*[m]}(x)^2 \dd x.
\]
\end{lm}
\begin{proof}
The same as that for Lemma \ref{lm:Um}.
\end{proof}

We now derive the covariance functions of the continuous centered
Gaussian processes $\widetilde{B}^{[m]}(\cdot)$ on $[0,1]$.
\begin{thm} For $m\ge1$
\label{thm:cov-cvm}
\begin{align}
    \Cov\lpa{\widetilde{B}^{[m]}(x),\widetilde{B}^{[m]}(y)}
    &= -2^{2m-1} \left( \b_{2m} \Lpa{\sfrac{x+y}2}
    + (-1)^m \b_{2m} \Lpa{\sfrac{|x-y|}2} \right) \nonumber \\
    &= \displaystyle
    2\sum_{k\ge 1} \frac{\cos\lpa{k\pi x-\frac{1}{2}m\pi}
    \cos\lpa{k\pi y-\frac{1}{2}m\pi}}{(k \pi)^{2m}},
\label{Kmxy2-cvm}
\end{align}
\begin{equation}
\label{Kmxy2-cvm*}
\begin{aligned}
    \Cov\lpa{\widetilde{B}^{*[m]}(x),\widetilde{B}^{*[m]}(y)}
    &=
    \Cov\lpa{\widetilde{B}^{[m]}(x),\widetilde{B}^{[m]}(y)} \\
    \displaystyle
    &\quad
    - 2\sum_{1\le k<m} \frac{\cos\lpa{k\pi x-\frac{1}{2}m\pi}
    \cos\lpa{k\pi y-\frac{1}{2}m\pi}}{(k \pi)^{2m}}.
\end{aligned}
\end{equation}
In particular,
\begin{align*}
    & \Var\lpa{\widetilde{B}^{[m]}(x)} 
    = -2^{2m-1} \left( \b_{2m}(x) + (-1)^m \b_{2m}(0) \right), \\
    & \Var\lpa{\widetilde{B}^{*[m]}(x)} 
    = \Var\lpa{\widetilde{B}^{[m]}(x)} 
    - 2\sum_{1\le k<m,\,k:\rm even} \frac{1}{(k \pi)^{2m}}.
\end{align*}
\end{thm}

The covariance functions (\ref{Kmxy2-cvm}) for $m=1,2$ are given, 
respectively, by 
\begin{align*}
    \Cov\lpa{\widetilde{B}^{[1]}(x),\widetilde{B}^{[1]}(y)}
    &= x\wedge y - xy, \\
    \Cov\lpa{\widetilde{B}^{[2]}(x),\widetilde{B}^{[2]}(y)}
    &= \sfrac{1}{45} -\sfrac{x^2+y^2}{6} 
    +\sfrac{(x\vee y)^3 +3(x\wedge y)xy}{6} 
    -\sfrac{x^4+6 x^2 y^2+y^4}{24}.
\end{align*}

\begin{proof}[Proof of Theorem \ref{thm:cov-cvm}]
The generating function of $g_m(t;x)$ in \eqref{gmtx-Bm} is 
\begin{align}\label{Gztx}
    G(z;t,x) 
    = \sum_{m\ge 1}g_m(t;x) z^{m-1} 
    = \frac{e^{xz}\left(e^{tz}
    + e^{(2-t)z}\right)}{e^{2z}-1} -\frac{1}{z},
\end{align}
and
\[
    H(z;t,x) = \sum_{m\ge 1} \widetilde{\tau}_m(t;x) z^{m-1} 
    = e^{(x-t)z}\Ind_{\{0\le t\le x\}} - G(z;t,x).
\]
Thus the covariance function of $\widetilde{B}^{[m]}(\cdot)$ can be 
represented as 
\begin{align*}
    K_m(x,y)
    &:= \Cov\lpa{\widetilde{B}^{[m]}(x),\widetilde{B}^{[m]}(y)} \\
    & = \int_0^1 \widetilde{\tau}_m(t;x)\widetilde{\tau}_m(t;y)\dd t 
    = [z^{m-1}w^{m-1}] J(z,w;x,y),
\end{align*}
where
\begin{align*}
    J(z,w;x,y)
    &= \int_0^1 H(z;t,x)H(w;t,y)\dd t \\
    &= \int_0^{x\wedge y} e^{(x-t)z+(y-t)w}\dd t 
    - \int_0^x e^{(x-t)z} G(w;t,y)\dd t \\
    &\quad
    - \int_0^y e^{(y-t)w} G(z;t,x)\dd t 
    + \int_0^1 G(z;t,x)G(w;t,y)\dd t.
\end{align*}
Let $\Delta := x+y$. Straightforward calculations using \eqref{Gztx} 
yield
\begin{align*}
    J(z,w;x,y)
    &= \frac1{z-w}\left(\frac{e^{w\Delta}}{e^{2w}-1} 
    - \frac{e^{z\Delta}}{e^{2z}-1} \right) - \frac1{zw} \\
    &\qquad
    + \begin{cases}\displaystyle
        \frac1{z+w}\left(\frac{e^{(2-\delta)z}}{e^{2z}-1} 
        + \frac{e^{\delta w}}{e^{2w}-1}\right),
        & \mbox{if }\delta := y-x>0, \\
        \displaystyle
        \frac1{z+w}\left(\frac{e^{\delta z}}{e^{2z}-1} 
        + \frac{e^{(2-\delta)w}}{e^{2w}-1}\right),
        & \mbox{if }\delta := x-y>0.
    \end{cases}
\end{align*}
It follows that 
\begin{align*}
    & [z^{m-1}w^{m-1}]\frac1{z-w}
    \left(\frac{e^{w\Delta}}{e^{2w}-1} 
    - \frac{e^{z\Delta}}{e^{2z}-1} \right) - \frac1{2zw} \\
    &\qquad
    = [z^{m-1}w^{m-1}]
    \sum_{\ell\ge2} \b_{\ell}\left(\sfrac\Delta2\right)\, 
    2^{\ell-1}\frac{w^{\ell-1}-z^{\ell-1}}{z-w}
    \\
    &\qquad= -\b_{2m}\left(\sfrac\Delta2\right)\,2^{2m-1}.
\end{align*}
Similarly, by \eqref{Bmx-dual},
\begin{align*}
    & [z^{m-1}w^{m-1}]\frac1{z+w}
    \left(\frac{e^{(2-\delta)z}}{e^{2z}-1} 
    + \frac{e^{\delta w}}{e^{2w}-1}\right) - \frac1{2zw} \\
    &\qquad
    = [z^{m-1}w^{m-1}]
    \sum_{\ell\ge2} \b_{\ell}\left(\sfrac\delta2\right)\,
    2^{\ell-1}\frac{w^{\ell-1}-(-z)^{\ell-1}}{z+w}
    \\
    &\qquad= (-1)^{m-1} \b_{2m}\left(\sfrac\delta2\right)\,2^{2m-1},
\end{align*}
when $x<y$. The same expression also holds when $x>y$ provided we
replace $\delta$ by $x-y$. Thus in either case
\[
    K_m(x,y) = -2^{2m-1}
    \left(\b_{2m} \left(\sfrac{x+y}2\right)
    + (-1)^m\b_{2m} \left(\sfrac{|x-y|}2\right) \right).
\]
Note that this identity implies another one by its decomposition
\begin{align}
    &\int_0^{x\wedge y} 
    \frac{(x-t)^{m-1}(y-u)^{m-1}} {(m-1)!^2}\dd t 
    - \int_0^x \frac{(x-t)^{m-1}}{(m-1)!}\, g_m(t;y) \dd t \nonumber \\
    &\qquad
    - \int_0^y \frac{(y-t)^{m-1}}{(m-1)!}\, g_m(t;x) \dd t 
    + \int_0^1 g_m(t;x)g_m(t;y) \dd t \nonumber \\
    &= -2^{2m-1}
    \left(\b_{2m} \Lpa{\sfrac{x+y}2}
    +(-1)^m\b_{2m} \left(\sfrac{|x-y|}2\right)\right).
\label{Km-id}
\end{align}
The proof of \eqref{Kmxy2-cvm*} relies on the relations
\begin{equation*}
    \int_0^1 \widetilde{\tau}_m(t;x)
    \cos\Bigl(k\pi y-\sfrac{m\pi}{2} \Bigr)\cos(k\pi t) \dd t
    = \frac{\cos\bigl(k\pi x -\frac{m\pi}{2}\bigr)
    \cos\bigl(k\pi y -\frac{m}{2}\pi\bigr)}{(k\pi)^m}
\end{equation*}
and
\[
    \begin{split}	
    &\int_0^1 
    \cos\Bigl(k\pi x-\sfrac{m\pi}{2} \Bigr)
    \cos(k\pi t)\cos\Bigl(k\pi y-\sfrac{m\pi}{2} \Bigr)
    \cos(k\pi t) \dd t \\
    &\qquad = \sfrac{1}{2} \cos\Bigl(k\pi x -\sfrac{m\pi}{2}\Bigr) 
    \cos\Bigl(k\pi y -\sfrac{m\pi}{2}\Bigr).
    \end{split}
\]
The former follows from the same generating function approach used 
above and the latter is trivial.
\end{proof}

Starting from the expansion
\begin{equation}\label{cvm-eigen_expansion}
    \widetilde{\tau}_m(t;x) 
    = 2 \sum_{k\ge 1} 
    \frac{1}{(\pi k)^m} \cos\Bigl(k\pi x - 
    \sfrac{m\pi}{2}\Bigr) \cos(k\pi t),
\end{equation}
and taking the stochastic integral $\int_0^1 \dd B(t)$ on both sides, 
we obtain formally the KL expansion
\begin{equation}\label{KL-cvm}
    \widetilde{B}^{[m]}(x)
    = \sqrt{2}\sum_{k\ge 1} 
    \frac{\cos(k\pi x-\frac{m\pi}{2})}{(k\pi)^m} \,\xi_k,
\end{equation}
where
\begin{equation}\label{xi-cvm}
    \xi_k 
    = \sqrt{2} \int_0^1 \cos(k\pi t) \dd B(t) 
    \sim N(0,1)\ \ \mbox{i.i.d.}
\end{equation}
These expansions hold true, again justified by Mercer's theorem.

\begin{thm}\label{thm:KL-cvm}
For $m\ge 1$, \eqref{KL-cvm} with $\xi_j$ defined by \eqref{xi-cvm} 
holds uniformly in $x\in[0,1]$ with probability 1, and
\begin{equation}\label{KL2-cvm}
    \omega^{[m]} 
    := \int_0^1 \widetilde{B}^{[m]}(x)^2 \dd x 
    = \sum_{k\ge 1} \frac{\xi_k^2}{(k \pi)^{2m}}
\end{equation}
and
\begin{equation}\label{KL2-cvm*}
    \omega^{*[m]} 
    := \int_0^1 \widetilde{B}^{*[m]}(x)^2 \dd x 
    = \sum_{k\ge m} \frac{\xi_k^2}{(k \pi)^{2m}}
\end{equation}
converges with probability 1.
\end{thm}
The proof follows \emph{mutatis mutandis} that of Theorem 
\ref{thm:KL-w} and is omitted.

\subsection{Statistic in terms of samples}
\label{subsec:sample-cvm}

We now derive the samples representation of the generalized 
Cram\'er-von Mises statistic. Since
\[
    \omega^{[m]}_{N}
    = N \int_0^1 \biggl( 
    \int_0^1 \widetilde{\tau}_m(t;x) \widetilde{\tau}_m(u;x) \dd x \biggr) 
    \dd F_{N}(t) \dd F_{N}(u),
\]
it suffices to evaluate the integral between the large parentheses.
The method of proof follows closely that of Theorem 
\ref{thm:cov-cvm}.

\begin{thm}\label{thm:sample-cvm}
The generalized Cram\'er-von Mises statistic is expressed
in terms of samples $X_i$'s as
\[
    \omega^{[m]}_{N}
    = \sfrac{1}{N} \sum_{1\le i,j\le N} 
    (-1)^{m-1} 2^{2m-1}
    \left( \b_{2m}\Lpa{\sfrac{X_i+X_j}{2}}
    +\b_{2m}\Lpa{\sfrac{|X_i-X_j|}{2}} \right),
\]
and
\[
    \omega^{*[m]}_{N}
    = \omega^{[m]}_{N} 
    - 2\sum_{\substack{1\le k<m  1\le i,j\le N}} 
    \frac{\cos\lpa{k\pi X_i-\frac{1}{2}m\pi}
    \cos\lpa{k\pi X_j-\frac{1}{2}m\pi}}{(k \pi)^{2m}},
\]
where $(2m)!\b_{2m}(x)$ denotes the Bernoulli polynomial of order 
$2m$.
\end{thm}

The statistics $\omega^{[m]}_{N}$ for $m=1,2$ are given, 
respectively, by 
\begin{align*}
    \omega^{[1]}_{N}
    &= \sfrac{1}{N} \sum_{1\le i,j\le N} \left(
    \sfrac{1}{3} - X_i\vee X_j 
	+ \sfrac{X_i^2 + X_j^2}{2} \right), \\
    \omega^{[2]}_{N}
    &= \sfrac{1}{N} \sum_{1\le i,j\le N} \left(
    \sfrac{1}{45} -\sfrac{X_i^2+X_j^2}{6} 
    +\sfrac{(X_i\vee X_j)^3+3 (X_i\wedge X_j) X_i X_j}{6} 
    \right.\\
	&\hspace*{4cm} \left.
	-\sfrac{X_i^4+6 X_i^2 X_j^2 +X_j^4}{24}
    \right).
\end{align*}

\begin{proof}
We consider the integral
\begin{align*}
    K_m^*(s,t)
    &:= \int_0^1 \widetilde{\tau}_m(s;x) 
	\widetilde{\tau}_m(t;x)\dd x \\
    &= \int_{s\vee t}^1 
    \frac{(x-s)^{m-1}(x-t)^{m-1}}{(m-1)!^2} \dd x
    - \int_s^1 \frac{(x-s)^{m-1}}{(m-1)!}\, g_m(t;x) \dd x \\
    &\quad
    - \int_t^1 \frac{(x-t)^{m-1}}{(m-1)!}\, g_m(s;x) \dd x
    + \int_0^1 g_m(s;x)g_m(t;x) \dd x.
\end{align*}
We will prove that
\[
    K_m^*(s,t) = (-1)^{m-1} 2^{2m-1}
    \left(\b_{2m}\left(\sfrac{s+t}2\right)
    +\b_{2m} \Lpa{\sfrac{|s-t|}2}\right).
\]
By \eqref{Bmx-dual} and \eqref{gmtx-Bm}, we have, by exchanging $x$ 
and $t$,
\[
    g_m(x;t) = 2^{m-1} \left(\b_m\left(\sfrac{x+t}2\right)
    +(-1)^m\b_m\left(\sfrac{x-t}2\right)\right),
\]
and  
\[
    \widetilde{g}_m(t;x) := g_m(1-t;1-x)
    = (-1)^m 2^{m-1} \left(\b_m\left(\sfrac{x+t}2\right)
    +\b_m\left(\sfrac{x-t}2\right)\right).
\]
These relations imply particularly that  
\begin{align}\label{gm-even}
    \widetilde{g}_m(t;x) = g(x;t) \qquad(m \mbox{ even}).
\end{align}
By the change of variables $x\mapsto 1-x$ and by the definition of
$\widetilde{g}_m(s;x)$, we see that
\begin{align*}
    K_m^*(1-s,1-t) 
    &= \int_0^{s\wedge t} 
    \frac{(s-x)^{m-1}(t-x)^{m-1}}{(m-1)!^2} \dd x
    - \int_0^{s} \frac{(s-x)^{m-1}}{(m-1)!}\, 
    \widetilde{g}_m(t;x) \dd x \\
    &\quad
    - \int_0^{t} \frac{(t-x)^{m-1}}{(m-1)!}\,
    \widetilde{g}_m(s;x) \dd x
    + \int_0^1 \widetilde{g}_m(s;x) \widetilde{g}_m(t;x) \dd x,
\end{align*}
which leads, by symmetry, \eqref{Km-id} and \eqref{gm-even}, to 
\[
    K_m^*(1-s,1-t)=K_m^*(s,t)
    =K_m(s;t) \quad(m \mbox{ even}).
\]

We are left with the case when $m$ is odd for which we have
\[
    \widetilde{g}_m(t;x) 
    = g(x;t) + 2^m\,\b_m\left(\sfrac{x+t}2\right),
\]
so that 
\[
    K_m^*(1-s,1-t) 
    = K_m^*(s,t) 
    = K_m(s,t) + L_m(s,t),
\]
where
\begin{align*}
    \sfrac{L_m(s,t)}2
    &:= 2^{m-1} \int_0^{s} \sfrac{(s-x)^{m-1}}{(m-1)!}\, 
    \b_m\left(\sfrac{x+t}2\right)\dd x \\
	&\qquad + 2^{m-1}
    \int_0^{t} \sfrac{(t-x)^{m-1}}{(m-1)!}\,
    \b_m\left(\sfrac{x+s}2\right)\dd x \\
    &\qquad
    + 4^{m-1} \int_0^1 \b_m\left(\sfrac{x+s}2\right)
    \b_m\left(\sfrac{x-t}2\right) \dd x\\
    &\qquad 
	+ 4^{m-1} \int_0^1 \b_m\left(\sfrac{x-s}2\right)
    \b_m\left(\sfrac{x+t}2\right) \dd x \\
    &=: L_{m,1}+L_{m,2}+L_{m,3}+L_{m,4}.
\end{align*}
We will prove that for odd $m$
\[
    L_m(s,t) 
    = 2^{2m} \,\b_{2m}\left(\sfrac{s+t}2\right).
\]
Indeed, this relation holds also for even $m$ but in the form
\[
    L_m(s,t) 
    = (-1)^{m-1} 2^{2m} \,
    \b_{2m}\left(\sfrac{s+t}2\right) \quad (m\ge 1).
\]
We will apply the Fourier series expansion
\begin{align}\label{Bm-sin}
    2^{m-1}\,\b_m(x)
    = (-1)^{\lceil \frac m2\rceil} 
    \sum_{j\ge1}\sfrac{\sin(2j\pi x)}{(j\pi)^m}
    \quad (m \mbox{ odd}).
\end{align}
which holds only for $0\le x\le1$ (for $0<x<1$ when $m=1$). Thus we 
cannot simply substitute the Fourier series into the integrals of
$L_m$. For that purpose, we split first the integrals at $x=t$ and 
use the relation
\[
    \b_m(-x) 
    = -\b_m(x) - \frac{x^{m-1}}{(m-1)!} 
    \quad (m \mbox{ odd})
\]
as follows.
\begin{align*}
    L_{m,3}
    &= 4^{m-1} \int_0^1 \b_m\left(\sfrac{x+s}2\right)
    \b_m\left(\sfrac{x-t}2\right)\dd x \\
    & = -4^{m-1} \int_0^t 
    \b_m\left(\sfrac{x+s}2\right)
    \b_m\left(\sfrac{t-x}2\right) \dd x \\
    &\quad
    -\underbrace{2^{m-1}\int_0^t \frac{(t-x)^{m-1}}
    {(m-1)!} \b_m\left(\sfrac{x+s}2\right) \dd x}_{=L_{m,2}} \\
    &\quad
    + 4^{m-1}
    \int_t^1 \b_m\left(\sfrac{x+s}2\right)
    \b_m\left(\sfrac{x-t}2\right)\dd x.
\end{align*}
Similarly, 
\begin{align*}
    L_{m,4}
    &= 4^{m-1} \int_0^1 \b_m\left(\sfrac{x-s}2\right)
    \b_m\left(\sfrac{x+t}2\right)\dd x \\
    & = -4^{m-1}\int_0^s 
    \b_m\left(\sfrac{s-x}2\right)
    \b_m\left(\sfrac{t+x}2\right) \dd x -L_{m,1} \\
    &\quad
    + 4^{m-1} \int_s^1 \b_m\left(\sfrac{x-s}2\right)
    \b_m\left(\sfrac{x+t}2\right)\dd x.
\end{align*}
Thus 
\begin{align*}
    \frac{L_m(s,t)}{2}
    :=&- 4^{m-1} \int_0^t \b_m\left(\sfrac{x+s}2\right)
    \b_m\left(\sfrac{t-x}2\right) \dd x\\
    &+ 4^{m-1} \int_t^1 \b_m\left(\sfrac{x+s}2\right)
    \b_m\left(\sfrac{x-t}2\right) \dd x\\
    &- 4^{m-1} \int_0^s \b_m\left(\sfrac{s-x}2\right)
    \b_m\left(\sfrac{t+x}2\right) \dd x\\
    &+ 4^{m-1} \int_s^1 \b_m\left(\sfrac{x-s}2\right)
    \b_m\left(\sfrac{x+t}2\right) \dd x.
\end{align*}
All parameters are now inside the unit interval, and we can then
apply the Fourier series expansion \eqref{Bm-sin}. It follows that
\[
    \sfrac{L_m(s,t)}2 := \sum_{j,\ell\ge 1}
    \sfrac{I(j,\ell;s,t)}{(j\pi)^m(\ell\pi)^m},
\]
where
\begin{align*}
    I(j,\ell;s,t)
    :=&-\int_0^t \sin(j\pi(x+s))\sin(\ell\pi (t-x))\dd x\\
    &+ \int_t^1 \sin(j\pi(x+s))\sin(\ell\pi (x-t))\dd x\\
    &-\int_0^s \sin(j\pi(s-x))\sin(\ell\pi (t+x))\dd x\\
    &+ \int_s^1 \sin(j\pi(x-s))\sin(\ell\pi (x+t))\dd x.
\end{align*}
It is then straightforward to show that 
\[
    I(j,\ell;s,t) =
    \begin{cases}
        0,                & \mbox{if }\ell\ne j \\
        \cos(j\pi (s+t)), & \mbox{if }\ell= j,
    \end{cases}
\]
implying that
\[
    L_m(s,t)
    = 2\sum_{j\ge1} \sfrac{\cos(j\pi (s+t))}{(j\pi)^{2m}}
    = 2^{2m}\,\b_{2m}\left(\sfrac{s+t}2\right)
\]
for odd $m$.
\end{proof}

\subsection{Power comparisons}
\label{subsec:power-cvm}

We gather here our Monte Carlo simulation results in Figures
\ref{fig:cvm-normal1}--\ref{fig:cvm-skew} for the power comparison of
the test statistics $\omega_N^{*[1]}=\omega_N^{[1]}$,
$\omega_N^{*[2]}$ and $\omega_N^{*[3]}$.
We adopt the same simulation settings as in the generalized AD
statistics in Sec.~\ref{subsec:finite}. The null hypothesis $H_0$ is
that $F$ is the normal distribution with zero mean and unit variance;
we assume that $N=100$ samples are available. The size of the tests
is 0.05, and the critical values are estimated by simulations. We
conducted Monte Carlo simulations with 100,000 replications.

We first assume that the true distribution is normal with mean $\mu$
and variance $\sigma^2$. The estimated power functions of
$\omega_N^{*[m]}$ for $m=1,2,3$ are summarized in Figures
\ref{fig:cvm-normal1} and \ref{fig:cvm-normal2}.

The pattern of the highest power variations exhibited in Figures
\ref{fig:cvm-normal1} and \ref{fig:cvm-normal2} is almost identical
to that of the generalized AD statistics in Figures \ref{fig:normal1}
and \ref{fig:normal2}, although the generalize AD is slightly
superior to the generalized CvM.

\begin{figure}[h]
\begin{center}
\begin{tabular}{cc}
\includegraphics[scale=0.25]{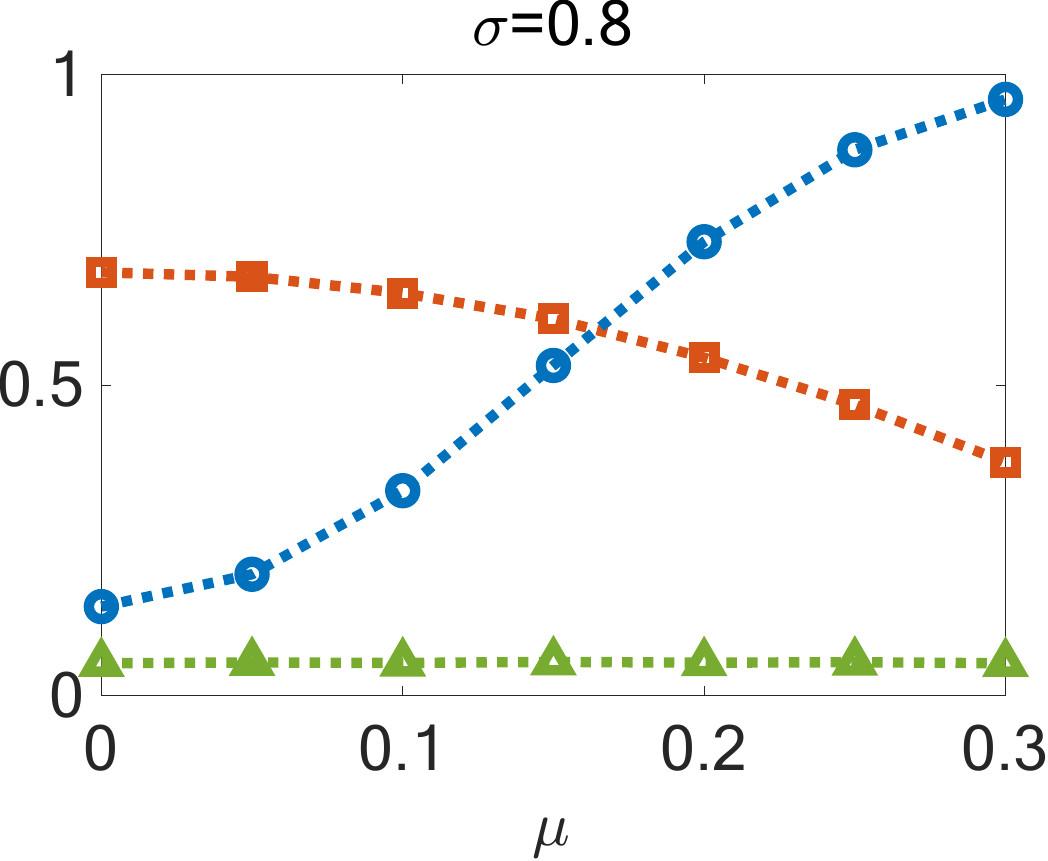} &
\includegraphics[scale=0.25]{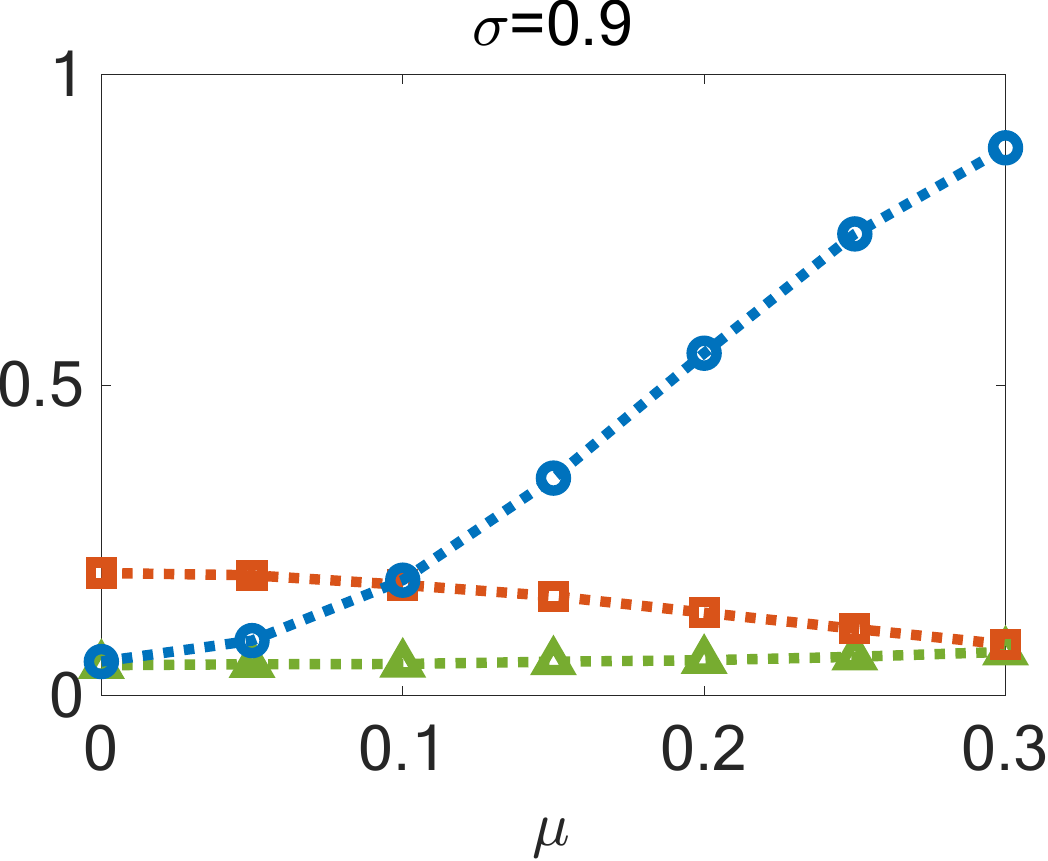} \\
\multicolumn{2}{c}{\includegraphics[scale=0.25]{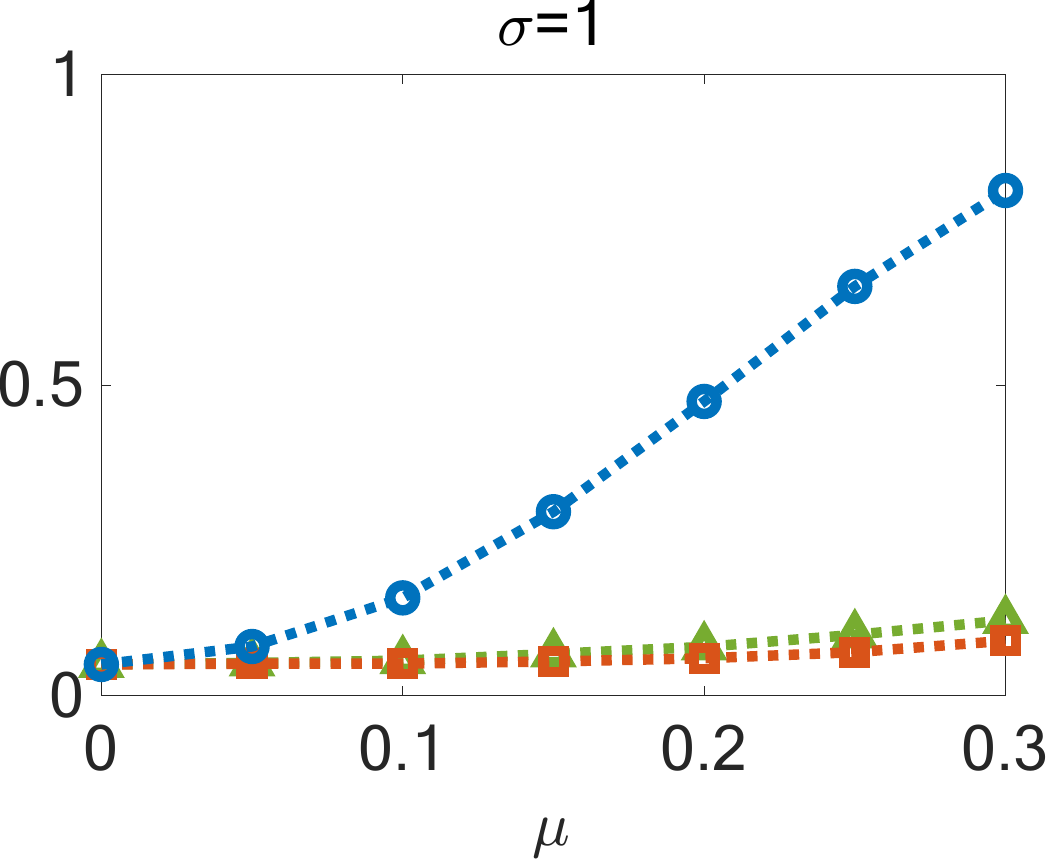}} \\
\includegraphics[scale=0.25]{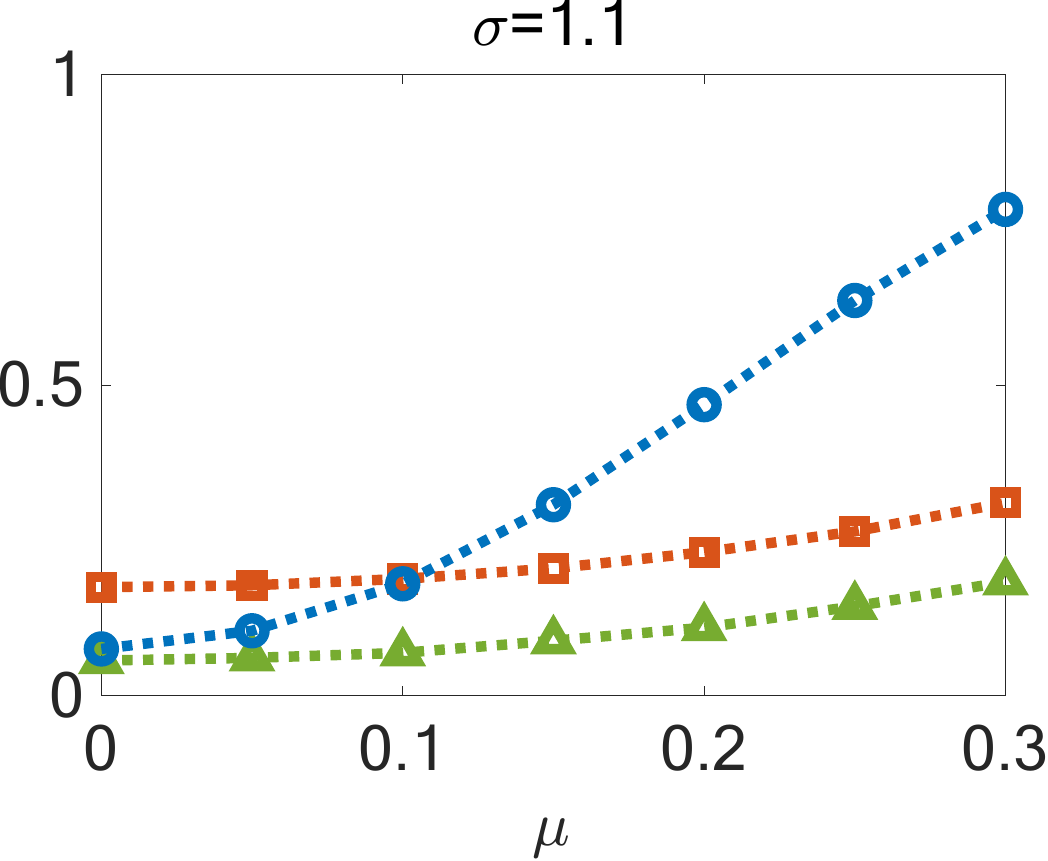} &
\includegraphics[scale=0.25]{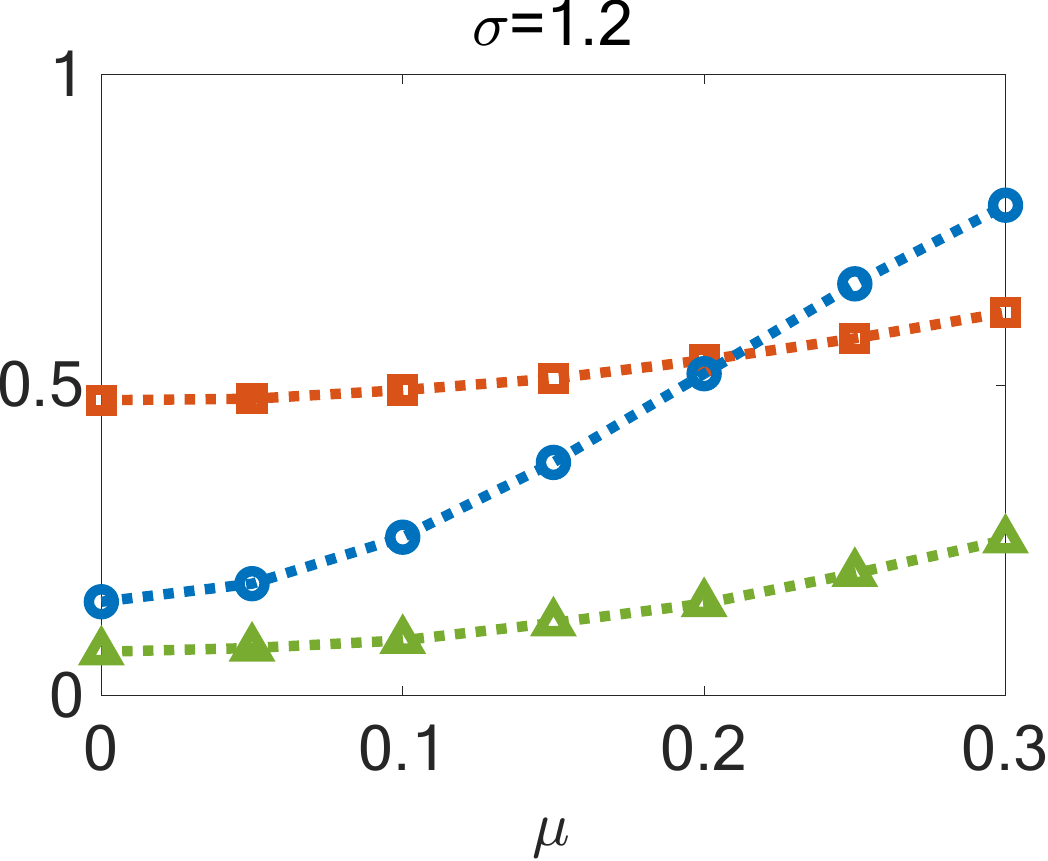}
\end{tabular}
\caption{Powers of $\omega^{[m]}_N$ for $\mu\in[0,0.3]$ and $\sigma\in\{0.8,0.9,1,1.1,1.2\}$.}
\smallskip
(\blueredgreen)
\label{fig:cvm-normal1}
\end{center}
\end{figure}

\begin{figure}
\begin{center}
\begin{tabular}{ccc}
\includegraphics[scale=0.25]{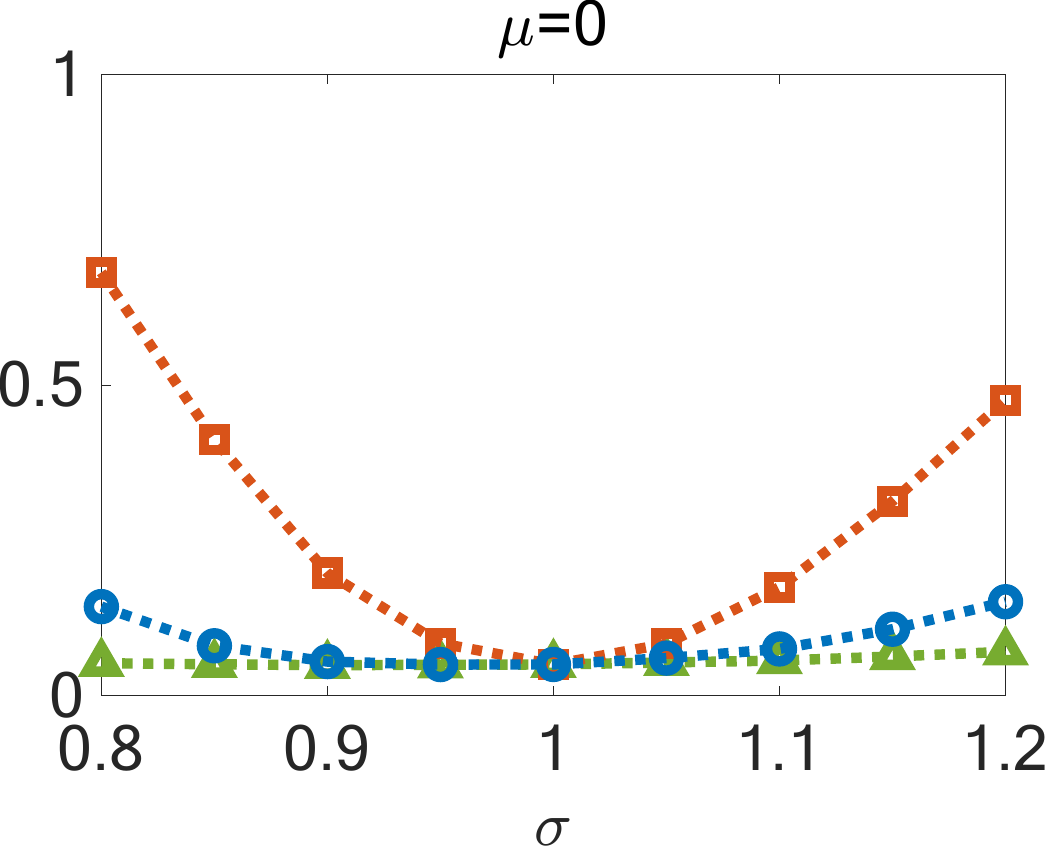} &
\includegraphics[scale=0.25]{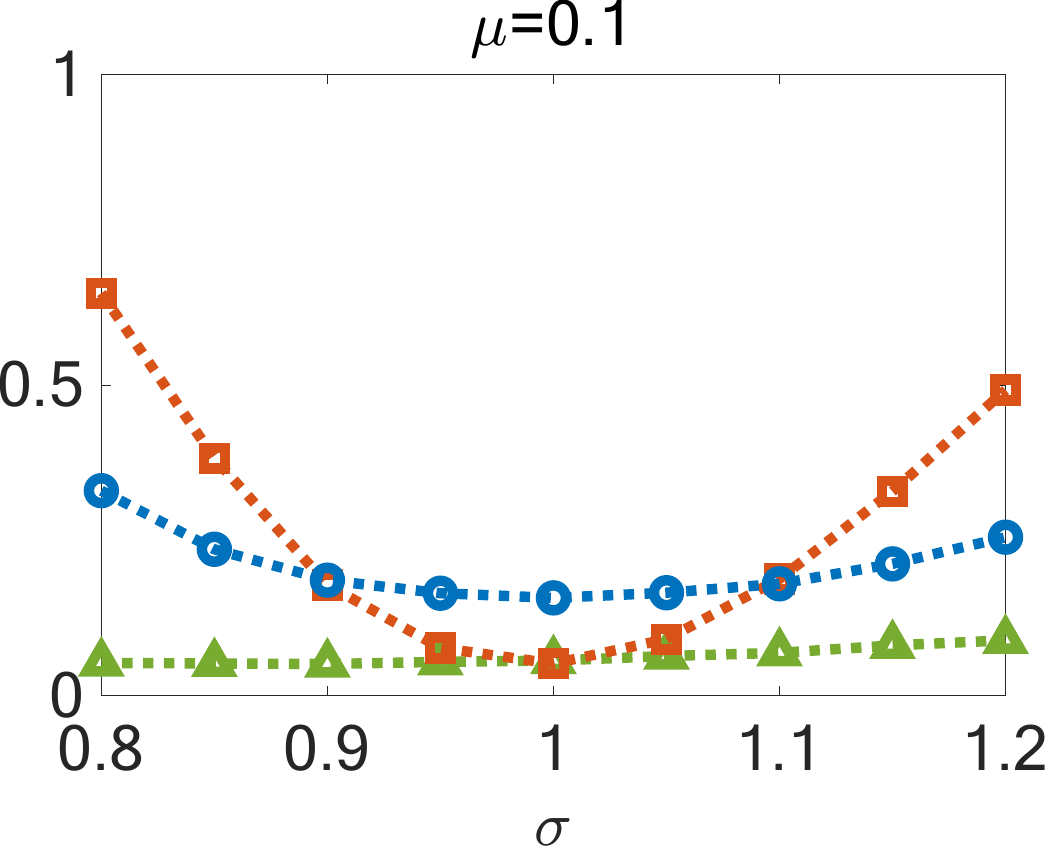} \\
\includegraphics[scale=0.25]{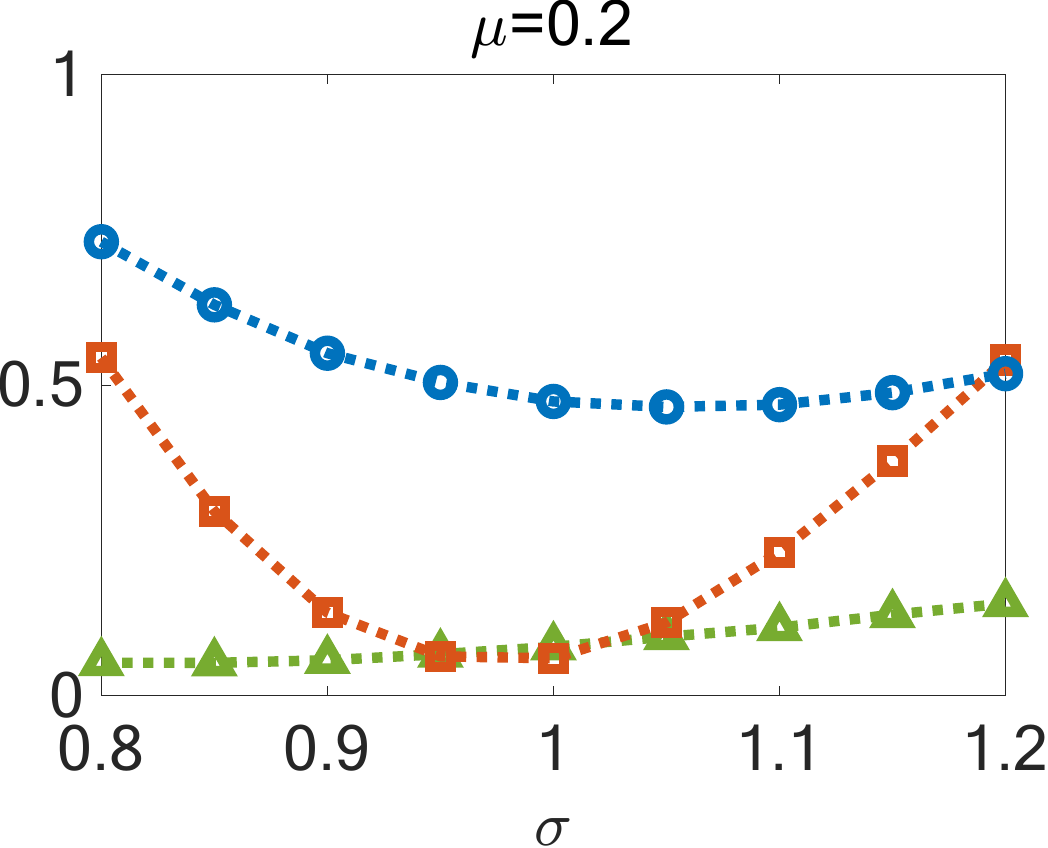} &
\includegraphics[scale=0.25]{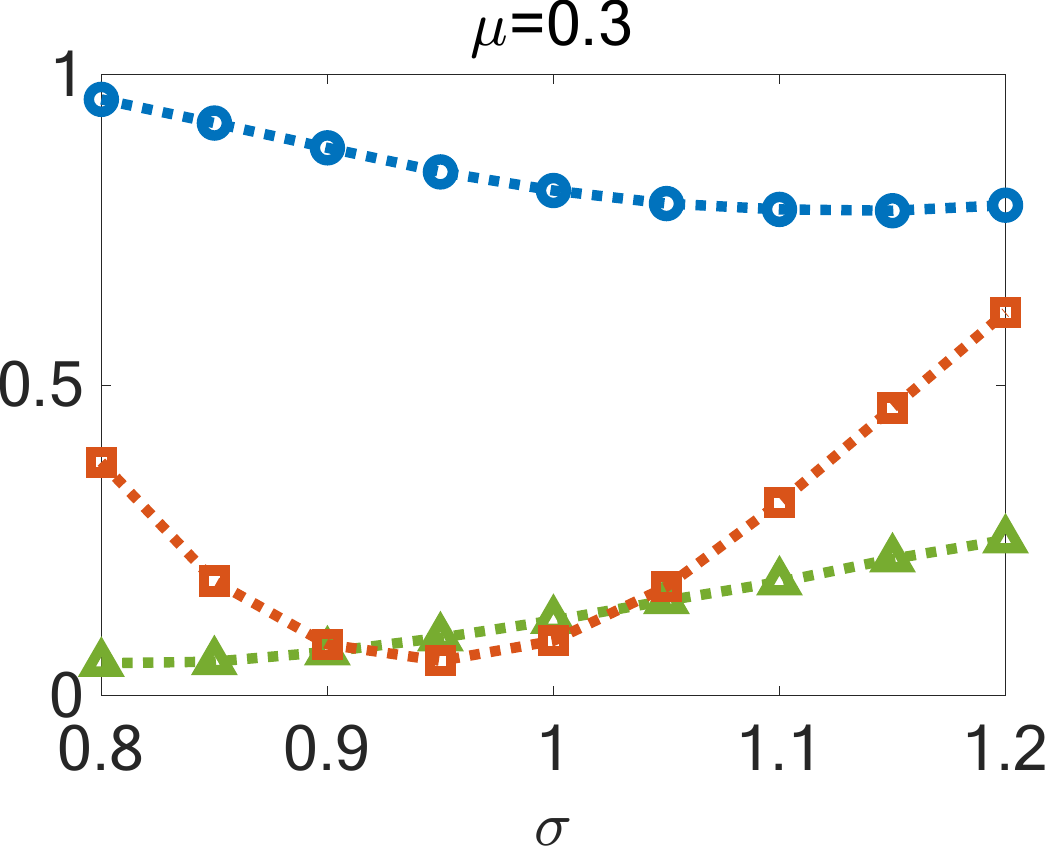}
\end{tabular}
\caption{Powers of $\omega^{[m]}_N$ for $\sigma\in[0.8,1.2]$ and $\mu\in\{0,0.1,0.2,0.3\}$.}
\smallskip
(\blueredgreen)
\label{fig:cvm-normal2}
\end{center}
\end{figure}

In the case of skew-normal distribution alternatives, Figure
\ref{fig:cvm-skew} shows the powers of $\omega_N^{*[m]}$ for
$m=1,2,3$ when the true density is (\ref{skew-normal}). We set the
null hypothesis to be the standard normal distribution which
corresponds to the case $\xi=0$, $\omega=1$, and $\alpha=0$ in
\eqref{skew-normal}. As in the $\alpha=0$ case, the highest power
patterns displayed in Figure \ref{fig:cvm-skew} are almost identical
to those of the generalized AD statistics in Figure \ref{fig:skew}.
Furthermore, there is a tendency that the powers of $\omega_N^{*[m]}$
in Figure \ref{fig:cvm-skew} are smaller than those of the
corresponding $A_N^{*[m]}$ in Figure \ref{fig:skew}. 

\begin{figure}
\begin{center}
\begin{tabular}{ccc}
\includegraphics[scale=0.21]{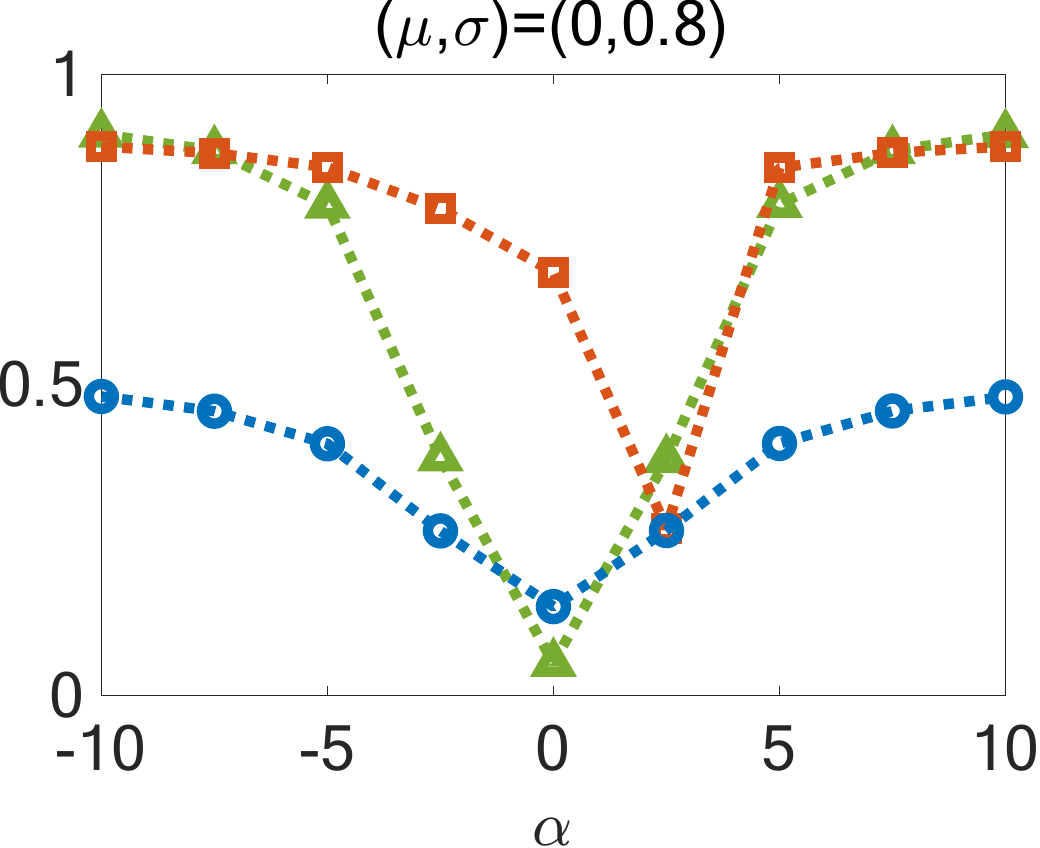} &
\includegraphics[scale=0.21]{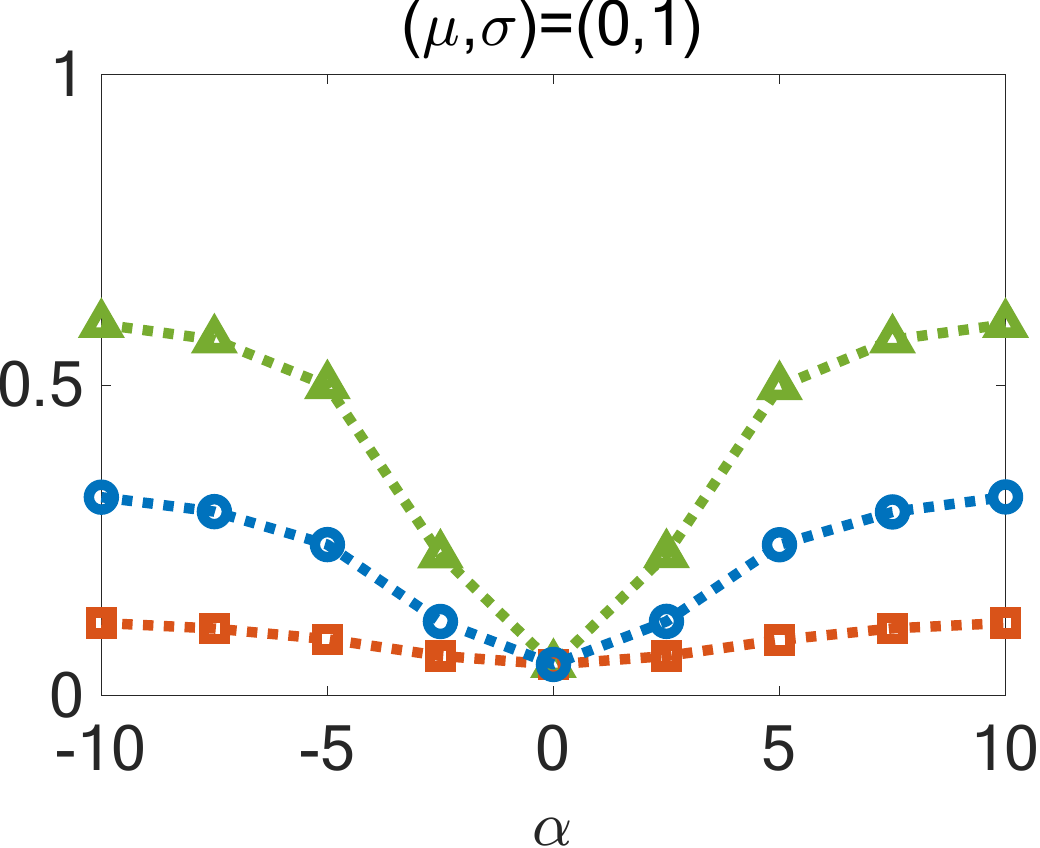} &
\includegraphics[scale=0.21]{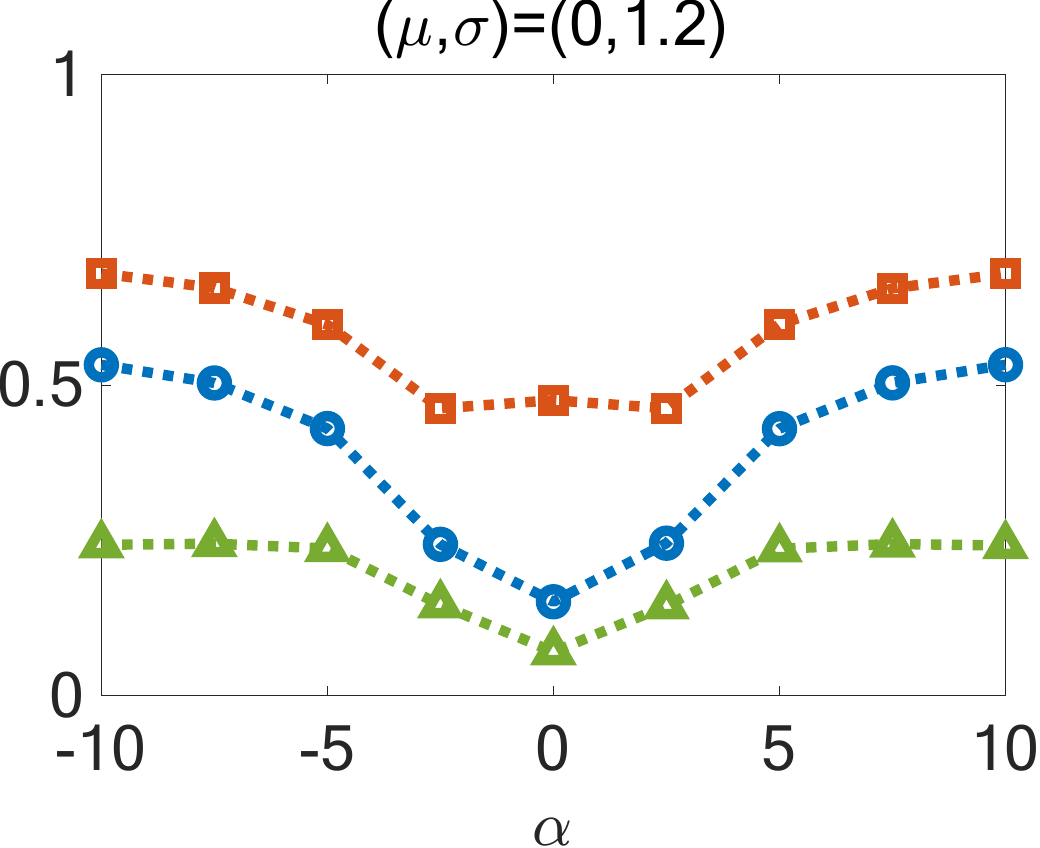} \\
\includegraphics[scale=0.21]{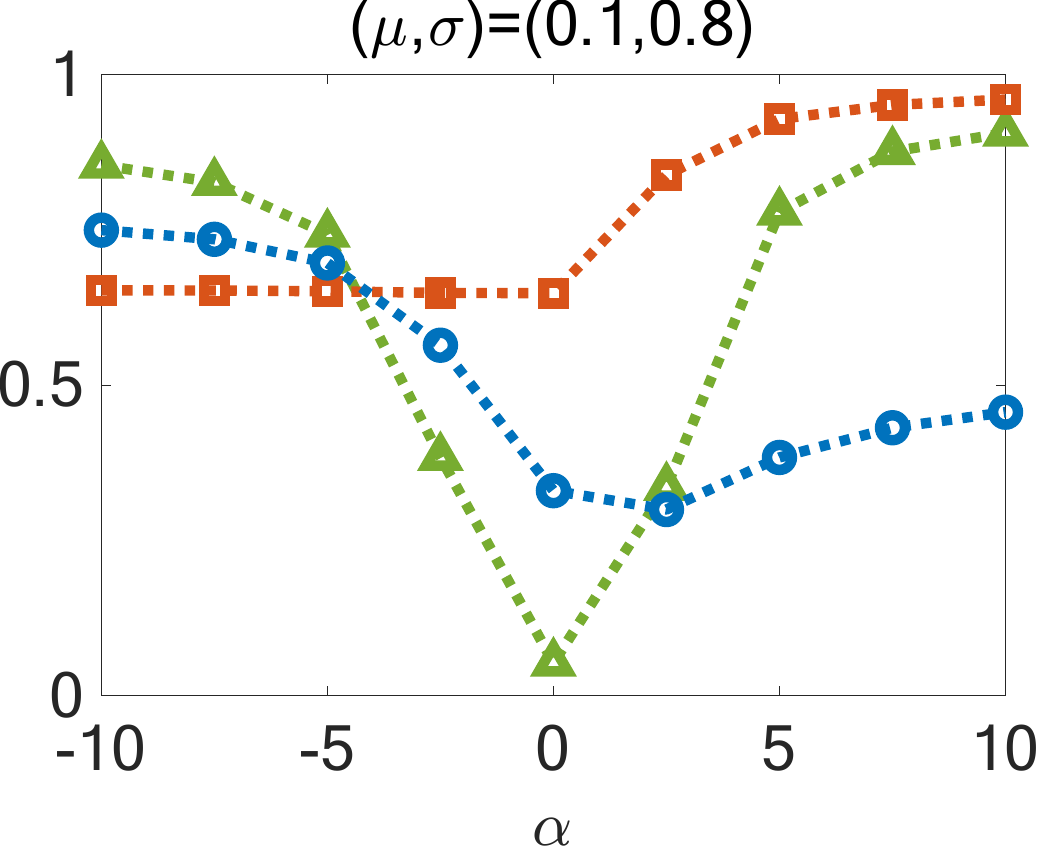} &
\includegraphics[scale=0.21]{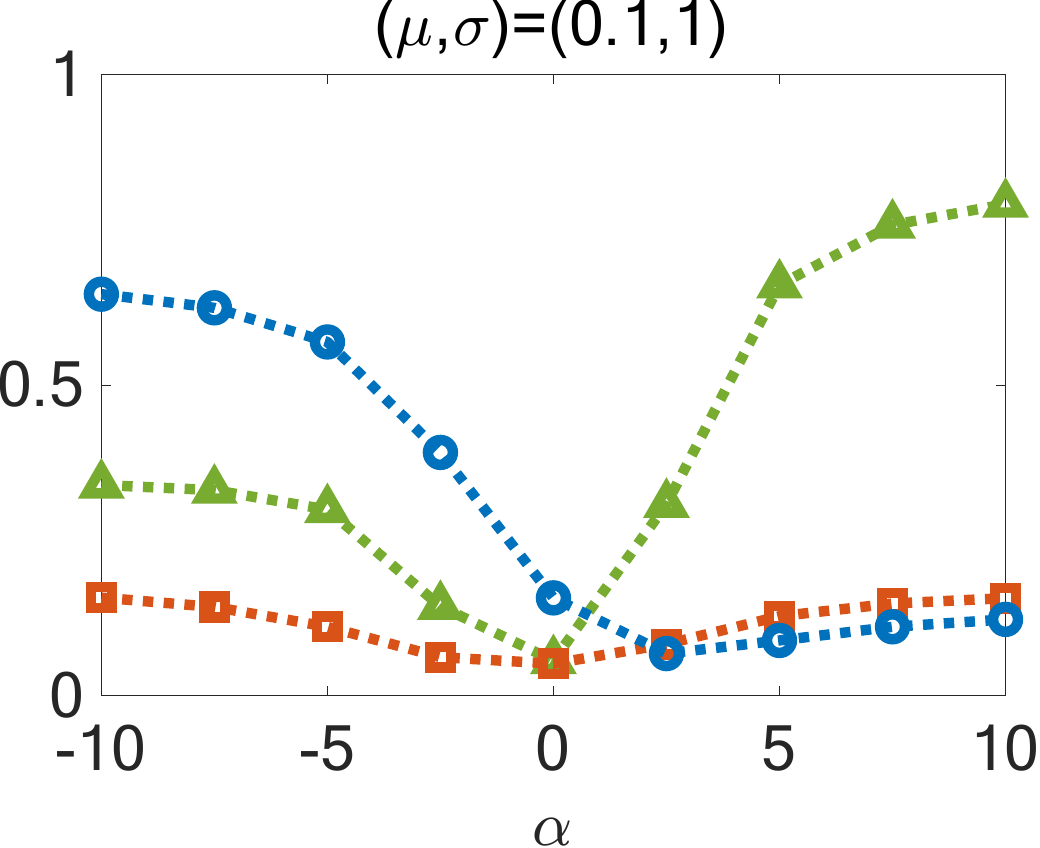} &
\includegraphics[scale=0.21]{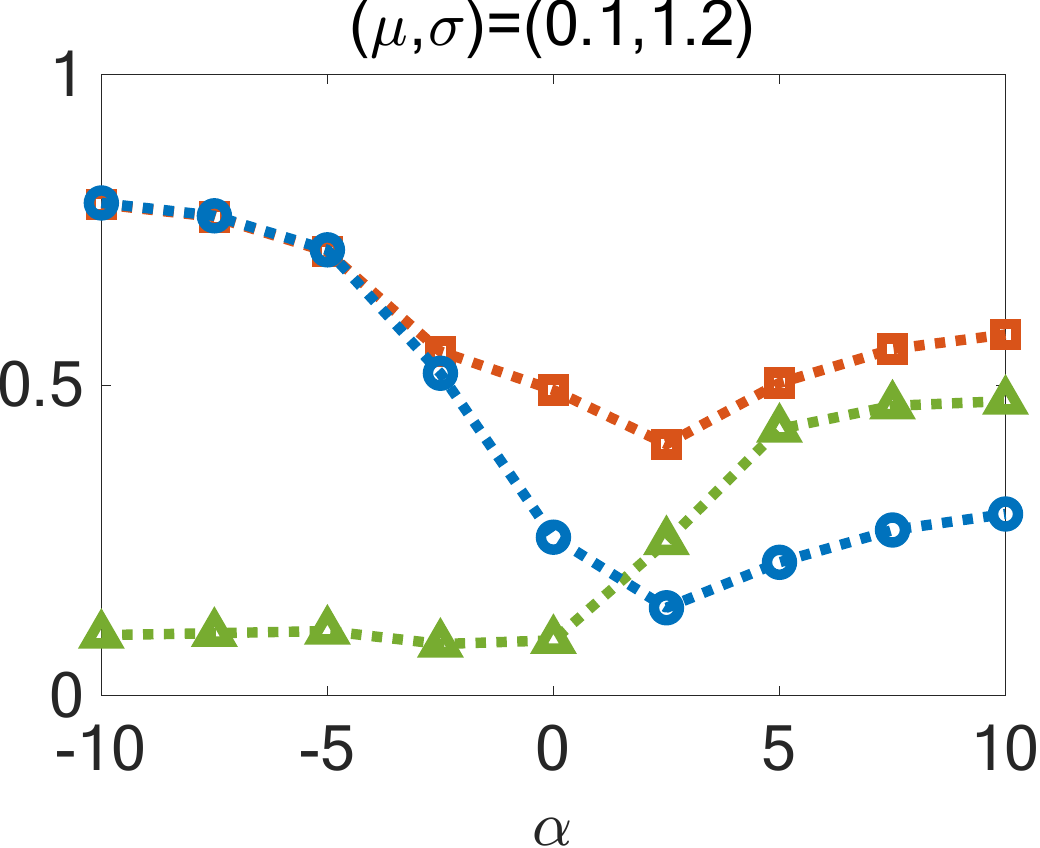} \\
\includegraphics[scale=0.21]{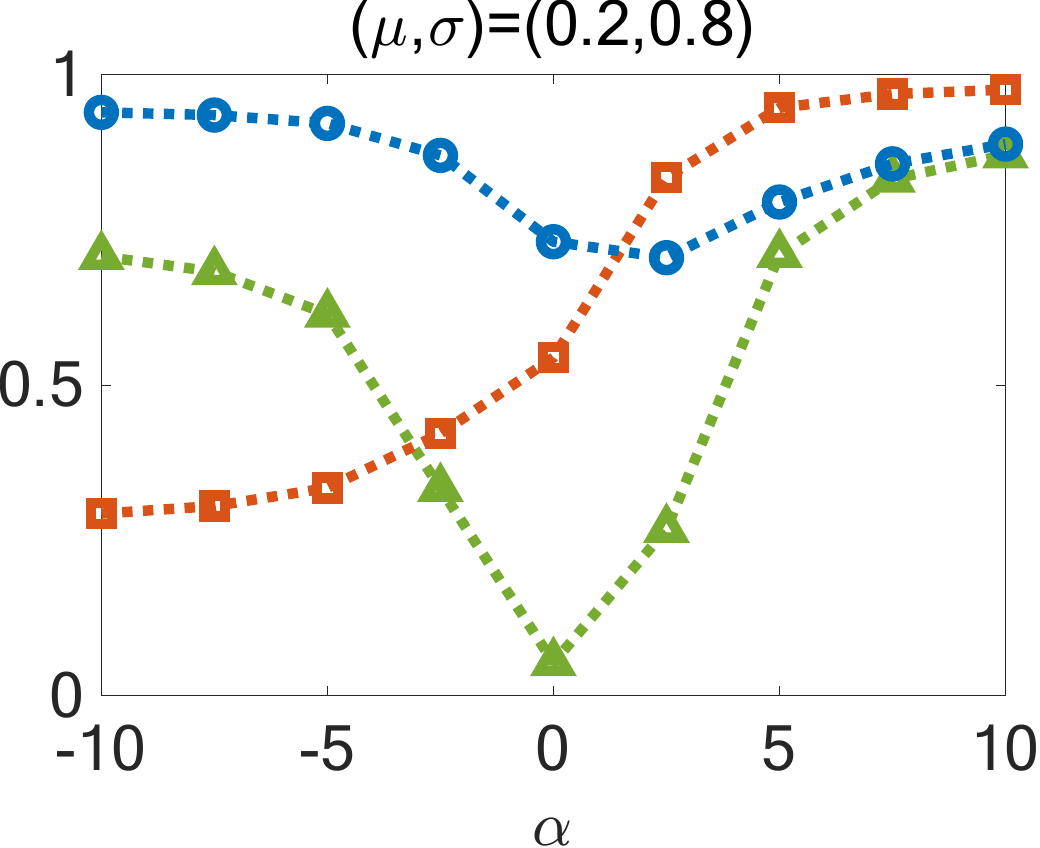} &
\includegraphics[scale=0.21]{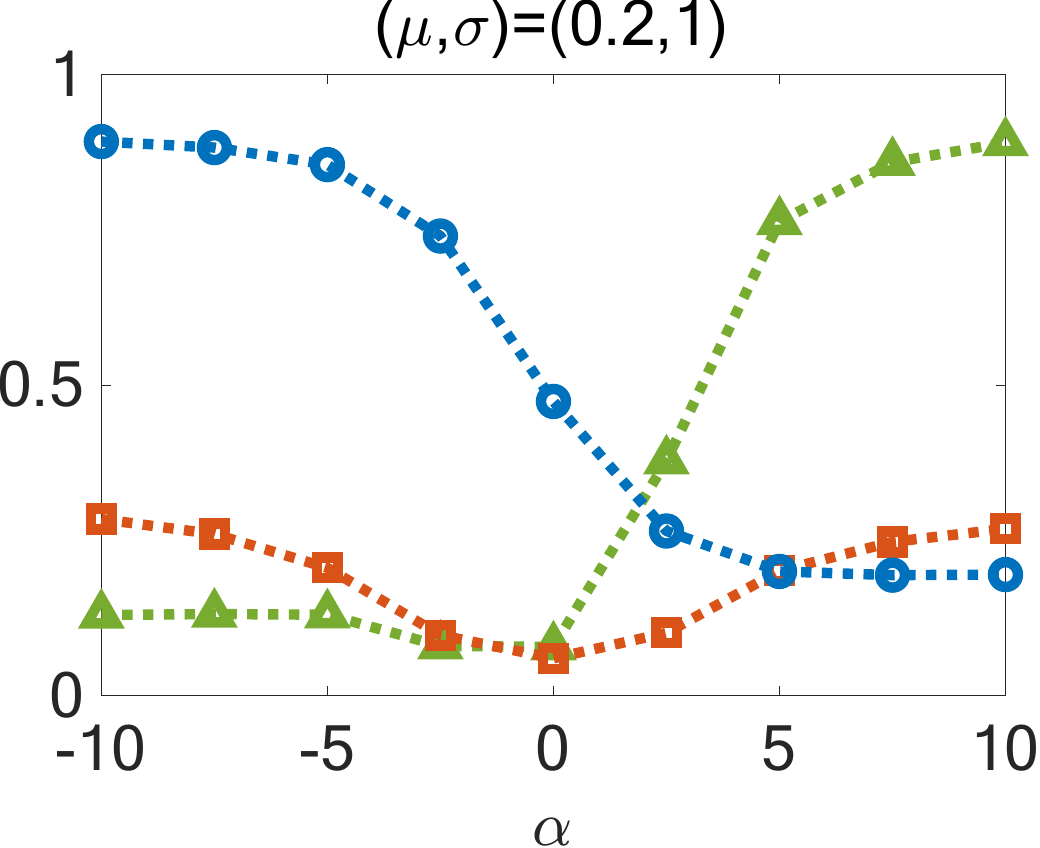} &
\includegraphics[scale=0.21]{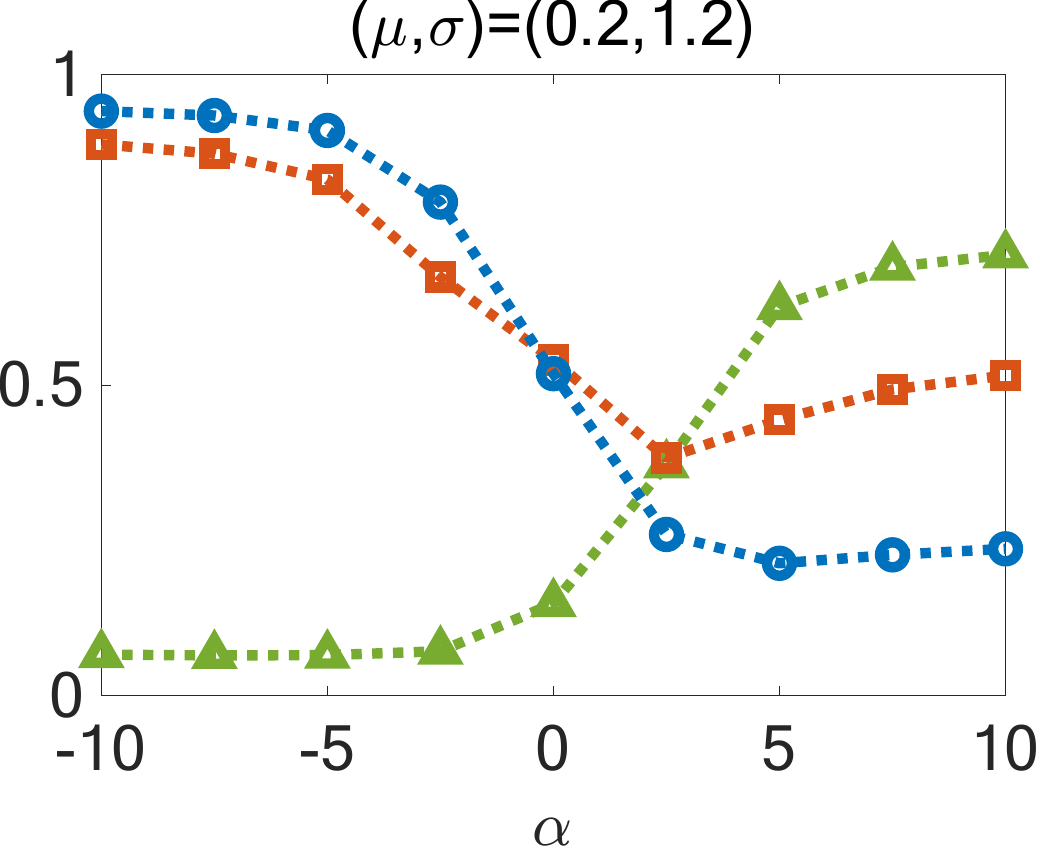}
\end{tabular}
\caption{Powers of $\omega^{[m]}_N$ under skew-normal distributions with index $\alpha$,
$\mu\in\{0,0.1,0.2\}$ and $\sigma\in\{0.8,1.0,1.2\}$.}
\smallskip
(\blueredgreen)
\label{fig:cvm-skew}
\end{center}
\end{figure}

\subsection{Moment generating function}
\label{subsec:mgf-cvm}

The MGF of the limiting GCvM statistics 
in \eqref{KL2-cvm} and \eqref{KL2-cvm*} have the representations
\[
    \mathbb{E}\lpa{e^{s \omega^{[m]}}}
    = \prod_{k\ge1} \biggl(1-\frac{2s}
    {\lambda_k}\biggr)^{-\frac{1}{2}}
    \quad\text{where}\quad
    \lambda_k=(\pi k)^{2m},
\]
and
\[
    \mathbb{E}\lpa{e^{s \omega^{*[m]}}}
    = \mathbb{E}\lpa{e^{s \omega^{[m]}}}
    \prod_{1\le k<m} \biggl(1-\frac{2s}
    {\lambda_k}\biggr)^{\frac{1}{2}},
\]
which have finite-product expressions as follows.
\begin{thm}\label{thm:mgf-cvm}
For $m\ge1$
\[
    \mathbb{E}\lpa{e^{s \omega^{[m]}}}
    = \sqrt{\frac{\sqrt{2s} \,e^{\frac{m-1}{2}\pi i}}  
    {\prod_{0\le j<m}
    \sin\lpa{(2s)^{\frac{1}{2m}}
    e^{\frac{j\pi i}{m}}}}}.
\]
\end{thm}
In particular, with $\sigma_m := (2s)^{\frac1{2m}}$, we have 
\begin{align*}
    \mathbb{E}\lpa{e^{s \omega^{[1]}}}
    &= \sqrt{\frac{\sigma_1}{\sin(\sigma_1)}}, \\
    \mathbb{E}\lpa{e^{s \omega^{[2]}}}
    &= \sqrt{\frac{\sigma_2^2}
    {\sin(\sigma_2)\sinh(\sigma_2)}}, \\
    \mathbb{E}\lpa{e^{s \omega^{[3]}}}
    &= \sqrt{\frac{\sigma_3^3}
    {\sin(\sigma_3)\lpa{\cosh(\sqrt{3}\sigma_3)-\cos(\sigma_3)}}}, \\
    \mathbb{E}\lpa{e^{s \omega^{[4]}}}
    &= \sqrt{\frac{2\sigma_4^4}
    {\sin(\sigma_4)\sinh(\sigma_4)\lpa{\cosh(\sqrt{2}\sigma_4)
    -\cos(\sqrt{2}\sigma_4)}}}.
\end{align*}

\begin{appendix}

\addcontentsline{toc}{section}{Appendices}

\section{Proof of (\ref{ANm-PQ}): generalized Anderson-Darling 
statistic in terms of the samples}
\label{subsec:Anm-PQ} 

We prove \eqref{ANm-PQ} in this Appendix, which is copied again here
for convenience:
\[
    A_N^{[m]} 
    = -\sfrac1N\sum_{1\le i,j\le N}
    R_m(a,b)(\log a + \log b)
    + \sfrac{1}{N}\sum_{1\le i,j\le N}Q_m(a,b),
\]
where for simplicity $a=a_{i,j} := 1-X_i\wedge X_j$ and $b=b_{i,j} := 
X_i\vee X_j$. While the expression we derive for $Q_m(a,b)$, based 
mostly on direct expansion and straightforward simplification, may 
not be optimal and likely to be further simplified, it is
sufficient to prove \eqref{ANm-PQ}.

For convenience, define 
\begin{align}\label{fmkx}
    f_{m,k}(x) := x^{-m} P_k^{(-m)}(x)
    = \sqrt{2k+1}\sum_{0\le \ell\le k}  
    \frac{(-1)^{k+\ell}(k+\ell)!}
    {\ell!(k-\ell)!(m+\ell)!}\, x^\ell.
\end{align}
Then, by the proof of Lemma~\ref{lm:invariance},
\begin{align}\label{A-P-I}
    A_N^{[m]} := \sfrac1N\sum_{1\le i,j\le N}
    \sum_{0\le k,\ell< m} P_k(1-a)P_\ell(b)
    I_{k,\ell}(a,b),
\end{align}
where, by \eqref{parentheses},
\begin{align}
    I_{k,\ell}(a,b)
    &:= \int_0^{1-a}\frac{x^{m}}{(1-x)^{m}}
    \,f_{m,k}(x)f_{m,\ell}(x)\dd x \nonumber \\
    &\quad
    +(-1)^{m+k} \int_{1-a}^b 
    f_{m,k}(1-x)f_{m,\ell}(x)\,\dd x \nonumber \\
    &\quad
    +(-1)^{k+\ell} \int_b^1 \frac{(1-x)^{m}}{x^{m}}
    f_{m,k}(1-x)f_{m,\ell}(1-x) \dd x \nonumber \\
    &= (-1)^{m+k} 
    \int_{1-a}^b f_{m,k}(1-x)f_{m,\ell}(x) \dd x \nonumber \\
    &\quad
    +\left(\int_a^1+(-1)^{k+\ell} \int_b^1\right) 
    \frac{(1-x)^{m}}{x^{m}}
    f_{m,k}(1-x)f_{m,\ell}(1-x)\,\dd x.
\label{Iklab}
\end{align}
Define
\[
    \phi_{m,k}(x) = 
    \frac{1}{(m-1)!}\int_0^1(1-x-t)^{m-1} P_k(t)\dd t.
\]
We now derive an alternative expression for $I_{k,\ell}(a,b)$
in which the dependence on $a$ and $b$ is explicitly separated.
\begin{lm} 
\begin{equation}\label{Ikl}
\begin{aligned}
    I_{k,\ell}(a,b)
    =& (-1)^{m+k+1} \int_a^1 \frac{\phi_{m,k}(1-x)}{x^{m}}
    \, f_{m,\ell}(1-x) \dd x \\
    & +(-1)^{k+\ell}
    \int_b^1 \frac{\phi_{m,\ell}(x)}{x^{m}}
    \, f_{m,k}(1-x) \dd x \\
    & +(-1)^{m+k}
    \int_0^1 f_{m,k}(1-x)f_{m,\ell}(x)\dd x.
\end{aligned}
\end{equation}
\end{lm}
\begin{proof}
The reason of introducing the additional polynomials $\phi_{m,k}$ is
because the relation
\begin{equation*}
    P_k(1-x) = (-1)^k P_k(x),
\end{equation*}
does not extend directly to $P_k^{(-m)}(x)$ for $m\ge0$
but we do have 
\[
    \phi_{m,k}(1-x) = (-1)^{m+k+1}\phi_{m,k}(x).
\]

The definition also implies that 
\[
    P_k^{(-m)}(1-x) + (-1)^{m+k+1}P_k^{(-m)}(x)  
    = (-1)^{m+k+1}\phi_{m,k}(x),
\]
which can be readily proved by definition. Here $\phi_{m,k}(x)$ is a
polynomial in $x$ of degree $m$. Note that both $P_k^{(-m)}(x)$ and
$P_k^{(-m)}(1-x)$ are of degree $m+k$, which means that the
coefficients of $x^j$ for $j>m$ on the left-hand side are all
cancelled.

With this relation, we have
\begin{align*}
    &\int_0^{1-a}\frac{x^{m}}{(1-x)^{m}}
    \,f_{m,k}(x)f_{m,\ell}(x)\dd x \\
    &\qquad = (-1)^{m+k}\left(
    \int_0^{1-a} f_{m,k}(1-x)f_{m,\ell}(x)\dd x
    -\int_0^{1-a} \frac{\phi_{m,k}(x)}{(1-x)^{m}}
    \, f_{m,\ell}(x) \dd x\right).
\end{align*}
Similarly, 
\begin{align*}
    &(-1)^{k+\ell}\int_b^1 \frac{(1-x)^{m}}{x^{m}}
    f_{m,k}(1-x)f_{m,\ell}(1-x) \dd x \\
    &\qquad = (-1)^{m+k}\left(
    \int_b^1 f_{m,k}(1-x)f_{m,\ell}(x)\dd x\right. \\
    &\hspace*{3cm} \left.
	-(-1)^{m+\ell+1}\int_b^1 \frac{\phi_{m,\ell}(x)}{x^{m}}
    \, f_{m,k}(1-x) \dd x\right).
\end{align*}
These decompositions, coupling with \eqref{Iklab}, prove \eqref{Ikl}.
\end{proof}

We now examine the contribution to $A_N^{[m]}$ \eqref{A-P-I}
of the first integral in \eqref{Ikl}.
\begin{lm}\label{lmm:inta1}
For $m\ge1$
\begin{equation}\label{inta1}
\begin{aligned}
    & \sum_{0\le k,\ell<m} P_k(1-a)P_\ell(b)
    \int_a^1\frac1{x^{m}}\,
    \phi_{m,k}(x)f_{m,\ell}(1-x)\dd x \\
    &\quad= -R_m(a,b) \log a +
    \sum_{0\le \ell\le r\le m-1} 
    c_{m,\ell,r} b^{m-1-r} (ab)^\ell (1-a-b)^{r-\ell},
\end{aligned}
\end{equation}
where $R_m(a,b)$ is given in \eqref{Rm} and 
\[
    c_{m,\ell,r} := 
    \begin{cases}\displaystyle
        \frac{(-1)^r\binom{m-1}{r}
        \binom{r}{\ell} \binom{m-1+\ell}{r}}
        {(m-1)!^2(m-1-r)}, 
        & \text{if } 0\le r\le m-2, \\
        \qquad & \\
        \displaystyle
        \frac{(-1)^{m-1}
        \binom{m-1}{\ell}\binom{m-1+\ell}{\ell}
        \left(H_{m-1}+H_\ell - H_{m-1+\ell} \right)}
        {(m-1)!^2}, &\text{if }r=m-1.
    \end{cases}
\]
Here $H_m := \sum_{1\le j\le m}j^{-1}$ denotes the harmonic numbers. 
\end{lm}

By symmetry, we also have the following explicit contribution to 
$A_N^{[m]}$ from the second integral in \eqref{Ikl}.
\begin{lm}
For $m\ge1$
\begin{align*} 
    & \int_b^1\frac1{x^{m}}
    \sum_{0\le k,\ell<m}(-1)^{k+\ell} P_k(1-a)P_\ell(b)
    \phi_{m,\ell}(x)f_{m,k}(1-x)\dd x \\
    &\quad = -R_m(a,b) \log b + 
    \sum_{0\le \ell\le r\le m-1} c_{m,\ell,r} a^{m-1-r} (ab)^\ell (1-a-b)^{r-\ell}.
\end{align*}
\end{lm}

\begin{proof}[Proof of Lemma~\ref{lmm:inta1}]
Consider first the integral
\[
    (-1)^{m+k+1}\int_a^1 \frac{\phi_{m,k}(1-x)}{x^{m}}
    \, f_{m,\ell}(1-x) \dd x
    = \int_a^1 \frac{\phi_{m,k}(x)}{x^{m}}
    \, f_{m,\ell}(1-x) \dd x.
\]
We begin by deriving a more explicit expression for $\phi_{m,k}(x)$ 
using \eqref{Pkx-closed-form}:
\begin{align*}
    \frac{\phi_{m,k}(x)}{\sqrt{2k+1}} 
    &= \frac{1}{(m-1)!\sqrt{2k+1}}\int_0^1(1-x-t)^{m-1} 
    P_k(t)\dd t \\
    &=\frac{1}{(m-1)!}
    \sum_{0\le r<m}\binom{m-1}{r}(-x)^{m-1-r}\\
	&\hspace*{2.5cm}\times
    \sum_{0\le \ell \le k} \binom{k+\ell}{\ell}
    \binom{k}{\ell}(-1)^{k+\ell} 
    \int_0^1(1-t)^r t^\ell\dd t \\
    &= \frac{1}{(m-1)!}\sum_{0\le r< m}
	\binom{m-1}{r}(-x)^{m-1-r}\\
    &\hspace*{2.5cm}\times
	\sum_{0\le \ell \le k} \binom{k+\ell}{\ell}
    \binom{k}{\ell}(-1)^{k+\ell}  
    \frac{\ell!r!}{(r+\ell+1)!} \\
    &= \sum_{0\le r< m}\frac{(-x)^{m-1-r}}{(m-1-r)!}
    \sum_{0\le \ell \le k} \frac{(k+\ell)!}
    {\ell!(k-\ell)!(r+\ell+1)!}\, (-1)^{k+\ell}.
\end{align*}
By the identity
\[
    \sum_{0\le \ell \le k} \frac{(k+\ell)!(-1)^{k+\ell}}
    {\ell!(k-\ell)!(r+\ell+1)!}
    = \frac{(-1)^kr!}{(r-k)!(r+k+1)!} \qquad (r\ge k),
\]
we then obtain 
\[
    \frac{\phi_{m,k}(x)}{\sqrt{2k+1}} 
    = \sum_{k\le r< m}\frac{r!(-1)^{k+m-1-r}}{(m-1-r)!
    (r-k)!(r+k+1)!} \, x^{m-1-r}.
\]
Consider now
\begin{align*}
    &\sum_{0\le k< m} P_k(1-a)\phi_{m,k}(x) \\
    &\quad = (-1)^{m-1}\sum_{0\le k< m}(2k+1)\\
    &\hspace*{2cm}\times \sum_{0\le r\le k}\sum_{k\le \ell< m}
    \frac{(k+r)!\ell! (-1)^{r+\ell}(1-a)^r x^{m-1-\ell}}
    {r!r!(k-r)!(m-1-\ell)!(\ell-k)!(\ell+k+1)!} \\
    &\quad = (-1)^{m-1}\sum_{0\le r\le \ell< m}
    \frac{\ell! (1-a)^r x^{m-1-\ell}}{r!r!(m-1-\ell)!
    }(-1)^{r+\ell}\\
	&\hspace*{2cm} \times\sum_{r\le k\le \ell} 
    \frac{(2k+1)(k-r)!}{(k-r)!(\ell-k)!(\ell+k+1)!}  \\
    &\quad = (-1)^{m-1}\sum_{0\le r\le \ell< m}
    \frac{\ell! (1-a)^r x^{m-1-\ell}}
    {r!r!(m-1-\ell)!}(-1)^{r+\ell}
    {\frac{r!}{\ell!(\ell-r)!}} \\
    &\quad=  (-1)^{m-1} \sum_{0\le r\le \ell< m}
    \frac{(1-a)^r x^{m-1-\ell}}
    {r!(m-1-\ell)!(\ell-r)!}(-1)^{r+\ell} \\
    &\quad= (-1)^{m-1}\frac{(x-a)^{m-1}}{(m-1)!}.
\end{align*}

Now consider the partial sum
\begin{align*}
    &\sum_{0\le k< m} P_k(1-a) f_{m,k}(1-x) \\
    &\quad= \sum_{0\le k< m} (2k+1)\sum_{0\le r_1,r_2\le k}
    \frac{(-1)^{r_1+r_2}(k+r_1)!(k+r_2)!(1-a)^{r_1}(1-x)^{r_2}}
    {r_1!r_1!(k-r_1)!r_2!(m+r_2)!(k-r_2)!} \\
    &\quad =\sum_{0\le r_1,r_2< m}
    \frac{(-1)^{r_1+r_2}(1-a)^{r_1}(1-x)^{r_2}}
    {r_1!r_1!r_2!(m+r_2)!}{
    \sum_{0\le k< m} (2k+1)\frac{(k+r_1)!(k+r_2)!}
    {(k-r_1)!(k-r_2)!}} \\
    &\quad= \sum_{0\le r_1,r_2< m}
    \frac{(-1)^{r_1+r_2}(1-a)^{r_1}(1-x)^{r_2}}
    {r_1!r_1!r_2!(m+r_2)!} \cdot
    {\frac{(m+r_1)!(m+r_2)!}
    {(r_1+r_2+1)(m-1-r_1)!(m-1-r_2)!}} \\
    &\quad= \sum_{0\le r_1< m} \frac{(-1)^{r_1}
    (m+r_1)!(1-a)^{r_1}}{r_1!r_1!(m-1-r_1)!}
    \sum_{0\le r_2< m}\frac{(-1)^{r_2}(1-x)^{r_2}}
    {r_2!(r_1+r_2+1)(m-1-r_2)!},
\end{align*}
where we used the same identity \eqref{inner_sum} as above.
Observe that
\begin{align*}
    &\sum_{0\le r_2< m}\frac{(-1)^{r_2}(1-x)^{r_2}}
    {r_2!(r_1+r_2+1)(m-1-r_2)!} \\
    &\qquad= \frac1{(m-1)!}\int_0^1 u^{r_1}
    \left(1-(1-x)u\right)^{m-1}\dd u \\
    &\qquad=\frac1{(m-1)!}\sum_{0\le r_2< m}
    \binom{m-1}{r_2} x^{r_2}
    \int_0^1 (1-u)^{m-1-r_2} u^{r_1+r_2} \dd u \\
    &\qquad= \frac1{(m-1)!}\sum_{0\le r_2< m}\binom{m-1}{r_2} 
    \frac{(m-1-r_2)!(r_1+r_2)!}{(m+r_1)!}\,x^{r_2} \\
    &\qquad= \sum_{0\le r_2< m} 
    \frac{(r_1+r_2)!}{r_2!(m+r_1)!}\,x^{r_2}.
\end{align*}
It follows that
\begin{equation}\label{double-sum-a}
\begin{aligned}
    &\sum_{0\le k< m} P_k(1-a) f_{m,k}(1-x) \\
    &\qquad= \frac1{(m-1)!}\sum_{0\le r_2< m} x^{r_2}
    \sum_{0\le r_1< m}\binom{r_1+r_2}{r_1}\binom{m-1}{r_1}
    (-1)^{r_1}(1-a)^{r_1}.
\end{aligned}
\end{equation}
In a similar way, 
\begin{equation}\label{sumPlfl}
\begin{aligned}
    &\sum_{0\le \ell< m} P_\ell(b) f_{m,\ell}(1-x) \\
    &\qquad= \frac1{(m-1)!}
    \sum_{0\le \ell_2< m} x^{\ell_2}
    \sum_{0\le \ell_1< m}
    \binom{\ell_1+\ell_2}{\ell_1}\binom{m-1}{\ell_1}
    (-1)^{\ell_1}b^{\ell_1}.
\end{aligned}
\end{equation}
Consequently, 
\begin{align*} 
    & \int_a^1\frac1{x^{m}}
    \sum_{0\le k,\ell<m} P_k(1-a)P_\ell(b)
    \phi_{m,k}(x)f_{m,\ell}(1-x)\dd x \\
    &\quad = \frac{(-1)^{m-1}}{(m-1)!^2}
    \int_a^1\frac{(x-a)^{m-1}}{x^{m}} 
    \sum_{0\le k,\ell< m} \binom{k+\ell}{\ell}
    \binom{m-1}{k}
    (-1)^{k}b^{k} x^{\ell} \dd x \\
    &\quad = \frac{(-1)^{m-1}}{(m-1)!^2}
    \int_a^1\frac{(x-a)^{m-1}}{x^{m}} 
    [z^0]\sum_{0\le k,\ell< m}
    z^{-\ell}(1-z)^{-k-1}\binom{m-1}{k}
    (-1)^{k}b^{k}  x^\ell  \dd x \\
    &\quad = \frac{(-1)^{m-1}}{(m-1)!^2} 
    [z^{m-1}] \frac{(1-b-z)^{m-1}}
    {(1-z)^{m}}\int_a^1\frac{(x-a)^{m-1}}{x-z} \dd x \\
    &\quad = \frac{1}{(m-1)!^2} [z^{m-1}] \frac{(1-b-z)^{m-1}}
    {(1-z)^{m}} \\
	&\qquad\qquad\times
	\sum_{0\le \ell < m}\binom{m-1}{\ell}
    (-1)^{\ell} (a-z)^{m-1-\ell} \int_a^1 (x-z)^{\ell-1} \dd x \\
    &\quad = \frac{1}{(m-1)!^2} [z^{m-1}] \frac{(1-b-z)^{m-1}}
    {(1-z)^{m}} \left( (a-z)^{m-1}\log\frac{1-z}{a-z} \right. \\
    &\qquad \left. +
    \sum_{1\le \ell < m}\binom{m-1}{\ell}
    \frac{(-1)^{\ell}}{\ell} \left((1-z)^\ell
    (a-z)^{m-1-\ell}-(a-z)^{m-1} \right)\right) \\
    &\quad = \frac{1}{(m-1)!^2} [z^{m-1}]
    \Biggl\{\frac{(1-b-z)^{m-1}(a-z)^{m-1}}
    {(1-z)^{m}}\left(\log\frac{1-z}{1-\frac{z}{a}} -\log a\right) 
    \\ &\hspace*{4cm} +
    \sum_{1\le \ell < m}\binom{m-1}{\ell}
    \frac{(-1)^{\ell}}{\ell} \cdot
    \frac{(1-b-z)^{m-1}(a-z)^{m-1-\ell}}
    {(1-z)^{m-\ell}} \\
    &\hspace*{4cm}- \frac{(1-b-z)^{m-1}(a-z)^{m-1}}
    {(1-z)^{m}}\sum_{1\le \ell < m}\binom{m-1}{\ell}
    \frac{(-1)^{\ell}}{\ell}
    \Biggr\}.
\end{align*}
Note that, by the change of variables $z \mapsto -au/(1-a-u)$, we
obtain
\begin{align*}
    &[z^{m-1}] \frac{(1-b-z)^{m-1}(a-z)^{m-1}}
    {(1-z)^{m}} \\
    &\qquad= (-1)^{m-1} [u^{m-1}]
    \frac{((1-a-b)(1-u)+ab)^{m-1}}
    {(1-u)^{m}} \\
    &\qquad= (-1)^{m-1} [u^{m-1}]\frac{((1-a-b)u+ab)^{m-1}}
    {(1-u)^{m}} \\
    &\qquad= (-1)^{m-1} \sum_{0\le \ell< m}
    \binom{m-1}{\ell}(ab)^\ell(1-a-b)^{m-1-\ell}
    [u^\ell](1-u)^{-m} \\
    &\qquad= (-1)^{m-1} \sum_{0\le \ell< m}\binom{m-1}{\ell}
    \binom{m-1+\ell}{\ell}(ab)^\ell(1-a-b)^{m-1-\ell}.
\end{align*}
Similarly, 
\begin{align*}
    &[z^{m-1}] \frac{(1-b-z)^{m-1}(a-z)^{m-1}}
    {(1-z)^{m}}\log\frac{1-z}{1-\frac za} \\
    &\quad= (-1)^{m} \sum_{0\le \ell< m}\binom{m-1}{\ell}
    \binom{m-1+\ell}{\ell}
    (ab)^\ell(1-a-b)^{m-1-\ell}\left(H_{m-1+\ell}-H_\ell\right).
\end{align*}
Also
\[
    \sum_{1\le \ell < m}\binom{m-1}{\ell}
    \frac{(-1)^{\ell}}{\ell} = -H_{m-1}.
\]
This proves \eqref{inta1}.
\end{proof}

We now examine the partial sum involving last integral in 
\eqref{Ikl}.
\[
    \sum_{0\le k,\ell<m} (-1)^{m+k} P_k(1-a)P_\ell(b)
    \int_{0}^1 f_{m,k}(1-x)f_{m,\ell}(x) \dd x.
\]
By \eqref{Pkx-closed-form} and \eqref{fmkx}, this is already a 
polynomial in $a$ of degree $m-1$ and in $b$ of degree $m-1$. We can 
derive a more precise but cumbersome expression for the coefficients.
\begin{lm}\label{int01}
\[
    \sum_{0\le k,\ell< m} (-1)^{m+k} P_k(1-a)P_\ell(b)
    \int_0^1 f_{m,k}(1-x)f_{m,\ell}(x) \dd x
    = \sum_{0\le k,\ell< m} d_{m,k,\ell} a^k b^\ell,
\]
where
\[
    d_{m,k,\ell} := (-1)^{m+k+\ell}\frac{
    (k+\ell)! (H_{m+\ell}-H_{\ell})+\displaystyle
    \sum_{\substack{0\le r< m \\ r\neq \ell}}
        \frac{(k+r)!}{\ell-r}\left(
        \frac{(m+\ell)!}{(m+r)!}- \frac{\ell!}{r!}
        \right)
    }{k!^2\ell!^2(m-1-k)!(m-1-\ell)!}.
\]
\end{lm}
\begin{proof}
We use the two expansions \eqref{double-sum-a} and 
\eqref{sumPlfl}, and obtain
\begin{align*}
    & \sum_{0\le k,\ell< m} (-1)^{m+k}P_k(1-a)P_\ell(b)
    f_{m,k}(1-x)f_{m,\ell}(x) \\
    &\quad  = (-1)^{m}\sum_{0\le k,\ell< m} P_k(a)P_\ell(b)
    f_{m,k}(1-x)f_{m,\ell}(x) \\
    &\quad  = \frac{(-1)^{m}}{(m-1)!^2}
    \sum_{\substack{0\le k_2,\ell_2<m \\
    0\le k_1,\ell_1<m}} 
    \binom{k_1+k_2}{k_1}\binom{\ell_1+\ell_2}{\ell_1}
    \binom{m-1}{k_1}\binom{m-1}{\ell_1}\\
    &\hspace*{5cm}\times(-1)^{k_1+\ell_1}a^{k_1} b^{\ell_1} 
    x^{k_2}(1-x)^{\ell_2} \\
    &\quad  = (-1)^{m}\hspace*{-.5cm}
    \sum_{0\le k_1,k_2,\ell_1,\ell_2<m} 
    \frac{(k_1+k_2)!(\ell_1+\ell_2)!
    (-1)^{k_1+\ell_1}a^{k_1} b^{\ell_1}}
    {k_1!k_1!k_2!\ell_1!\ell_1!\ell_2!(m—1-k_1)!(m-1-\ell_1)!}\,
    x^{k_2}(1-x)^{\ell_2}.
\end{align*}
Thus
\begin{align*}
    & \sum_{0\le k,\ell< m} (-1)^{m+k} P_k(1-a)P_\ell(b)
    \int_0^1 f_{m,k}(1-x)f_{m,\ell}(x) \dd x \\
    &\quad  = (-1)^{m}
    \sum_{0\le k_1,k_2,\ell_1,\ell_2<m} 
    \frac{(k_1+k_2)!(\ell_1+\ell_2)!(-1)^{k_1+\ell_1}
    a^{k_1} b^{\ell_1}}
    {k_1!k_1!\ell_1!\ell_1!(m-1-k_1)!(m-1-\ell_1)!(k_2+\ell_2+1)!} \\
    &\quad = (-1)^{m}\sum_{0\le k_1,\ell_1<m} 
    \frac{(-1)^{k_1+\ell_1}a^{k_1} b^{\ell_1}}
    {k_1!k_1!\ell_1!\ell_1!(m-1-k_1)!(m-1-\ell_1)!}\\
	&\hspace*{3cm}\times
    \sum_{0\le k_2,\ell_2<m} \frac{(k_1+k_2)!(\ell_1+\ell_2)!}
    {(k_2+\ell_2+1)!}.
\end{align*}
Now
\[
    \sum_{0\le \ell_2<m} \frac{(k_1+k_2)!(\ell_1+\ell_2)!}
    {(k_2+\ell_2+1)!} = 
    \begin{cases}
        \displaystyle \frac{(k_1+k_2)!}{\ell_1-k_2}\left(
        \frac{(m+\ell_1)!}{(m+k_2)!}- \frac{\ell_1!}{k_2!}
        \right), & \mbox{if } k_2\ne \ell_1, \\
        (k_1+k_2)!
        (H_{m+\ell_1}-H_{\ell_1}), &\mbox{if }k_2=\ell_1.
    \end{cases}
\]
This proves \eqref{int01}.
\end{proof}	
\end{appendix}

\begin{acks}[Acknowledgments]
The authors thank Yoichi Nishiyama for his helpful comments, and
Chihiro Hirotsu for bringing the authors' attention to change-point
analysis.
\end{acks}

\bibliographystyle{imsart-nameyear}
\bibliography{iem-ggof-2024}

\begin{thebibliography}{66}

\bibitem[\protect\citeauthoryear{Abramowitz and Stegun}{1964}]{Abramowitz64}
\begin{bbook}[author]
\bauthor{\bsnm{Abramowitz},~\bfnm{Milton}\binits{M.}} \AND
  \bauthor{\bsnm{Stegun},~\bfnm{Irene~A.}\binits{I.~A.}}
(\byear{1964}).
\btitle{Handbook of mathematical functions with formulas, graphs, and
  mathematical tables}.
\bseries{National Bureau of Standards Applied Mathematics Series}.
\bpublisher{U.S. Government Printing Office, Washington, D.C.}
\bmrnumber{0167642}
\end{bbook}
\endbibitem

\bibitem[\protect\citeauthoryear{Adler}{1990}]{Adler90}
\begin{bbook}[author]
\bauthor{\bsnm{Adler},~\bfnm{Robert~J.}\binits{R.~J.}}
(\byear{1990}).
\btitle{An introduction to continuity, extrema, and related topics for general
  {G}aussian processes}.
\bpublisher{Institute of Mathematical Statistics, Hayward, CA}.
\bmrnumber{1088478}
\end{bbook}
\endbibitem

\bibitem[\protect\citeauthoryear{Anderson and Darling}{1952}]{Anderson52}
\begin{barticle}[author]
\bauthor{\bsnm{Anderson},~\bfnm{T.~W.}\binits{T.~W.}} \AND
  \bauthor{\bsnm{Darling},~\bfnm{D.~A.}\binits{D.~A.}}
(\byear{1952}).
\btitle{Asymptotic theory of certain ``goodness of fit'' criteria based on
  stochastic processes}.
\bjournal{Ann. Math. Statistics}
\bvolume{23}
\bpages{193--212}.
\bdoi{10.1214/aoms/1177729437}
\bmrnumber{0050238}
\end{barticle}
\endbibitem

\bibitem[\protect\citeauthoryear{Azzalini}{1985}]{Azzalini85}
\begin{barticle}[author]
\bauthor{\bsnm{Azzalini},~\bfnm{A.}\binits{A.}}
(\byear{1985}).
\btitle{A class of distributions which includes the normal ones}.
\bjournal{Scand. J. Statist.}
\bvolume{12}
\bpages{171--178}.
\bmrnumber{808153}
\end{barticle}
\endbibitem

\bibitem[\protect\citeauthoryear{Baringhaus and Henze}{2008}]{Baringhaus08}
\begin{barticle}[author]
\bauthor{\bsnm{Baringhaus},~\bfnm{L.}\binits{L.}} \AND
  \bauthor{\bsnm{Henze},~\bfnm{N.}\binits{N.}}
(\byear{2008}).
\btitle{A new weighted integral goodness-of-fit statistic for exponentiality}.
\bjournal{Statist. Probab. Lett.}
\bvolume{78}
\bpages{1006--1016}.
\bdoi{10.1016/j.spl.2007.09.060}
\bmrnumber{2418918}
\end{barticle}
\endbibitem

\bibitem[\protect\citeauthoryear{Broniatowski and
  Stummer}{2022}]{Broniatowski22}
\begin{bincollection}[author]
\bauthor{\bsnm{Broniatowski},~\bfnm{Michel}\binits{M.}} \AND
  \bauthor{\bsnm{Stummer},~\bfnm{Wolfgang}\binits{W.}}
(\byear{2022}).
\btitle{Chapter 5 -- A unifying framework for some directed distances in
  statistics}.
In \bbooktitle{Geometry and Statistics},
(\beditor{\bfnm{Frank}\binits{F.}~\bsnm{Nielsen}}, \beditor{\bfnm{Arni
  S.~R.}\binits{A.~S.~R.}~\bsnm{{Srinivasa Rao}}} \AND
  \beditor{\bfnm{C.~R.}\binits{C.~R.}~\bsnm{Rao}}, eds.).
\bseries{Handbook of Statistics}
\bvolume{46}
\bpages{145--223}.
\bpublisher{Elsevier}.
\bdoi{https://doi.org/10.1016/bs.host.2022.03.007}
\bmrnumber{4599155}
\end{bincollection}
\endbibitem

\bibitem[\protect\citeauthoryear{Chang and Ha}{1999}]{Chang99}
\begin{barticle}[author]
\bauthor{\bsnm{Chang},~\bfnm{Ching-Hua}\binits{C.-H.}} \AND
  \bauthor{\bsnm{Ha},~\bfnm{Chung-Wei}\binits{C.-W.}}
(\byear{1999}).
\btitle{On eigenvalues of differentiable positive definite kernels}.
\bjournal{Integral Equations Operator Theory}
\bvolume{33}
\bpages{1--7}.
\bdoi{10.1007/BF01203078}
\bmrnumber{1664351}
\end{barticle}
\endbibitem

\bibitem[\protect\citeauthoryear{Chang and Ha}{2001}]{Chang01}
\begin{barticle}[author]
\bauthor{\bsnm{Chang},~\bfnm{Ching-Hua}\binits{C.-H.}} \AND
  \bauthor{\bsnm{Ha},~\bfnm{Chung-Wei}\binits{C.-W.}}
(\byear{2001}).
\btitle{The {G}reen functions of some boundary value problems via the
  {B}ernoulli and {E}uler polynomials}.
\bjournal{Arch. Math. (Basel)}
\bvolume{76}
\bpages{360--365}.
\bdoi{10.1007/PL00000445}
\bmrnumber{1824255}
\end{barticle}
\endbibitem

\bibitem[\protect\citeauthoryear{Chen and Li}{2003}]{Chen03}
\begin{barticle}[author]
\bauthor{\bsnm{Chen},~\bfnm{Xia}\binits{X.}} \AND
  \bauthor{\bsnm{Li},~\bfnm{Wenbo~V.}\binits{W.~V.}}
(\byear{2003}).
\btitle{Quadratic functionals and small ball probabilities for the {$m$}-fold
  integrated {B}rownian motion}.
\bjournal{Ann. Probab.}
\bvolume{31}
\bpages{1052--1077}.
\bdoi{10.1214/aop/1048516545}
\bmrnumber{1964958}
\end{barticle}
\endbibitem

\bibitem[\protect\citeauthoryear{Chernobai, Rachev and
  Fabozzi}{2015}]{Chernobai15}
\begin{bincollection}[author]
\bauthor{\bsnm{Chernobai},~\bfnm{Anna}\binits{A.}},
  \bauthor{\bsnm{Rachev},~\bfnm{Svetlozar~T.}\binits{S.~T.}} \AND
  \bauthor{\bsnm{Fabozzi},~\bfnm{Frank~J.}\binits{F.~J.}}
(\byear{2015}).
\btitle{Composite goodness-of-fit tests for left-truncated loss samples}.
In \bbooktitle{Handbook of Financial Econometrics and Statistics}
(\beditor{\bfnm{Cheng-Few}\binits{C.-F.}~\bsnm{Lee}} \AND
  \beditor{\bfnm{John~C.}\binits{J.~C.}~\bsnm{Lee}}, eds.)
\bpages{575--596}.
\bpublisher{Springer New York}, \baddress{New York, NY}.
\bdoi{10.1007/978-1-4614-7750-1_20}
\end{bincollection}
\endbibitem

\bibitem[\protect\citeauthoryear{Cochran and Lukas}{1988}]{Cochran88}
\begin{barticle}[author]
\bauthor{\bsnm{Cochran},~\bfnm{James~A.}\binits{J.~A.}} \AND
  \bauthor{\bsnm{Lukas},~\bfnm{Mark~A.}\binits{M.~A.}}
(\byear{1988}).
\btitle{Differentiable positive definite kernels and {L}ipschitz continuity}.
\bjournal{Math. Proc. Cambridge Philos. Soc.}
\bvolume{104}
\bpages{361--369}.
\bdoi{10.1017/S030500410006552X}
\bmrnumber{948920}
\end{barticle}
\endbibitem

\bibitem[\protect\citeauthoryear{Cram{\'e}r}{1928}]{Cramer28}
\begin{barticle}[author]
\bauthor{\bsnm{Cram{\'e}r},~\bfnm{Harald}\binits{H.}}
(\byear{1928}).
\btitle{On the composition of elementary errors. {S}econd paper: {S}tatistical
  applications}.
\bjournal{Skand. Aktuarietidskr}
\bvolume{11}
\bpages{141--180}.
\end{barticle}
\endbibitem

\bibitem[\protect\citeauthoryear{D'Agostino}{2017}]{DAgostino17}
\begin{bbook}[author]
\bauthor{\bsnm{D'Agostino},~\bfnm{RalphB}\binits{R.}}
(\byear{2017}).
\btitle{Goodness-of-fit-techniques}.
\bpublisher{Routledge}.
\end{bbook}
\endbibitem

\bibitem[\protect\citeauthoryear{Deheuvels}{1981}]{Deheuvels81}
\begin{barticle}[author]
\bauthor{\bsnm{Deheuvels},~\bfnm{Paul}\binits{P.}}
(\byear{1981}).
\btitle{An asymptotic decomposition for multivariate distribution-free tests of
  independence}.
\bjournal{J. Multivariate Anal.}
\bvolume{11}
\bpages{102--113}.
\bdoi{10.1016/0047-259X(81)90136-6}
\bmrnumber{612295}
\end{barticle}
\endbibitem

\bibitem[\protect\citeauthoryear{Deheuvels}{2005}]{Deheuvels05}
\begin{barticle}[author]
\bauthor{\bsnm{Deheuvels},~\bfnm{Paul}\binits{P.}}
(\byear{2005}).
\btitle{Weighted multivariate {C}ram\'{e}r-von {M}ises-type statistics}.
\bjournal{Afr. Stat.}
\bvolume{1}
\bpages{1--14}.
\bmrnumber{2298871}
\end{barticle}
\endbibitem

\bibitem[\protect\citeauthoryear{Deheuvels and Martynov}{2003}]{Deheuvels03}
\begin{bincollection}[author]
\bauthor{\bsnm{Deheuvels},~\bfnm{Paul}\binits{P.}} \AND
  \bauthor{\bsnm{Martynov},~\bfnm{Guennady}\binits{G.}}
(\byear{2003}).
\btitle{Karhunen-{L}o\`{e}ve expansions for weighted {W}iener processes and
  {B}rownian bridges via {B}essel functions}.
In \bbooktitle{High dimensional probability, {III} ({S}andjberg, 2002)}.
\bseries{Progr. Probab.}
\bpages{57--93}.
\bpublisher{Birkh\"{a}user, Basel}.
\bmrnumber{2033881}
\end{bincollection}
\endbibitem

\bibitem[\protect\citeauthoryear{Deheuvels and Martynov}{2008}]{Deheuvels08}
\begin{barticle}[author]
\bauthor{\bsnm{Deheuvels},~\bfnm{Paul}\binits{P.}} \AND
  \bauthor{\bsnm{Martynov},~\bfnm{Guennadi~V.}\binits{G.~V.}}
(\byear{2008}).
\btitle{A {K}arhunen-{L}o{\`e}ve decomposition of a {G}aussian process
  generated by independent pairs of exponential random variables}.
\bjournal{J. Funct. Anal.}
\bvolume{255}
\bpages{2363--2394}.
\bdoi{10.1016/j.jfa.2008.07.021}
\bmrnumber{2473261}
\end{barticle}
\endbibitem

\bibitem[\protect\citeauthoryear{Devroye}{1986}]{Devroye86}
\begin{bbook}[author]
\bauthor{\bsnm{Devroye},~\bfnm{Luc}\binits{L.}}
(\byear{1986}).
\btitle{Lecture notes on bucket algorithms}.
\bseries{Progress in Computer Science}.
\bpublisher{Birkh\"{a}user Boston, Inc., Boston, MA}.
\bdoi{10.1007/978-1-4899-3531-1}
\bmrnumber{890327}
\end{bbook}
\endbibitem

\bibitem[\protect\citeauthoryear{Durbin}{1973}]{Durbin73}
\begin{bbook}[author]
\bauthor{\bsnm{Durbin},~\bfnm{J.}\binits{J.}}
(\byear{1973}).
\btitle{Distribution theory for tests based on the sample distribution
  function}.
\bpublisher{Society for Industrial and Applied Mathematics, Philadelphia, PA}.
\bmrnumber{0305507}
\end{bbook}
\endbibitem

\bibitem[\protect\citeauthoryear{Durbin and Knott}{1972}]{Durbin72}
\begin{barticle}[author]
\bauthor{\bsnm{Durbin},~\bfnm{James}\binits{J.}} \AND
  \bauthor{\bsnm{Knott},~\bfnm{M.}\binits{M.}}
(\byear{1972}).
\btitle{Components of {C}ram\'{e}r-von {M}ises statistics. {I}}.
\bjournal{J. Roy. Statist. Soc. Ser. B}
\bvolume{34}
\bpages{290--307}.
\bmrnumber{0365880}
\end{barticle}
\endbibitem

\bibitem[\protect\citeauthoryear{Durio and Nikitin}{2016}]{Durio16}
\begin{barticle}[author]
\bauthor{\bsnm{Durio},~\bfnm{A.}\binits{A.}} \AND
  \bauthor{\bsnm{Nikitin},~\bfnm{Ya.~Yu.}\binits{Y.~Y.}}
(\byear{2016}).
\btitle{Local efficiency of integrated goodness-of-fit tests under skew
  alternatives}.
\bjournal{Statist. Probab. Lett.}
\bvolume{117}
\bpages{136--143}.
\bdoi{10.1016/j.spl.2016.05.016}
\bmrnumber{3519836}
\end{barticle}
\endbibitem

\bibitem[\protect\citeauthoryear{Feuerverger}{2016}]{Feuerverger16}
\begin{barticle}[author]
\bauthor{\bsnm{Feuerverger},~\bfnm{Andrey}\binits{A.}}
(\byear{2016}).
\btitle{On goodness of fit for operational risk}.
\bjournal{Int. Stat. Rev.}
\bvolume{84}
\bpages{434--455}.
\bdoi{10.1111/insr.12112}
\bmrnumber{3580424}
\end{barticle}
\endbibitem

\bibitem[\protect\citeauthoryear{Gao et~al.}{2003}]{Gao03a}
\begin{barticle}[author]
\bauthor{\bsnm{Gao},~\bfnm{Fuchang}\binits{F.}},
  \bauthor{\bsnm{Hannig},~\bfnm{Jan}\binits{J.}},
  \bauthor{\bsnm{Lee},~\bfnm{Tzong-Yow}\binits{T.-Y.}} \AND
  \bauthor{\bsnm{Torcaso},~\bfnm{Fred}\binits{F.}}
(\byear{2003}).
\btitle{Laplace transforms via {H}adamard factorization}.
\bjournal{Electron. J. Probab.}
\bvolume{8}
\bpages{no. 13, 20}.
\bdoi{10.1214/EJP.v8-151}
\bmrnumber{1998764}
\end{barticle}
\endbibitem

\bibitem[\protect\citeauthoryear{Gonz\'alez-Manteiga and
  Crujeiras}{2013}]{Gonzalez13}
\begin{barticle}[author]
\bauthor{\bsnm{Gonz\'alez-Manteiga},~\bfnm{W.}\binits{W.}} \AND
  \bauthor{\bsnm{Crujeiras},~\bfnm{R.~M.}\binits{R.~M.}}
(\byear{2013}).
\btitle{An updated review of goodness-of-fit tests for regression models}.
\bjournal{TEST}
\bvolume{22}
\bpages{361--411}.
\bmrnumber{3093195}
\end{barticle}
\endbibitem

\bibitem[\protect\citeauthoryear{Groeneboom and Shorack}{1981}]{Groeneboom81}
\begin{barticle}[author]
\bauthor{\bsnm{Groeneboom},~\bfnm{Piet}\binits{P.}} \AND
  \bauthor{\bsnm{Shorack},~\bfnm{Galen~R.}\binits{G.~R.}}
(\byear{1981}).
\btitle{Large deviations of goodness of fit statistics and linear combinations
  of order statistics}.
\bjournal{Ann. Probab.}
\bvolume{9}
\bpages{971--987}.
\bmrnumber{632970}
\end{barticle}
\endbibitem

\bibitem[\protect\citeauthoryear{Henze and Nikitin}{2000}]{Henze00}
\begin{barticle}[author]
\bauthor{\bsnm{Henze},~\bfnm{N.}\binits{N.}} \AND
  \bauthor{\bsnm{Nikitin},~\bfnm{Ya.~Yu.}\binits{Y.~Y.}}
(\byear{2000}).
\btitle{A new approach to goodness-of-fit testing based on the integrated
  empirical process}.
\bjournal{J. Nonparametr. Statist.}
\bvolume{12}
\bpages{391--416}.
\bdoi{10.1080/10485250008832815}
\bmrnumber{1760715}
\end{barticle}
\endbibitem

\bibitem[\protect\citeauthoryear{Henze and Nikitin}{2002}]{Henze02}
\begin{barticle}[author]
\bauthor{\bsnm{Henze},~\bfnm{N.}\binits{N.}} \AND
  \bauthor{\bsnm{Nikitin},~\bfnm{Ya.~Yu.}\binits{Y.~Y.}}
(\byear{2002}).
\btitle{Watson-type goodness-of-fit tests based on the integrated empirical
  process}.
\bjournal{Math. Methods Statist.}
\bvolume{11}
\bpages{183--202}.
\bmrnumber{1941315}
\end{barticle}
\endbibitem

\bibitem[\protect\citeauthoryear{Henze and Nikitin}{2003}]{Henze03}
\begin{barticle}[author]
\bauthor{\bsnm{Henze},~\bfnm{Norbert}\binits{N.}} \AND
  \bauthor{\bsnm{Nikitin},~\bfnm{Ya.~Yu.}\binits{Y.~Y.}}
(\byear{2003}).
\btitle{Two-sample tests based on the integrated empirical process}.
\bjournal{Comm. Statist. Theory Methods}
\bvolume{32}
\bpages{1767--1788}.
\bdoi{10.1081/STA-120022708}
\bmrnumber{1998981}
\end{barticle}
\endbibitem

\bibitem[\protect\citeauthoryear{Hirotsu}{1986}]{Hirotsu86}
\begin{barticle}[author]
\bauthor{\bsnm{Hirotsu},~\bfnm{C.}\binits{C.}}
(\byear{1986}).
\btitle{Cumulative chi-squared statistic as a tool for testing goodness of
  fit}.
\bjournal{Biometrika}
\bvolume{73}
\bpages{165--173}.
\bdoi{10.1093/biomet/73.1.165}
\bmrnumber{836444}
\end{barticle}
\endbibitem

\bibitem[\protect\citeauthoryear{Hirotsu}{2017}]{Hirotsu17}
\begin{bbook}[author]
\bauthor{\bsnm{Hirotsu},~\bfnm{Chihiro}\binits{C.}}
(\byear{2017}).
\btitle{Advanced analysis of variance}.
\bpublisher{John Wiley \& Sons, Inc., Hoboken, NJ}.
\bmrnumber{3729256}
\end{bbook}
\endbibitem

\bibitem[\protect\citeauthoryear{Knuth}{1998}]{Knuth98}
\begin{bbook}[author]
\bauthor{\bsnm{Knuth},~\bfnm{Donald~E.}\binits{D.~E.}}
(\byear{1998}).
\btitle{The art of computer programming. {V}ol. 3. Sorting and searching},
\bedition{Second} ed.
\bpublisher{Addison-Wesley, Reading, MA}.
\bmrnumber{3077154}
\end{bbook}
\endbibitem

\bibitem[\protect\citeauthoryear{Kuriki and Hwang}{2013}]{Kuriki13}
\begin{barticle}[author]
\bauthor{\bsnm{Kuriki},~\bfnm{Satoshi}\binits{S.}} \AND
  \bauthor{\bsnm{Hwang},~\bfnm{Hsien-Kuei}\binits{H.-K.}}
(\byear{2013}).
\btitle{Anderson-{D}arling type goodness-of-fit statistic based on a multifold
  integrated empirical distribution function}.
\bjournal{Proceedings of the 59th ISI World Statistics Congress, Hong Kong,
  25--30 August}
\bpages{3773--3778}.
\end{barticle}
\endbibitem

\bibitem[\protect\citeauthoryear{Liu}{2023}]{Liu23}
\begin{barticle}[author]
\bauthor{\bsnm{Liu},~\bfnm{Chuanhai}\binits{C.}}
(\byear{2023}).
\btitle{Reweighted and circularised {A}nderson-{D}arling tests of
  goodness-of-fit}.
\bjournal{Journal of Nonparametric Statistics}
\bvolume{35}
\bpages{869--904}.
\bdoi{10.1080/10485252.2023.2213782}
\bmrnumber{4664880}
\end{barticle}
\endbibitem

\bibitem[\protect\citeauthoryear{Lockhart and Stephens}{1998}]{Lockhart98}
\begin{bincollection}[author]
\bauthor{\bsnm{Lockhart},~\bfnm{R.~A.}\binits{R.~A.}} \AND
  \bauthor{\bsnm{Stephens},~\bfnm{M.~A.}\binits{M.~A.}}
(\byear{1998}).
\btitle{The probability plot: tests of fit based on the correlation
  coefficient}.
In \bbooktitle{Order statistics: applications}.
\bseries{Handbook of Statist.}
\bvolume{17}
\bpages{453--473}.
\bpublisher{North-Holland, Amsterdam}.
\bdoi{10.1016/S0169-7161(98)17018-9}
\bmrnumber{1672299}
\end{bincollection}
\endbibitem

\bibitem[\protect\citeauthoryear{Luce{\~n}o}{2006}]{Luceno06}
\begin{barticle}[author]
\bauthor{\bsnm{Luce{\~n}o},~\bfnm{Alberto}\binits{A.}}
(\byear{2006}).
\btitle{Fitting the generalized {P}areto distribution to data using maximum
  goodness-of-fit estimators}.
\bjournal{Comput. Statist. Data Anal.}
\bvolume{51}
\bpages{904--917}.
\bmrnumber{2297496}
\end{barticle}
\endbibitem

\bibitem[\protect\citeauthoryear{Ma, Kitani and Murakami}{2022}]{Ma22}
\begin{barticle}[author]
\bauthor{\bsnm{Ma},~\bfnm{Yuyan}\binits{Y.}},
  \bauthor{\bsnm{Kitani},~\bfnm{Masato}\binits{M.}} \AND
  \bauthor{\bsnm{Murakami},~\bfnm{Hidetoshi}\binits{H.}}
(\byear{2022}).
\btitle{On modified {A}nderson-{D}arling test statistics with asymptotic
  properties}.
\bjournal{Comm. Statist. Theory Methods}
\bvolume{53}
\bpages{1420--1439}.
\bdoi{10.1080/03610926.2022.2101121}
\end{barticle}
\endbibitem

\bibitem[\protect\citeauthoryear{MacNeill}{1978a}]{MacNeill78a}
\begin{barticle}[author]
\bauthor{\bsnm{MacNeill},~\bfnm{Ian~B.}\binits{I.~B.}}
(\byear{1978}a).
\btitle{Limit processes for sequences of partial sums of regression residuals}.
\bjournal{Ann. Probab.}
\bvolume{6}
\bpages{695--698}.
\bmrnumber{0494708}
\end{barticle}
\endbibitem

\bibitem[\protect\citeauthoryear{MacNeill}{1978b}]{MacNeill78}
\begin{barticle}[author]
\bauthor{\bsnm{MacNeill},~\bfnm{Ian~B.}\binits{I.~B.}}
(\byear{1978}b).
\btitle{Properties of sequences of partial sums of polynomial regression
  residuals with applications to tests for change of regression at unknown
  times}.
\bjournal{Ann. Statist.}
\bvolume{6}
\bpages{422--433}.
\bmrnumber{0474645}
\end{barticle}
\endbibitem

\bibitem[\protect\citeauthoryear{Mansuy}{2005}]{Mansuy05}
\begin{barticle}[author]
\bauthor{\bsnm{Mansuy},~\bfnm{Roger}\binits{R.}}
(\byear{2005}).
\btitle{An interpretation and some generalizations of the {A}nderson-{D}arling
  statistics in terms of squared {B}essel bridges}.
\bjournal{Statist. Probab. Lett.}
\bvolume{72}
\bpages{171--177}.
\bdoi{10.1016/j.spl.2005.01.001}
\bmrnumber{2137123}
\end{barticle}
\endbibitem

\bibitem[\protect\citeauthoryear{Martynov}{2015}]{Martynov15}
\begin{bincollection}[author]
\bauthor{\bsnm{Martynov},~\bfnm{Gennady}\binits{G.}}
(\byear{2015}).
\btitle{A {C}ram\'{e}r-von {M}ises test for {G}aussian processes}.
In \bbooktitle{Mathematical statistics and limit theorems}
\bpages{209--229}.
\bpublisher{Springer, Cham}.
\bmrnumber{3380738}
\end{bincollection}
\endbibitem

\bibitem[\protect\citeauthoryear{Matsui and Takemura}{2008}]{Matsui08}
\begin{barticle}[author]
\bauthor{\bsnm{Matsui},~\bfnm{Muneya}\binits{M.}} \AND
  \bauthor{\bsnm{Takemura},~\bfnm{Akimichi}\binits{A.}}
(\byear{2008}).
\btitle{Goodness-of-fit tests for symmetric stable distributions---empirical
  characteristic function approach}.
\bjournal{TEST}
\bvolume{17}
\bpages{546--566}.
\bdoi{10.1007/s11749-007-0045-y}
\bmrnumber{2470098}
\end{barticle}
\endbibitem

\bibitem[\protect\citeauthoryear{Medovikov}{2016}]{Medovikov16}
\begin{barticle}[author]
\bauthor{\bsnm{Medovikov},~\bfnm{Ivan}\binits{I.}}
(\byear{2016}).
\btitle{Non-parametric weighted tests for independence based on empirical
  copula process}.
\bjournal{J. Stat. Comput. Simul.}
\bvolume{86}
\bpages{105--121}.
\bdoi{10.1080/00949655.2014.995657}
\bmrnumber{3403626}
\end{barticle}
\endbibitem

\bibitem[\protect\citeauthoryear{Nikitin}{1995}]{Nikitin95}
\begin{bbook}[author]
\bauthor{\bsnm{Nikitin},~\bfnm{Yakov}\binits{Y.}}
(\byear{1995}).
\btitle{Asymptotic efficiency of nonparametric tests}.
\bpublisher{Cambridge University Press, Cambridge}.
\bdoi{10.1017/CBO9780511530081}
\bmrnumber{1335235}
\end{bbook}
\endbibitem

\bibitem[\protect\citeauthoryear{Nikitin and Pusev}{2013}]{Nikitin13}
\begin{barticle}[author]
\bauthor{\bsnm{Nikitin},~\bfnm{Ya.~Yu.}\binits{Y.~Y.}} \AND
  \bauthor{\bsnm{Pusev},~\bfnm{R.~S.}\binits{R.~S.}}
(\byear{2013}).
\btitle{Exact small deviation asymptotics for some {B}rownian functionals}.
\bjournal{Theory of Probability \& Its Applications}
\bvolume{57}
\bpages{60-81}.
\bdoi{10.1137/S0040585X97985790}
\bmrnumber{3201638}
\end{barticle}
\endbibitem

\bibitem[\protect\citeauthoryear{OEIS}{2019}]{OEIS}
\begin{bmisc}[author]
\bauthor{\bsnm{OEIS}}
(\byear{2019}).
\btitle{The on-line encyclopedia of integer sequences, \url{https://oeis.org}}.
\end{bmisc}
\endbibitem

\bibitem[\protect\citeauthoryear{Pettitt}{1976}]{Pettitt76}
\begin{barticle}[author]
\bauthor{\bsnm{Pettitt},~\bfnm{A.~N.}\binits{A.~N.}}
(\byear{1976}).
\btitle{A two-sample {A}nderson-{D}arling rank statistic}.
\bjournal{Biometrika}
\bvolume{63}
\bpages{161--168}.
\bdoi{10.1093/biomet/63.1.161}
\bmrnumber{0413359}
\end{barticle}
\endbibitem

\bibitem[\protect\citeauthoryear{Pycke}{2003}]{Pycke03}
\begin{barticle}[author]
\bauthor{\bsnm{Pycke},~\bfnm{J.~R.}\binits{J.~R.}}
(\byear{2003}).
\btitle{Multivariate extensions of the {A}nderson-{D}arling process}.
\bjournal{Statist. Probab. Lett.}
\bvolume{63}
\bpages{387--399}.
\bdoi{10.1016/S0167-7152(03)00111-1}
\bmrnumber{1996187}
\end{barticle}
\endbibitem

\bibitem[\protect\citeauthoryear{Pycke}{2021a}]{Pycke21a}
\begin{barticle}[author]
\bauthor{\bsnm{Pycke},~\bfnm{J.~R.}\binits{J.~R.}}
(\byear{2021}a).
\btitle{On three families of {K}arhunen--{L}o{\`e}ve expansions associated with
  classical orthogonal polynomials}.
\bjournal{Results in Mathematics}
\bvolume{76}
\bpages{148}.
\bdoi{10.1007/s00025-021-01454-x}
\bmrnumber{4279486}
\end{barticle}
\endbibitem

\bibitem[\protect\citeauthoryear{Pycke}{2021b}]{Pycke21}
\begin{barticle}[author]
\bauthor{\bsnm{Pycke},~\bfnm{J.~R.}\binits{J.~R.}}
(\byear{2021}b).
\btitle{A new family of omega-square-type statistics with {B}ahadur local
  optimality for the location family of generalized logistic distributions}.
\bjournal{Statistics \& Probability Letters}
\bvolume{170}
\bpages{108999}.
\bdoi{https://doi.org/10.1016/j.spl.2020.108999}
\bmrnumber{4184365}
\end{barticle}
\endbibitem

\bibitem[\protect\citeauthoryear{Pycke}{2023}]{Pycke23}
\begin{barticle}[author]
\bauthor{\bsnm{Pycke},~\bfnm{J.~R.}\binits{J.~R.}}
(\byear{2023}).
\btitle{On a {K}arhunen-{L}o\`eve expansion based on {K}rawtchouk polynomials
  with application to {B}ahadur optimality for the binomial location family}.
\bjournal{Filomat}
\bvolume{37}
\bpages{4795--4807}.
\bdoi{10.2298/fil2014795v}
\bmrnumber{4576616}
\end{barticle}
\endbibitem

\bibitem[\protect\citeauthoryear{Rodr\'{\i}guez and
  Viollaz}{1995}]{Rodriguez95}
\begin{barticle}[author]
\bauthor{\bsnm{Rodr\'{\i}guez},~\bfnm{Juan~C.}\binits{J.~C.}} \AND
  \bauthor{\bsnm{Viollaz},~\bfnm{Aldo~J.}\binits{A.~J.}}
(\byear{1995}).
\btitle{A {C}ram\'{e}r-von {M}ises type goodness of fit test with asymmetric
  weight function}.
\bjournal{Comm. Statist. Theory Methods}
\bvolume{24}
\bpages{1095--1120}.
\bdoi{10.1080/03610929508831542}
\bmrnumber{1323267}
\end{barticle}
\endbibitem

\bibitem[\protect\citeauthoryear{Rodr\'{\i}guez and
  Viollaz}{1999}]{Rodriguez99}
\begin{barticle}[author]
\bauthor{\bsnm{Rodr\'{\i}guez},~\bfnm{Juan~C.}\binits{J.~C.}} \AND
  \bauthor{\bsnm{Viollaz},~\bfnm{Aldo~J.}\binits{A.~J.}}
(\byear{1999}).
\btitle{A weighted {C}ram\'er-von {M}ises statistic derived from a
  decomposition of the {A}nderson-{D}arling statistic}.
\bjournal{Comm. Statist. Theory Methods}
\bvolume{28}
\bpages{23330--2346}.
\end{barticle}
\endbibitem

\bibitem[\protect\citeauthoryear{Scott}{1999}]{Scott99}
\begin{barticle}[author]
\bauthor{\bsnm{Scott},~\bfnm{W.~F.}\binits{W.~F.}}
(\byear{1999}).
\btitle{A weighted {C}ram\'{e}r-von {M}ises statistic, with some applications
  to clinical trials}.
\bjournal{Comm. Statist. Theory Methods}
\bvolume{28}
\bpages{3001--3008}.
\bdoi{10.1080/03610929908832461}
\bmrnumber{1729011}
\end{barticle}
\endbibitem

\bibitem[\protect\citeauthoryear{Shapiro and Wilk}{1965}]{Shapiro65}
\begin{barticle}[author]
\bauthor{\bsnm{Shapiro},~\bfnm{S.~S.}\binits{S.~S.}} \AND
  \bauthor{\bsnm{Wilk},~\bfnm{M.~B.}\binits{M.~B.}}
(\byear{1965}).
\btitle{An analysis of variance test for normality (complete samples)}.
\bjournal{Biometrika}
\bvolume{52}
\bpages{591--611}.
\bdoi{10.1093/biomet/52.3-4.591}
\bmrnumber{0205384}
\end{barticle}
\endbibitem

\bibitem[\protect\citeauthoryear{Shi, Wang and Reid}{2022}]{Shi22}
\begin{barticle}[author]
\bauthor{\bsnm{Shi},~\bfnm{Xiaoping}\binits{X.}},
  \bauthor{\bsnm{Wang},~\bfnm{Xiang-Sheng}\binits{X.-S.}} \AND
  \bauthor{\bsnm{Reid},~\bfnm{Nancy}\binits{N.}}
(\byear{2022}).
\btitle{A New Class of Weighted {CUSUM} Statistics}.
\bjournal{Entropy}
\bvolume{24}
\bpages{1652}.
\bmrnumber{4534844}
\end{barticle}
\endbibitem

\bibitem[\protect\citeauthoryear{Shorack and Wellner}{1986}]{Shorack86}
\begin{bbook}[author]
\bauthor{\bsnm{Shorack},~\bfnm{Galen~R.}\binits{G.~R.}} \AND
  \bauthor{\bsnm{Wellner},~\bfnm{Jon~A.}\binits{J.~A.}}
(\byear{1986}).
\btitle{Empirical processes with applications to statistics}.
\bpublisher{John Wiley \& Sons, Inc., New York}.
\bmrnumber{838963}
\end{bbook}
\endbibitem

\bibitem[\protect\citeauthoryear{Sinclair, Spurr and Ahmad}{1990}]{Sinclair90}
\begin{barticle}[author]
\bauthor{\bsnm{Sinclair},~\bfnm{C.~D.}\binits{C.~D.}},
  \bauthor{\bsnm{Spurr},~\bfnm{B.~D.}\binits{B.~D.}} \AND
  \bauthor{\bsnm{Ahmad},~\bfnm{M.~I.}\binits{M.~I.}}
(\byear{1990}).
\btitle{Modified {A}nderson {D}arling test}.
\bjournal{Comm. Statist. Theory Methods}
\bvolume{19}
\bpages{3677--3686}.
\end{barticle}
\endbibitem

\bibitem[\protect\citeauthoryear{Slepian}{1958}]{Slepian58}
\begin{barticle}[author]
\bauthor{\bsnm{Slepian},~\bfnm{D.}\binits{D.}}
(\byear{1958}).
\btitle{Fluctuations of random noise power}.
\bjournal{Bell System Tech. J.}
\bvolume{37}
\bpages{163--184}.
\bdoi{10.1002/j.1538-7305.1958.tb03873.x}
\bmrnumber{0092256}
\end{barticle}
\endbibitem

\bibitem[\protect\citeauthoryear{Smirnov}{1937}]{Smirnov37}
\begin{barticle}[author]
\bauthor{\bsnm{Smirnov},~\bfnm{Nikolai~Vasil'evich}\binits{N.~V.}}
(\byear{1937}).
\btitle{On the distribution of the $\omega^2$-criterion of von {M}ises}.
\bjournal{Math. Sbornik}
\bvolume{2}
\bpages{973--993}.
\end{barticle}
\endbibitem

\bibitem[\protect\citeauthoryear{Tanaka}{2008}]{Tanaka08}
\begin{barticle}[author]
\bauthor{\bsnm{Tanaka},~\bfnm{Katsuto}\binits{K.}}
(\byear{2008}).
\btitle{On the distribution of quadratic functionals of the ordinary and
  fractional {B}rownian motions}.
\bjournal{J. Statist. Plann. Inference}
\bvolume{138}
\bpages{3525--3537}.
\bdoi{10.1016/j.jspi.2005.12.015}
\bmrnumber{2450093}
\end{barticle}
\endbibitem

\bibitem[\protect\citeauthoryear{Tanaka}{2017}]{Tanaka17}
\begin{bbook}[author]
\bauthor{\bsnm{Tanaka},~\bfnm{Katsuto}\binits{K.}}
(\byear{2017}).
\btitle{Time series analysis},
\bedition{Second} ed.
\bpublisher{John Wiley \& Sons, Inc., Hoboken, NJ}.
\bdoi{10.1002/9781119132165}
\bmrnumber{3676578}
\end{bbook}
\endbibitem

\bibitem[\protect\citeauthoryear{Tsukuda and Nishiyama}{2014}]{Tsukuda14}
\begin{barticle}[author]
\bauthor{\bsnm{Tsukuda},~\bfnm{Koji}\binits{K.}} \AND
  \bauthor{\bsnm{Nishiyama},~\bfnm{Yoichi}\binits{Y.}}
(\byear{2014}).
\btitle{On {$L^2$} space approach to change point problems}.
\bjournal{J. Statist. Plann. Inference}
\bvolume{149}
\bpages{46--59}.
\bdoi{10.1016/j.jspi.2014.02.007}
\bmrnumber{3199893}
\end{barticle}
\endbibitem

\bibitem[\protect\citeauthoryear{van~der Vaart and
  Wellner}{1996}]{vanderVaart96}
\begin{bbook}[author]
\bauthor{\bparticle{van~der} \bsnm{Vaart},~\bfnm{Aad~W.}\binits{A.~W.}} \AND
  \bauthor{\bsnm{Wellner},~\bfnm{Jon~A.}\binits{J.~A.}}
(\byear{1996}).
\btitle{Weak convergence and empirical processes}.
\bseries{Springer Series in Statistics}.
\bpublisher{Springer-Verlag, New York}.
\bdoi{10.1007/978-1-4757-2545-2}
\bmrnumber{1385671}
\end{bbook}
\endbibitem

\bibitem[\protect\citeauthoryear{Viollaz and Rodr\'{\i}guez}{1996}]{Viollaz96}
\begin{barticle}[author]
\bauthor{\bsnm{Viollaz},~\bfnm{Aldo~J.}\binits{A.~J.}} \AND
  \bauthor{\bsnm{Rodr\'{\i}guez},~\bfnm{Juan~C.}\binits{J.~C.}}
(\byear{1996}).
\btitle{A {C}ram\'{e}r-von {M}ises type goodness-of-fit test with asymmetric
  weight function. {T}he {G}aussian and exponential cases}.
\bjournal{Comm. Statist. Theory Methods}
\bvolume{25}
\bpages{235--256}.
\bdoi{10.1080/03610929608831691}
\bmrnumber{1378958}
\end{barticle}
\endbibitem

\bibitem[\protect\citeauthoryear{von Mises}{1931}]{vonMises31}
\begin{bbook}[author]
\bauthor{\bparticle{von} \bsnm{Mises},~\bfnm{Richard}\binits{R.}}
(\byear{1931}).
\btitle{Wahrscheinlichkeitsrechnung und ihre Anwendung in der Statistik und
  theoretischen Physik}.
\bpublisher{Leipzig and Wien, Franz Deuticke}.
\end{bbook}
\endbibitem

\bibitem[\protect\citeauthoryear{Watson}{1961}]{Watson61}
\begin{barticle}[author]
\bauthor{\bsnm{Watson},~\bfnm{George~S}\binits{G.~S.}}
(\byear{1961}).
\btitle{Goodness-of-fit tests on a circle}.
\bjournal{Biometrika}
\bvolume{48}
\bpages{109--114}.
\bmrnumber{131930}
\end{barticle}
\endbibitem

\end{thebibliography}
\end{document}